\documentclass[12pt,final,nofigureslist,notableslist]{ucdavisthesis}
\pdfoutput=1

\usepackage[us,nodayofweek,12hr]{datetime}
\usepackage{graphicx}
\usepackage{amssymb,amsmath, amsthm, tikz, tkz-euclide, adjustbox, mathtools,lmodern}
\usepackage{pgfplots}
\pgfplotsset{compat=newest} 
\usetikzlibrary{matrix}

\usepackage[psdextra, bookmarks=false]{hyperref}
\hypersetup{
    colorlinks=true,
    linkcolor=blue,
}
\usepackage{cleveref}
\usepackage{thm-restate}
\usepackage{bookmark}

\usepackage{color}
\usepackage{mathrsfs}

\usepackage[style=alphabetic,
        backend=biber,
        isbn=true,
        doi=false,
        url=true,
        hyperref=true, 
        firstinits=true,
        maxbibnames=99,
        maxalphanames=99,
        maxnames=99,
        block=none]{biblatex}
\bibliography{thesis.bib}

\allowdisplaybreaks

\tikzset{midto/.style={postaction={decorate,
    decoration={markings,mark=at position .5 with
    {\draw (-.035,-.07) -- (.035,0) -- (-.035,.07);}}}}}
\tikzset{midfrom/.style={postaction={decorate,
    decoration={markings,mark=at position .5 with
    {\draw (.035,-.07) -- (-.035,0) -- (.035,.07);}}}}}
\tikzset{web/.style={black,semithick}}

\newcommand{\surj}{\twoheadrightarrow}

\newcommand{\tbf}{\textbf} 
\newcommand{\mf}{\mathfrak} 

\newcommand{\ms}{\mathscr}

\newcommand{\mbf}{\mathbf}
\newcommand{\R}{\mathbb{R}}

\newcommand{\Q}{\mathbb{Q}}
\newcommand{\C}{\mathbb{C}}

\newcommand{\Z}{\mathbb{Z}}

\newcommand{\N}{\mathbb{N}}

\newcommand{\K}{\mathbb{K}}

\newcommand{\ep}{\varepsilon}
\newcommand{\la}{\lambda}

\newcommand{\ph}{\varphi}
\newcommand{\ga}{\gamma} 
\newcommand{\Ga}{\Gamma}
\newcommand{\al}{\alpha}

\newcommand{\om}{\omega}

\newcommand{\fg}{\mathfrak{g}}

\newcommand{\sV}{\mathscr{V}}
\newcommand{\sT}{\mathscr{T}}
\newcommand{\natiso}{\xrightarrow{\sim}}

\newcommand{\inj}{\hookrightarrow}
\newcommand{\Ug}{U(\mathfrak{g})}
\newcommand{\Uqg}{U_q(\mathfrak{g})} 
\newcommand{\rb}{\rbrace} 
\newcommand{\lb}{\lbrace}

\newcommand{\lan}{\langle} 
\newcommand{\ran}{\rangle}
\newcommand{\ol}{\overline} 

\newcommand{\wh}{\widehat}

\newcommand{\op}{\oplus} 
\newcommand{\ot}{\otimes}

\newcommand{\wt}{\widetilde}
\newcommand{\bu}{\bullet}

\newcommand{\sbs}{\subseteq}
\newcommand{\sps}{\supseteq}

\newcommand{\tit}{\textit}

\newcommand{\xto}{\xrightarrow}
\newcommand{\sqbinom}{\genfrac{[}{]}{0pt}{}}

\DeclareMathOperator{\Hom}{Hom}

\DeclareMathOperator{\rk}{rk}

\DeclareMathOperator{\Ob}{Ob}
\DeclareMathOperator{\Mor}{Mor}

\DeclareMathOperator{\Span}{Span}
\DeclareMathOperator{\gr}{gr}
\DeclareMathOperator{\rep}{rep}

\newcommand\eqA{\stackrel{\mathclap{\mbox{A}}}{=}}
\newcommand\eqB{\stackrel{\mathclap{\mbox{B}}}{=}}

\newtheorem{thm}{Theorem}
\numberwithin{thm}{chapter}
\newtheorem{prop}[thm]{Proposition}
\newtheorem{fact}[thm]{Fact}
\newtheorem{lemma}[thm]{Lemma}
\newtheorem{defn}[thm]{Definition}
\newtheorem{coro}[thm]{Corollary}
\theoremstyle{remark}
\newtheorem*{remark}{Remark}

\numberwithin{equation}{section}

\title          {Spiders and Generalized Confluence}

\author         {Colin Scott Hagemeyer}

\authordegrees  {\quad \quad \quad}

\officialmajor  {Mathematics}

\graduateprogram{Graduate Program Name}

\degreeyear     {2018}

\degreemonth    {September}

\committee{Professor Greg Kuperberg}{Professor Eric Babson}{Professor Monica Vazirani}{}{}

\nocopyright

\dedication{In memory of my father Alan Hagemeyer for giving me the intellectual curiosity and support that allowed me to get to where I am now.}

\abstract{
\linespread{1}\selectfont
Given a semisimple Lie algebra $\fg$, we can represent invariants of tensor products of fundamental representations of the quantum enveloping algebra $U_q(\fg)$ using particular directed graphs called webs. In particular webs are trivalent graphs (with leaves) whose edges are labeled by fundamental representations. Picking generating morphisms and relators we can construct a presentation of the representation category. We examine the properties of this presentation in the case of rank $3$ spiders and certain higher rank non-simple spiders. In particular, we prove a PBW-type theorem in the case of $\mf{sl}_4$, $(\mf{sl}_2)^n$, and $\mf{sl}_2 \op \mf{sl}_3$ and also give counterexamples showing that no such result is true in the case of $(\mf{sl}_2)^2 \op \mf{sl}_3$ and $\mf{sl}_3 \op \mf{sl}_3$. Nevertheless we rephrase the PBW-type theorem as a degeneration of a particular spectral sequence, and prove that this spectral sequence converges on the second page for $(\mf{sl}_2)^n \op \mf{sl}_3$, giving generalized and weaker form of confluence.

We then apply the above results to the geometry of the Euclidean building in the case of $\mf{sl}_4$ and $(\mf{sl}_2)^n$. In particular, we prove an upper triangularity result with respect to the geometric Satake basis for $\mf{sl}_4$, improving the results of Fontaine in \cite{fontaine:generating}. Finally we give a geometric interpretation of webs as minimal combinatorial disks in the Euclidean building, reinterpreting many of the combinatorial results of paper in geometric terms.}

\acknowledgments{
\linespread{1}\selectfont
First, I'd like to thank my advisor Greg Kuperberg for teaching me about spiders, providing the problems of spectral sequence convergence and higher dimensional pockets that led me to the work in this dissertation, and supporting me through my time in graduate school.

I'd also like to thank my fellow graduate students at UC Davis, especially Jianping Pan, Wencin Poh, Eric Samperton, and Li Yang, for the lively math discussions which taught me a lot about a diverse subset of mathematics.

Finally, I'd like to thank my family for supporting me emotionally through the very long process of going through graduate school.

This research was partly supported by NSF grants CCF-1319245 and CCF-1716990.}

\begin{document}

\makeintropages 

\clearpage
\pagenumbering{arabic}
\chapter{Introduction}
\linespread{1}\selectfont
In this paper, the quantum group $U_q(\mathfrak{g})$ is a deformation of the universal enveloping algebra $U(\mathfrak{g})$ of some semisimple Lie algebra $\mathfrak{g}$ (cf. \cite{kassel:quantum}). For generic choices of the parameter $q$, the simple representations are in natural bijection with the simple representations of $\mathfrak{g}$, however this deformation causes the tensor category structure of the representation category to no longer be symmetric which gives us a richer theory and provides applications such as knot invariants. In \cite{morrison:diagrammatic} and \cite{ckm:howe} a combinatorial model in type $A_n$ (ie for $\mathfrak{g}=\mathfrak{sl}_{n+1}$), was developed using generators and relations. This extended the work of \cite{kuperberg:rank} and \cite{kim:graphical} in type $A_2$ and $A_3$ respectively. Simple tensors and compositions of generators are called webs, and can be represented as trivalent graphs. Each of these models is called a spider. Here we will analyze the simple spider corresponding to $A_3$ along with the spiders of the semisimple Lie groups of the form $A_1^n \times A_2^m$ (we'll mainly use product notation from now on because the resulting representations are tensor products of representations of the factors). Our main goal for each spider will develop and prove a generalization of confluence (see the end of section \ref{GeneralizedDiamond} for a definition of confluence), and then use this to find algebraic and geometric applications.

\section{Generalized Confluence}

It was shown in \cite{kuperberg:rank} that the rank 2 spiders were confluent with respect to the filtration given by the number of vertices of each web. In the rank 3 case, this no longer makes sense in the usual format because there are relators which have two leading terms with the same number and type of vertices. We could attempt to refine the filtration, but this would necessarily cause it to stop being rotationally symmetric. Instead we generalize the notion of confluence. First we note that a filtered presentation of vector spaces is just a filtered complex with two columns (one for the generators and one for the relators), and so induces a spectral sequence, whose convergence is closely related to confluence.

\begin{restatable}{prop}{Confluence}
\label{intro:confluence}
A presentation of a filtered vector space whose relators have only a single leading term is confluent if and only if the associated spectral sequence converges at the first page $E^1$.
\end{restatable}

If we instead compare to filtered algebras with more general relators, convergence of this spectral sequence amounts to a PBW-type theorem where the associated graded is isomorphic to the algebra with the same presentation except that we truncate away the lower-order terms of each relator.

Taking this as motivation we can generalize this notion to $E^k$ convergence for those presentations whose associated spectral sequence converges at the $k$th page. Intuitively, this property tells us that in the process of reducing an element $w$ to a minimal representative, we don't need to use an web which has $k$ more vertices than $w$. In this paper we will be looking at presentations of additive tensor categories over a field, and whose filtration will be given by the number of generators in each word. We will say that such a presentation is $E^k$ convergent if every such $Hom$ space has a spectral sequence which converges on the $k$th page as a vector space presentation.
\begin{remark}
For other spiders, it may be important to use more complicated filtrations where different generating morphisms are given different weight. Nevertheless, in the case of $A_3$ and $A_1^n$ there is only one generating morphism up to duality, so we won't lose anything by restricting to this filtration. For $A_1^m \times A_2^n$ where $n,m>0$ there are two types of vertices (trivalent and tetravalent) up to outer automorphisms, but it doesn't appear we will change anything significant by changing the filtration short of completely ignoring one type of generator in the filtration
\end{remark}

The non-simple semisimple cases are some of the simplest examples since the horizontal relators have no lower order terms, and the sizes of the invariant spaces are easy to characterize in terms of the simple factors. This gives us our first method of proving $E^1$ convergence which is more classical. We show that the dimensions in the $E^1$ page are the same as in the actual representation category, and therefore get our result for $A_1^n$ (which I can only imagine has been proven before in some other form due to its simplicity and similarity to the Coxeter group properties of the symmetric group):

\begin{restatable}{prop}{A1nConv}
\label{intro:A1nConv}
The $A_1^n$ spider with the natural product presentation has an $E^1$ convergent spectral sequence. 
\end{restatable}

We also prove a similar result for $A_1 \times A_2$ which follows with some extra work from the cut-path results in \cite{kuperberg:rank}.

\begin{prop}
The $A_1 \times A_2$ spider with the natural product presentation has an $E^1$ convergent spectral sequence.
\end{prop}

These two spiders have many nice properties that allow us to prove $E^1$ convergence more easily, but in the simple rank $3$ cases the spiders are more complicated. Nevertheless, they have some similarities which should make the amenable to a unified approach. Therefore, we prove a criterion for $E^1$ convergence which can be used in type $A_3$, and conjecturally for $B_3$ as well.

It was shown in \cite{ckm:howe} that it is possible to generate all relations in the $A_3$ spider without using the hexagon relators introduced in \cite{kim:graphical} and called Kekule relators in \cite{morrison:diagrammatic}, however this requires increasing the number of vertices of certain webs before we can reduce them. Since this would make $E^1$ convergence impossible, we will instead work with the larger set of relators given in \cite{kim:graphical}. Using this as our presentation, we construct a global criterion for when a web is reducible, and use it to obtain the following result using a generalized form of the diamond lemma.

\begin{thm}
\label{intro:E1}
The $\mathfrak{sl}(4)$ spider with the Scott-Morrison presentation has an $E^1$ convergent spectral sequence. In particular, two webs are equal up to lower order terms if and only if they are equal in the associated graded spider (which has the same presentation as the original spider, but where we omit the lower order terms of each relator)
\end{thm}

Although this is only a single case, it is actually quite flexible in that it simultaneously generalizes various combinatorial models such as plane partitions and square ice (after some superficial modifications).

Unfortunately, it may not be reasonable to expect $E^1$ convergence to hold for higher rank simple spiders. In particular, we show the following in Chapter $6$ for semisimple spiders 
\begin{restatable}{thm}{Semisimple}
\label{intro:Semisimple}
The $A_1^{n}\times A_2^{m}$ spider with the natural product presentation is $E^1$ convergent if and only if $m=0$ and/or $\rk \fg \leq 3$
\end{restatable}

However this also shows the utility of the more general notion of $E^k$ convergence because we still get the following weaker result: 

\begin{restatable}{thm}{A2A1}
\label{intro:A2A1}
The $A_1^n \times A_2$ spider with the natural product presentation has an $E^2$ convergent spectral sequence
\end{restatable}

\section{The Euclidean Building and Geometric Satake}
Using $E^1$ convergence results we get some geometric applications. For each group $G$ there is a geometric object called the affine Grassmannian: $Gr(G) = G(\C((x)))/G(\C[[x]])$. This object is related to the representation theory of the Langlands dual group $G^L$ via the Satake correspondence (developed by Lusztig, Ginzburg, and Mirkovi\'c-
Vilonen). In particular, we will use the following corollary (lifted directly from \cite{fkk:buildings}): for every sequence of dominants weights $\vec{\la}=(\la_1,...,\la_n)$ there is an associated variety $F(\vec{\la})$ called the Satake fiber (which is a fiber of a particular map to the affine Grassmannian which we'll explain in chapter 5) and is related to the representation category as follows:

\begin{thm}
Every invariant space in $\rep(G)$ is canonically isomorphic to the
top homology of the corresponding geometric Satake fiber with complex coefficients:

$\Phi : Inv(V(\la_1) \ot ... \ot V(\la_n)) \cong H^{top}(F(\vec{\la}),\C)$.

Each top-dimensional component $Z \sbs F(\vec{\la})$ thus yields a vector $[Z] \in Inv(V(\la_1) \ot ... \ot V(\la_n))$.
These vectors form a basis, the Satake basis.
\end{thm}

The affine Grassmannian then naturally embeds into a (very infinite) simplicial complex called the Euclidean building $X$. Following \cite{fkk:buildings}, the dual graph of each web $w$ naturally extends to a $2$D simplicial complex (with a metric). We will call this the dual web $D(w)$. Given a combinatorial local isometry from the dual web to the Euclidean building, there is a natural corresponding point in $F(\vec{\la})$, and we therefore get a map $\pi: \Hom_{loc. iso.}(D(w),X) \to F(\vec{\la})$.

Since the edges of duals webs are colored by fundamental webs, we can define weight-valued lengths on paths. Since weights naturally have a poset structure, we can then define a geodesic as a path with minimal length. However, in general you may have two geodesics connected the same two points, but the weight-lengths are incomparable. If this does not happen, then we say that the dual web has coherent geodesics (between those points). In \cite{fontaine:generating}, the author defines certain so-called coherent webs whose dual webs have coherent geodesics originating from fixed boundary point, and fulfill two other axioms (although we will prove that the 3rd axiom is redundant). See \cref{CoherenceDef} for the full definition of coherence. To each of these is associated an LS path by taking maximal geodesics connecting the fixed boundary point to each of the other boundary points. Conversely, we will argue that there exist minimal-vertex coherent webs for each LS path, and use the $E^1$ convergence to prove the following theorem relating minimal vertex webs to the geometric Satake basis:

\begin{restatable}{thm}{Satake}
\label{intro:Satake}
Let $\lb w_i \rb_{i=1}^n$ by a maximal set of minimal-vertex $A_3$ webs with fixed boundary which are distinct in the associated graded spider. $\lb w_i \rb_{i=1}^n$ form a basis for the spider category (and hence the invariant space of the corresponding tensor product representation by \cite{ckm:howe}). Moreover this basis is upper unitriangular with respect to the Satake basis for all choices of base points simultaneously.
\end{restatable}

This is the best we can expect since we know that web bases don't generally match the geometric Satake basis as shown in \cite{kk:dual}.

We also look at $A_1^n$ geometrically. Fix a web $w$, and look at the set of all webs $\lb w_i \rb_{i=1}^n$ which are equal $w$ as invariants. The set of dual webs $\lb D(w_i)\rb_{i=1}^n$ naturally correspond to $2$D subcomplexes of a simplicial complex $P$ with a particular boundary (where $P$ is contained in the Euclidean building). We can prove that $P$ is $CAT(0)$ (which is a generalization of hyperbolicity and will be explained in more detail in chapter $5$):

\begin{restatable}{thm}{CAT}
\label{intro:CAT0}
The complex $P$ is a $CAT(0)$ cube complex.
\end{restatable}

This gives us a way to reinterpret relators in the spider as homotopies across cells in the Euclidean building. We can attempt a similar construction in $A_3$ in order to generalize the observation that $A_2$ spiders are $CAT(0)$ if and only if they are minimal-vertex. However, it's not so easy to prove that the resulting complex is $CAT(0)$ since we don't have the simple cube complex criterion.

\chapter{Quantum Groups, Hopf Algebras, and Tensor Categories} 
\linespread{1}\selectfont
This chapter is designed to provide an informal introduction to quantum groups for readers less familiar with the technical details. For a more technical introduction we direct the readers to one of the following three books. \cite{kassel:quantum} provides a more $\mf{sl}_2$ oriented approach followed by a lot of the general details for monoidal categories, and the connections to knot invariants. \cite{hk:crystal} focuses on a different approach using crystals, but the first 3 chapters give a concise introduction to quantum groups with many of the formal definitions we omit. Finally, \cite{cp:guide} provides a more comprehensive resource which includes more information about the origins of quantum groups and also various canonical bases (although not the ones we'll be discussing).
 
\section{Hopf algebras and Quantum groups}
In this paper, a \textbf{quantum group} (denoted $U_q(\fg)$) is a particular kind of Hopf algebra coming from a deformation of the universal enveloping algebra, which we explain more below. It is also called a \textbf{quantum enveloping algebra}.

A \textbf{Hopf algebra} (over $k$) is a $k$-algebra $A$ together with three additional structures: a comultiplication $\Delta: A \to A \ot A$, a counit $\epsilon: A \to k$, and an antipode $S: A \to A$ which satisfy various compatibility axioms. The first canonical example is a group algebra $kG$ of a finite group $G$, which has an augmentation map $\epsilon: kG \to k$ defined by adding up the coefficients, a comultiplication obtained by letting $\Delta(g)=g\ot g$ and then extending to the entire algebra, and an antipode determined by letting $S(g) = g^{-1}$. From a representation theory perspective we get the following: the counit defines a trivial representation. The comultiplication allows an element $x \in A$ to act on a tensor product of representations $V \ot W$ via $x.(v \ot w) = \Delta(x) (v \ot w)$. Finally, the antipode allows elements to act on the dual space $V^*$ of a representation $V$ via $(x.\ell)(v)=\ell(S(x).v)$ for all $\ell \in V^*$ and $v \in V$. In the language of the next section, we'll say these extra structures give the representation category the structure of a pivotal tensor category.

The second canonical example is the universal enveloping algebra $U(\fg)$, which is more closely related to this paper. Recall, we define the universal enveloping algebra as the quotient of the tensor algebra $T \fg$ where we enforce that the associative algebra commutator: $[x,y]_{U(\fg)} = xy - yx$ is equal to the corresponding Lie algebra commutator $[x,y]_\fg$ when $x$ and $y$ are in $\fg \sbs T \fg$. Next we want to define the Hopf algebra structure. Since Lie algebras correspond to derivations, we require that elements act on tensors via the product rule, ie 
\begin{equation}
x.(v \ot w) = (x.v) \ot w + x \ot (x.w) = (x \ot 1 + 1 \ot x)(v \ot w)
\end{equation}
which tells us that we want 
\begin{equation}
\Delta(x) = x \ot 1 + 1 \ot x \qquad \text{ for every } x\in \fg \sbs T\fg
\end{equation} and then we extend this to the rest of $\Ug$ by making $\Delta$ a homomorphism.

In both $kG$ and $U(\fg)$, the image of $\Delta$ is contained the symmetric part of $A \ot A$, and so we say that these two Hopf algebras are \textbf{cocommutative}. On the other hand, the quantum enveloping algebra $\Uqg$ is an algebra which is constructed to be non-cocommutative (over the field $\Q(q)$ where $q$ is a formal variable). Informally speaking, in the limit "$q \to 0$" the representation theory of $\Uqg$ approaches the representation theory of $\Ug$. In fact, there is a correspondence between the irreducible representations of $\Uqg$ and $\Ug$, and even the decomposition of tensors into irreducible representations is the same (in the sense that we get the same number of each irreducible representation). However, the difference lies in the so-called braiding which we'll define in the next section.

The full presentation for $\Uqg$ is not particularly illuminating without knowing its origin, and we won't use it directly in this paper, but we will transcribe it from \cite{hk:crystal} for the sake of tradition. First define the quantum number 
\begin{equation}
[n]_q :=\frac{q^n-q^{-n}}{q-q^{-1}}:=q^{n-1}+q^{n-3}+...+q^{-(n-3)}+q^{-(n-1)} \in \Q(q)
\end{equation}
and the quantum binomial coefficients: 
\begin{equation}
\sqbinom{n}{k}_q:=\frac{[n]_q\cdot [n-1]_q\cdot\cdot\cdot [n+1-k]_q}{[k]_q\cdot [k-1]_q\cdot\cdot\cdot [2]_q \cdot [1]_q} \in \Q(q)
\end{equation}

\begin{defn}
Let $\fg$ be a semisimple Lie algebra with an associated Cartan datum $(A,\Pi,\Pi^\vee, P, P^\vee)$ and indexing set $I$. Moreover, let the diagonalizing matrix of the Cartan matrix $A$ be $diag(s_i)$. The quantum universal enveloping algebra $U_q(\fg)$ is a Hopf algebra defined as follows. As a unital associative algebra over $\Q(q)$, it is generated by elements $e_i$, $f_i$ for $i \in I$, and $q^h$ for $h \in P^\vee$ and satisfying the following relations:
\begin{flalign}
&\quad q^0 = 1, q^h q^{h'} = q^{h+h'} \text{ for } h,h' \in P^\vee &&\\\
&\quad q^h e_i q^{-h} = q^{\al_i(h)} e_i \text{ for } h \in P^\vee &&\\\
&\quad q^h f_i q^{-h} = q^{-\al_i(h)} f_i \text{ for } h \in P^\vee &&\\\
&\quad e_i f_j - f_j e_i = \delta_{ij} \frac{q^{s_ih_i} - q^{-s_ih_i}}{q^{s_i} - q^{-s_i}} \text{ for } i,j \in I &&\\\
&\quad \sum_{k=0}^{1-a_{ij}} (-1)^k \sqbinom{1-a_{ij}}{k}_q e^{1-a_{ij} - k} e_j e_i^k = 0 \text{ for } i \neq j &&\\\
&\quad \sum_{k=0}^{1-a_{ij}} (-1)^k \sqbinom{1-a_{ij}}{k}_q f^{1-a_{ij} - k} f_j f_i^k = 0 \text{ for } i \neq j &&
\end{flalign}
Its coalgebra structures and antipode induced by the following:
\begin{flalign}
&\quad \Delta(q^{h}) = q^{h} \ot q^h &&\\\
&\quad \Delta(e_i) = e_i \ot q^{-s_i h_i} + 1 \ot e_i, \; \Delta(f_i) = f_i \ot 1 + q^{s_i h_i} \ot f_i &&\\\
&\quad \ep(q^h)=1, \; \ep(e_i)=\ep(f_i)=0 &&\\\
&\quad S(q^h) = q^{-h}, S(e_i) = -e_i q^{s_i h_i}, \;S(f_i) = - q^{-s_i h_i} f_i &&
\end{flalign}
For $h \in P^\vee$ and $i\in I$
\end{defn}

\section{Pivotal Tensor Categories and Braidings}
Intuitively, a tensor category (also called a monoidal category) is a category where we can take tensor products of objects and morphisms, while a pivotal (tensor) category (also called rigid) is a tensor category which has a duality functor. These come up in representation theory because the category of representations of a Hopf algebra is naturally a pivotal tensor category, where the tensor product comes from the comultiplication, and the duality comes from the antipode map.

Formally (following \cite{kassel:quantum}) a \textbf{strict tensor category} is a category $\ms{C}$  together with a bifunctor $\ot: \ms{C} \times \ms{C} \to \ms{C}$ called the tensor product satisfying
\begin{enumerate}
	\item Associativity:  $(U \ot V) \ot W =U \ot (V \ot W)$
	\item Unity: There exists an object $\mbf{1}$ such that $U \ot \mbf{1} = U$ and $ \mbf{1}\ot U = U$
\end{enumerate}

\begin{remark}
It is actually more natural to define the above equalities as natural isomorphisms making some diagrams commute. Such a category is called a (possibly non-strict) tensor category. Every tensor category is equivalent to a strict tensor category, so we will avoid these complications by sweeping this distinction under the rug
\end{remark}

A (strict) \textbf{pivotal category}, is a tensor category $\ms{C}$ together with a contravariant functor $D: \ms{C} \to \ms{C}$ (where $D(A)$ is denoted $A^*$ for every $A \in \Ob(\ms{C})$) together with contraction maps $d_U:U^* \ot U \to \mbf{1}$, and cocontraction maps $b_V: \mbf{1} \to U \ot U^*$ that are compatible with the monoidal structure and each other in the sense that: 
\begin{enumerate}
	\item Duality reverses the order of tensors: $(A \ot B)^*=B^* \ot A^*$
	\item The maps $\Hom(U \ot V, W) \xto{(\bu \ot id_V) \circ (id_U \ot b_V)} \Hom(U, W\ot V^*)$ and $\Hom(U, W\ot V^*) \xto{(id_W \ot d_V) \circ (\bu\ot id_{V})} \Hom(U \ot V, W)$ are inverses
	
	\item Similarly, the maps $\Hom(U \ot V,  W) \to \Hom(V, U^* \ot W)$ and $\Hom(V, U^* \ot W) \to \Hom(U \ot V,  W)$ induced by $b_{U^*}$ and $d_{U^*}$ are inverses
\end{enumerate}
\begin{remark}
For axiom $2$, by letting $U=\mbf{1}$ and $V=W$ we get the more commonly used axiom $(id_U \ot d_U) \circ (b_U \ot id_U) = id_U$ and similarly for axiom $3$ and $(d_U \ot id_U) \circ (id_U \ot b_U) = id_U$
\end{remark}

This property is convenient because any $\Hom$ space is in natural bijection with any $\Hom$ space of the form $\Hom(\mbf{1},V)$. In a representation category, this is in natural bijection with the set invariant vectors in $V$ denoted $Inv(V)$ (we just look at the image of $1 \in \C$ under this map), and hence understanding the category reduces to understanding the invariant vectors along with their image under the various natural morphisms. Moreover, many of the invariant spaces are in natural bijection:

\begin{fact}
$Inv(U \ot V) \simeq Inv(V \ot U^{**}) $ 
\end{fact}
\begin{proof}
$Inv(U \ot V) \natiso \Hom(\mbf{1},U \ot V) \natiso \Hom(U^*, V) \natiso \Hom(\mbf{1}, V \ot U^{**})) \natiso Inv(V \ot U^{**})$
\end{proof}
Hence, tensors products are invariant under \tit{cyclic} rotations as long as $U^{**} \simeq U$. However, it so happens that in the case of quantum group representations, there are natural isomorphisms for \tit{any} permutation (not just cyclic permutations). This extra structure is called a braiding. We won't need the braiding in this paper, however given that it is one of the primary motivations of quantum groups, we define it below:

A (strict) \textbf{braided tensor category} is a tensor category $\ms{C}$ with a family of natural morphisms $c_{V,W}: V \ot W \to W \ot V$ indexed by $\Ob(\ms{C})\times \Ob(\ms{C})$ and called the \textbf{commutativity constraint} such that commuting an object $U$ past $V$ then $W$ is the same as commuting past $V \ot W$ in one go, ie: $U \ot (V \ot W) \xto{c_{U,V \ot W}} (V \ot W) \ot U$ is the same as $U \ot V \ot W  \xto{c_{U,V} \ot id_W} V \ot U \ot W \xto{id_V \ot c_{U,W}} V \ot W \ot U$.

In the case of the category of representations of a group or Lie algebra, we get a braiding defined by the swap map $c_{V,W}: v\ot w \mapsto w \ot v$ for every $v \in V$ and $w \in W$. This has the added property that swapping two objects and then swapping back is the identity, ie $c_{W,V}c_{V,W}=id_{V \ot W}$. In general, if a braided tensor category has this property, we say it is a \textbf{symmetric tensor category}. However, the fact that the swap map is a morphism of representations follows from the fact that the comultiplication of a group or Lie algebra is cocommutative. One of the key values of quantum group is that they are \textit{not} cocommutative, so the flip map is not a morphism. However, with some work, one can define a non-symmetric braiding on quantum groups. This braiding then allows us to construct invariants of links, but that is beyond the scope of this paper.
\chapter{Presentations of Tensor Categories}
\linespread{1}\selectfont

In this chapter we will give some of the technical definitions required to state our results in a rigorous fashion. Most of the chapter can be skipped and returned to for definitions if needed, except for \cref{Canceling} which generalizes the diamond lemma and is used to prove $E^1$ convergence in the case of $A_3$.
\section{Introduction}
\label{PoTC:Intro}
Ordinary confluence is ultimately a question about generators and relators, and so in order to generalize confluence to tensor categories we need to first define what we mean by generators and relators. In the case of abelian groups, a presentation of a group $G$ with $n$ generators and $m$ relators is defined as an exact sequence $ F_m \to F_n \to G \to 0$. Therefore, by analogy, we need to first define a notion of a free pivotal tensor category. It will be a category with fixed generating objects and morphisms, and having as few relations as possible while still being a pivotal tensor category.

The primary difficulty of this is that our generating morphisms won't be maps between two generating objects. Rather, they will correspond to contraction maps which involve the tensor product of three fundamental representations. In order to avoid the technicalities of the more direct construction, we will give a (somewhat informal) graphical definition of the desired category, which will be sufficient for what we want to prove.

Let $\sV$ be a finite set (corresponding to generating objects), and let $D: \sV \to \sV$ such that $D^2=id_{\sV}$ (corresponding to duality between the generating objects). We say $V \in \sV$ is \textbf{self-dual} if $D(V)=V$. This induces a monoidal antihomomorphism on the free monoid $\sV^*$. We will abuse notation and also denote it $D$ since there is little danger of confusion. 

Next, let $\sT$ be a finite subset of $\sV^*$ such that
\begin{enumerate}
	\item $\sT$ is closed under cyclic permutation of the letters in a word (eg: if $abc \in \sT$, then so is $cab$)
	\item $\sT$ is closed under the monoidal antihomomorphism induced by $D$ (eg: if $abc \in \sT$, then $D(abc)=D(c)D(b)D(a)$ is too)
\end{enumerate}
In our setting, generating objects $a_i \in \sV$ will correspond to fundamental representations $V_i$, and an element $a_1a_2...a_n \in \sT$ will correspond to a generating morphism in $Hom(V_1 \ot V_2 \ot... \ot V_n,\K)$. 

For example, we could have $\sV=\lb V,W \rb$, $D(V)=W$, and $\sT = \lb VVV , WWW\rb$ which corresponds to the free $A_2$ spider. $\sT = \lb VVV\rb$ wouldn't be allowed because $D(VVV)=WWW$ would then not be in $\sT$. In the cases we are concerned with here, $\sT$ will contain words of length 3 and/or 4.

Given a triple $(\sV,\sT,D)$ we can define special colored, partially directed graphs called webs. In order to avoid double counting with the duality map, we will pick a distinguished element in each set $\lb V , D(V) \rb$. This won't change the category, but will change how each web is drawn.

\begin{defn}
A $(\sV,\sT,D)$-\textbf{web} is an isotopy class of planar graphs $w$ drawn inside the unit square $[0,1] \times [0,1]$ as follows. Each edge is labeled by a distinguished element $V \in \sV$, together with a direction whenever $V$ isn't self-dual. There is a set of equally spaced boundary vertices on the top or bottom of the unit square (ie $[0,1] \times \lb 0,1 \rb$) where we require each such vertex to have valence $1$. The \textbf{type} $\tau(v)$ (for any internal or boundary vertex) is the cyclically ordered sequence of labels of adjacent edges, where each edge directed inwards contributes the type $V$, and each edge directed outwards contributes the type $D(V)$. We require that type of an internal vertex be in $\sT/(cyclic \; permutations)$. If the boundary vertices on $[0,1] \times \lb 0 \rb$ are $(v_1,...,v_n)$, and the boundary vertices on $[0,1] \times \lb 1 \rb$ are $(w_1,...,w_m)$ we define the \textbf{source} and \textbf{target functions} as $s(w)=D(\tau(v_1))...D(\tau(v_n)) \in \sV^*$, and $t(w)=\tau(w_1)...\tau(w_n) \in \sV^*$ respectively, where these strings are well defined because boundary vertices are adjacent to only a single edge.
\end{defn}

For example, in the case of the $A_2$ pair from above $(\lb V,W \rb, D(V)=W, \lb VVV$ , $WWW\rb)$, we choose $V$ as the distinguished element, and denote it by a single edge. In the example web $w$ below, the left trivalent vertex is type $VVV$ since its edges point inwards while the right trivalent vertex is type $WWW$ since its edges point outwards. $s(w)=VWV$ and $t(w)=WVW$

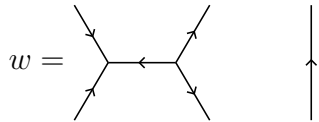
$$w=\begin{tikzpicture}[baseline=-.5ex,scale=.45]
\draw[midto,web] (-1,1.7) -- (0,0);
\draw[midfrom, web] (0,0) -- (2,0);
\draw[midto, web] (-1,-1.7) -- (0,0);
\draw[midto, web] (2,0) --  (3,1.7);
\draw[midto, web] (2,0) -- (3,-1.7);
\draw[midto, web] (6,-1.7) -- (6,1.7);
\end{tikzpicture}$$

Fix a ground field $K$. We will usually be working over $K=\Q(q)$
\begin{defn}
The \tbf{free spider} $FSp=FSp(\sV,\sT,D)$ is a pivotal category whose objects are strings in the free monoid $\sV^*$, and whose morphisms are $K$-linear combinations of webs which all have the same source and target (the source and target function of the category are then defined as these unique shared source and target of the webs in the linear combination). By abuse of notation, we will consider webs as special morphisms in $FSp$ which we denote as $Web$, and these give a basis for the morphisms by definition. A composition of webs $w_2 \circ w_1$ is done by "stacking vertically", in other words gluing the bottom edge of the first box to the top of the second box. Tensor products are done by "putting two webs side by side", in other words gluing the right side of the first box to the left side of the second box\footnote{Slightly more rigorously, the composition is the image of the map $([0,1] \times [0,1])_{w_2} \sqcup ([0,1] \times [0,1])_{w_1} \to  ([0,1] \times [0,1])_{w_2 \circ w_1}$ by linear contractions $([0,1] \times [0,1])_{w_1} \to  ([0,1] \times [0,1/2])_{w_2 \circ w_1}$ and $([0,1] \times [0,1])_{w_2} \to  ([0,1] \times [1/2,1])_{w_2 \circ w_1}$, then forget the vertices on the middle line $[0,1] \times \lb 1/2 \rb$. We can do the same thing for tensors except we contract horizontally, and then place the two webs horizontally}. 
\end{defn}
\begin{remark}
In our pictures we will allow for boundary vertices that aren't uniformly distributed by remembering that we can always isotopy the web into the correct form if needed
\end{remark}

Now that we have defined a free object, we need to define relations. Take a finite set $R \sbs \Mor(FSp)$ which we call the set of local relators, and let $\lan R \ran$ be the smallest linear subset of $Mor(FSp)$ which contains $R$ and is closed under compositions and tensors with morphisms of $FSp$. We then define a new category $Sp$ such that $\Ob(Sp)=\Ob(FSp)$, and $\Hom_{Sp}(X,Y) = \Hom_{FSp}(X,Y)/(\lan R \ran \cap \Hom_{FSp}(X,Y))$ with the induced tensor products and compositions.

This works abstractly, but we want to make this more combinatorial so that we can work with webs rather arbitrary morphisms, and local relators rather than general relations. To do this, we define the set of relators $Rel \sbs \lan R \ran$ as the smallest (generally non-linear) subset containing $R$ and closed under composition and tensors by \tit{webs}. Intuitively, these will correspond to locally replacing a subgraph of the web with a linear combination of webs. For example, we have a $A_2$ local relator 
$$\begin{tikzpicture}[baseline=-2ex,scale=.55]
\draw[midfrom,web] (-1.68,-1.5) -- (-1,-1);
\draw[midto,web] (-1.68,0.5) -- (-1,0);
\draw[midto,web] (0.68,-1.5) -- (0,-1);
\draw[midfrom,web] (0.68,0.5) -- (0,0);
\draw[midto,web] (-1,-1) -- (0,-1);
\draw[midto,web] (-1,-1) -- (-1,0);
\draw[midto,web] (0,0) -- (0,-1);
\draw[midto,web] (0,0) -- (-1,0);
\end{tikzpicture} = 
\begin{tikzpicture}[baseline=-2ex,scale=.55]
\draw [web, midfrom]  (-1.68,-1.5) to[out=60,in=-60] (-1.68,0.5);
\draw [web, midto]   (-.3,-1.5) to[out=120,in=240] (-.3,0.5);
\end{tikzpicture} +
\begin{tikzpicture}[baseline=-.5ex,scale=.65]
\draw [web, midfrom]  (-1.68,-.75) to[out=30,in=150] (-.3,-.75);
\draw [web, midto]   (-.3,0.5) to[out=210,in=-30] (-1.68,0.5);
\end{tikzpicture}  $$
After composing with 
$$
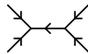
\begin{tikzpicture}[baseline=-2ex,scale=.45]
\draw[midfrom,web] (-1,0) -- (0,0);
\draw[midfrom,web] (-1,0) -- (-1.7,.7);
\draw[midfrom,web] (-1,0) -- (-1.7,-.7);
\draw[midfrom,web] (0,0) -- (.7,.7);
\draw[midfrom,web] (0,0) -- (.7,-.7);
\end{tikzpicture}$$
we get a relator containing each term in the original local relator as a subgraph:
$$\begin{tikzpicture}[baseline=-2ex,scale=.55]
\draw [blue] (-.5,-.5) circle (.9cm);
\draw[midfrom,web] (-1.2,-1.8) -- (-1,-1);
\draw[midto,web] (-1.68,0.5) -- (-1,0);
\draw[midto,web] (.2,-1.8) -- (0,-1);
\draw[midfrom,web] (0.68,0.5) -- (0,0);
\draw[midto,web] (.2,-1.8) -- (-1.2,-1.8);
\draw[midto,web] (.2,-1.8) - -(.7,-2.3) ;
\draw[midfrom,web]  (-1.2,-1.8) -- (-1.7,-2.3);

\draw[midto,web] (-1,-1) -- (0,-1);
\draw[midto,web] (-1,-1) -- (-1,0);
\draw[midto,web] (0,0) -- (0,-1);
\draw[midto,web] (0,0) -- (-1,0);
\end{tikzpicture} = 
\begin{tikzpicture}[baseline=-2ex,scale=.55]
\draw [blue] (-1,-.5) circle (.9cm);
\draw [web, midfrom]  (-1.68,-1.5) to[out=60,in=-60] (-1.68,0.5);
\draw [web, midto]   (-.3,-1.5) to[out=120,in=240] (-.3,0.5);
\draw[midto,web] (-.3,-1.5) -- (-1.68,-1.5);
\draw[midto,web]  (-.3,-1.5) - -(.2,-2) ;
\draw[midfrom,web]  (-1.68,-1.5) -- (-2.18,-2);
\end{tikzpicture} +
\begin{tikzpicture}[baseline=-.5ex,scale=.55]
\draw [blue] (-1,0) circle (.8cm);
\draw [web, midfrom]  (-1.68,-.75) to[out=30,in=150] (-.3,-.75);
\draw [web, midto]   (-.2,0.6) to[out=210,in=-30] (-1.78,0.6);
\draw[midfrom,web] (-1.68,-.75) to[out=-30,in=210]  (-.3,-.75);
\draw[midfrom,web]  (-1.68,-.75) -- (-2.18,-1.25) ;
\draw[midto,web] (-.3,-.75) -- (.2,-1.25);
\end{tikzpicture}  $$
Note that the elements of $R$ will be called \tbf{local relators}, the elements of $Rel$ will be called \tbf{relators}, and the elements of $\lan R \ran$ will be called \tbf{relations}. Each relator $r \in Rel$ corresponds to a local relator $\rho$. The \tbf{type} of $r$ (denoted $\tau(r)$) will be defined as the equivalence class of $\rho$ under equivalence by rotation and duality. In practice this will be well defined because our relators will be restricted to single faces.

The key point is that just as webs span $FSp$, we get the following:

\begin{lemma}
\label{presentation::span}
$\Span_{\Q(q)} Rel = \lan R \ran$
\end{lemma}
\begin{proof}
\begin{enumerate}
\item $(\sbs)$ This follows immediately from $Web \sbs Mor(FSp)$
\item $(\sps)$ We just need to prove that $\Span_{K} Rel$ is closed under compositions and tensors with $Mor(FSp)$ so $R \sbs Rel$ will imply $\lan R \ran \sbs \lan Rel \ran = \Span_{K} Rel$, but this follows from linearity of tensors and compositions. Indeed, if $r= \sum\limits_i a_i r_i \in Rel$ and $x=\sum\limits_j b_jw_j \in FSp$ where $a_i, b_j \in K$, $r_i \in \Span_{\Q(q)} Rel$, and $w_j$ are webs, then $r \circ x=\sum\limits_{i,j}a_ib_j (r_i \circ w_j)$, and similarly for tensors and compositions on the left side, and so the result is a linear combination of relators, which is what we needed to show.
\end{enumerate}
\end{proof}

Even though our definition of webs involves picking a target and source, any $\Hom$ space whose cyclically oriented set of boundary edges is the same is naturally isomorphic (since our category is pivotal), so we will typically draw our webs in a circle rather than a rectangle when we don't need to take tensors or compositions.

\section{Generalized Confluence: \texorpdfstring{$E^k$}{} Convergence}
Recall that $FSp$ is the free spider whose morphisms are linear combinations of webs, $\lan R \ran$ is the set of relations generated by a finite set of morphism $R$, while $Rel$ is the set of relators (of morphisms) corresponding to local replacement rules. At this point, we could try to generalize the notion of spectral sequences to tensor algebras, but since this introduces more complications without much apparent utility, we will stick to the classical setting of filtered, graded, differential vector spaces. First, we will fix two objects $X,Y \in Ob(FSp)$. For simplicity we define shorthand $Sp(X,Y):=\Hom_{Sp}(X,Y)$, $FSp(X,Y):=\Hom_{FSp}(X,Y)$, $Rel(X,Y):=Rel \cap \Hom_{FSp}(X,Y)$, $\lan R \ran (X,Y):=\lan R \ran \cap \Hom_{FSp}(X,Y)$,  and $Web(X,Y):=Web \cap \Hom_{FSp}(X,Y)$.

$FSp(X,Y)$ is naturally $\Z^k$ graded, where $k$ is the number of vertex types, just by counting the number of internal vertices of each type in a web. In the case of $A_3$ webs, there is only one type of vertex up to duality, so we will collapse the grading into a $\Z$ grading by just counting the total number of internal vertices, forgetting the type. If $w \in Web(X,Y)$, we will write $|w|$ to denote its grading, and more generally if $x = \sum_i a_i w_i \in FSp(X,Y)$ (where $a_i \in \Q(q)$ and $w_i \in Web(X,Y)$)  we will write $|x| = \max{\lb |w_i| : b_i \neq 0 \rb}$. Since $\lan R \ran$ isn't a homogeneous subspace, we will replace it with the following abstract graded vector space:

\begin{defn}
\label{RSP}
Let $RSp(X,Y)$ be the graded vector space with formal basis $\lb [r]: r \in Rel(X,Y) \rb$, where $|[r]| := |r|$ for all $r \in Rel(X,Y)$.
\end{defn}

Notice that by lemma \ref{presentation::span}, we obtain a surjection $d:RSp(X,Y) \surj \lan R \ran(X,Y)$ induced by $d([r]) = r$. Due to lower order terms, it isn't a graded map, but it is a filtered map via the induced filtration $F^k V = \lb x \in V : |x| \leq k \rb$. Hence, we get a filtered, free presentation of vector spaces:
\begin{equation}
RSp(X,Y) \xto{d} FSp(X,Y) \to Sp(X,Y) \to 0
\end{equation}
Where $\lan R \ran (X,Y)$ is exactly the image of $RSp(X,Y)$ in $FSp(X,Y)$. Taking the vector space $RSp(X,Y) \op FSp(X,Y)$ graded via the direct sum, we get a graded, filtered vector space with a filtered differential $d: RSp(X,Y) \op FSp(X,Y) \to RSp(X,Y) \op FSp(X,Y)$ which is defined as above on $RSp(X,Y)$, and is the zero map on $FSp(X,Y)$. Any filtered short exact sequence of vector spaces $R_V \xto{d} W \to V \to 0$ induces a $2$-column spectral sequence $E^k_{ij}$ where $E^0_{0j} = gr_j R_V := F^j R_V/F^{j-1} R_V$ and $E^0_{1j} =  gr_j W := F^j W/ F^{j-1} W$ where the differential $d^k: E^k_{0j} \to E^k_{1(j-k)}$ (rather than knight moves). Rigorously we define $E^{k}_{ij}$ recursively as $ker(d^k)/im(d^k)$, where $d^k(\ol{x})= \ol{d(x)}$ for $\ol{x} \in E^{k-1}_{ij}$ and $\ol{d(x)} \in E^{k-1}_{(i+1)(j-k)}$. Due to the simple structure, $E^k_{0j}$ is the subset of $gr_j R_V$ of elements which were zero on all previous differentials, while $E^k_{1j}$ is a quotient of $gr_j W$. Moreover, by the general theory, we get eventually convergence to $E^\infty_{1j}:=E^k_{1j}$ for $k \gg 0$, and $E^\infty_{1j}$ is exactly $gr_j V$. 

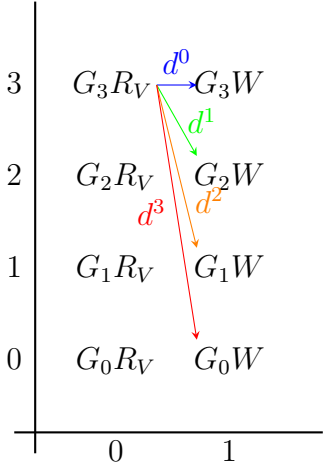
\begin{tikzpicture}
  \matrix (m) [matrix of math nodes,
    nodes in empty cells,nodes={minimum width=5ex,
    minimum height=5ex,outer sep=-5pt},
    column sep=1ex,row sep=1ex]{
               &      &     &  \\
          3     &  G_3 R_V \:    &  G_3 W   &  \\
          2     &  G_2 R_V    & G_2 W    &  \\		
          1     &  G_1 R_V &  G_1 W  &  \\
          0     &  G_0 R_V &  G_0 W &  \\
    \quad\strut &   0  &  1  &   \strut \\};
  \draw[blue,-stealth](m-2-2.east) --  (m-2-3.west) node[midway,above] {$d^0$};
  \draw[green,-stealth](m-2-2.east) --  (m-3-3.north west) node[midway, right] {$d^1$};
\draw[orange,-stealth](m-2-2.east) --  (m-4-3.north west) node[midway, right, xshift=.1cm, yshift=-.45cm] {$d^2$};
\draw[red,-stealth](m-2-2.east) --  (m-5-3.north west) node[midway, left] {$d^3$};
\draw[thick] (m-1-1.east) -- (m-6-1.east) ;
\draw[thick] (m-6-1.north) -- (m-6-4.north) ;
\end{tikzpicture}

We will show that convergence of this spectral sequence at the first page $E^1$ (uniformly with respect to the choice of $X$ and $Y$ in the case of presented spiders) is a generalization of confluence, however let's first see what this convergence means concretely in the case of spider presentations. The differential $d^0: gr_n RSp \to gr_n FSp$ is induced by the map $d$, i.e. $d^0 (\ol{[r]}) = \ol{r}$ for $r \in Rel$. This corresponds to forgetting about the lower order terms of $d^0(r)$. If we denote the set of highest order terms of $R$ by $L(R)$, then $E^1$ convergence corresponds to a natural isomorphism $Sp \cong FSp / \lan L(R) \ran$. The subsequent differentials $d^k$ are applied to $r= \sum_i a_i r_i$ which are zero in $gr_{|r|} FSp$, or in other words the leading terms cancel. We therefore need to prove that we can't get any new relations in $Sp$ via sums of relators where the leading order terms cancel. In the next section we will characterize this convergence in a more combinatorial way.

\section{A Generalized Diamond Lemma: Canceling Sequences}
\label{GeneralizedDiamond}
We wish to characterize $E^1$ convergence of the spectral sequence defined in the previous section in a more specific setting. In the cases we will be addressing in this paper, the leading term of each relator will consist of either one web or the difference of two webs. Unfortunately, in high ranks this generally fails. In that case, the definitions and theorems from this section can be extended, but the resulting criterion no longer appears as to be as useful. In any case, we don't get $E^1$ convergence for semisimple spiders such as $A_2 \times A_2$ or $A_1^2 \times A_2$ which suggests that this property may fail in most higher rank spiders. 

Take a presentations of filtered vector spaces $R_V \xto{d} W \to V \to 0$, together with a distinguished basis of $W$ which we'll call webs (for notational consistency with spiders), and a distinguished basis of $R$ which we'll call relators and denote $Rel$. If $r$ is an element of $Rel$, $L(r)$ will denote the leading order terms of $d(r)$ in $W$. In particular, if $r$ is a relator such that $L(r)$ is a web, we'll refer to $r$ as a \tbf{reduction relator}, and if $L(r)$ is the difference of webs, we'll call $r$ a \tbf{horizontal relator}. We'll assume that all relators are either reduction relators or horizontal relators. Note that the relator basis of $R_V$ together with the filtration gives us a grading as follows: we make each relator homogeneous with grading $|r| = \max\limits_{n \in \N} \lb d(r) \in F^n W \rb$, so we'll assume that the filtration on $R_V$ is the one induced by this grading.

What makes this case different is the existence of horizontal relators. There is a more typical approach to this type of problem which is used to prove the existence of PBW or Gr\"ober bases. That idea is to break symmetry, and give a refined filtration so that the horizontal relators become reduction relators. However, in this case, that method doesn't appear to be necessary.

Just as in the case of confluence, one thing that can go wrong is when two reductions of the same web reduce in two different irreconcilable ways. However, in this case horizontal relators make things more complicated. For one thing, the two different reductions might be on two different webs which are connected (up to lower order terms) by horizontal relators. Even more different, we might have a sequence of horizontal relators which connect a single web to itself, but the lower order terms don't vanish. We state these two problems more rigorously as follows:  

\begin{defn}
\begin{enumerate}
	\item A sequence of relators is n-tuple $(s_1,...,s_n)$ such that $L(s_i)=w_i-w_{i-1}$ for some sequence of webs $(w_0,...,w_n)$, where $w_0$ and $w_n$ are allowed to be zero.
	\item A reduction sequence is a sequence of relators $s = (s_0,s_1,...,s_n,s_{n+1})$ such that $L(s_0) = w_0$, $L(s_{n+1})=w_n$,  
	\item A horizontal sequence is a sequence of relators $s = (s_1,...,s_n)$ such that $w_0 = w_n$
	\item A canceling sequence is any sequence of relators that is either a reduction sequence or horizontal sequence
\end{enumerate}
\end{defn}

If $s$ is a sequence of relators, we write $\sum s$ as shorthand for the sum $\sum_i [s_i]$. Note that if $s$ is a canceling sequence then by construction, $|d(\sum s)| < |\sum s|$ since the higher order terms cancel. As shorthand, we will denote $|\sum s|$ by $|s|$ (which is also equal to the norm of any relator in the sequence $|s_i|=|[s_i]|$).

\begin{defn}
Let $s = (s_i)$ be a canceling sequence. We say $s$ is a \tbf{consistently reducible} if there exists a relation $r \in R_V$ such that $|r| < |s|$, and $d(\sum s) = d(r)$.
\end{defn}

This recursive definition might be a little different from what one might normally define, but being consistently reducible will imply that $\sum s$ doesn't give us any new relations in $Sp$ that didn't already appear in smaller webs. In particular, we can prove that this is all we need to check in general:

\begin{prop}
\label{Canceling}
Let $R_V \to W \to V \to 0$ be a presentation of a filtered vector spaces as above, such that the leading term of each relator is either a single web or a difference of webs, then the presentation converges on the $E^1$ page if and only if all canceling sequences are consistently reducible.
\end{prop}

This will follow almost immediately from the following decomposition lemma (recalling the definition of $RSp$ in \cref{RSP}):

\begin{lemma}
\label{decomp}
With the hypotheses of proposition \ref{Canceling}, if $r \in RSp$  is homogeneous, and $|d(r)| < |r|$ then $r=\sum_i b_i (\sum s^i)$ where $\lb s^i \rb$ is a set of canceling sequences, $b_i \in K$ for all $i$, and $|s^i|=|r|$.  
\end{lemma}
\begin{proof}
Let $r = \sum\limits_{i=1}^n a_i [r_i]$ where $r_i \in Rel$ and $a_i \in K^\times$ (where $K^\times$ is the units of the base field of our vector spaces). We induct on $n$, where the basis case of $n=0$ is vacuously true. Since $r$ has a vanishing leading term, every web that appears in $L(r_i)$ must appear in the leading order term of at least two relators. So if I have a sequence of relators $(r_{k_i})$ where each is a relator in the decomposition of $r$, I can always extend this sequence to another with the same property, provided that the first and last relators aren't both reduction relators. If I keep extending I will eventually get a reduction sequence or a repeated relator and hence a horizontal sequence as a subset. Denote this sequence $s^1$. If $r_{k_1}$ is the first relator in this sequence, we know $r-a_{k_1}\sum s^1$ no longer has a non-zero multiple of $r_{k_1}$ as a summand. Therefore it has at least one fewer relator in its decomposition, we can apply induction.
\end{proof}

\begin{proof}[Proof of Proposition \ref{Canceling}]
\begin{itemize}
	\item $(\Leftarrow)$ Let $\ol{r} \in \gr R_V$ be an arbitrary relation such that $d^0(\ol{r})=0$. We need to prove that $d^i(\ol{r})=0$ for all $i \in \N$ which will show that $d^i$ is trivial on all relations for $i>0$. Take $r$ to be the homogeneous representative. Since $d^0(\ol{r})=0$, we know that $d(r) < |r|$, so by \cref{decomp}, we have $r=\sum_i b_i (\sum s^i)$. By definition of being consistently reducible, we know there exists relations $r_i \in R_V$ s.t. $|r_i| < |s^i|=|r|$ and $d(r_i)=d(\sum s^i)$. This implies that $d(\sum b_ir_i)= d(r)$ and $|\sum b_ir_i| < |r|$. However, this tells us that $\ol{r}=\ol{r-\sum b_ir_i}$, but $d(r-\sum b_ir_i)=0$. Since the spectral sequence differentials $d^i$ are induced by $d$, this says that they $d^i(\ol{r})=0$ for all $i\in \N$.
	\item $(\Rightarrow)$ Let $s$ be a canceling sequence. Since the spectral sequence converges on the first page, and $d^0(\ol{\sum s})=0$ by definition of a canceling sequence, we know $d^i(\ol{\sum s})=0$ for all $i \in \N$. This shows that either $d(\sum s)=0$, or $d(\sum s)=d(r)$ for some relator $r$ with $|r|<|s|$, but that's the definition of being consistently reducible.
\end{itemize}
\end{proof}

Before we go on, we want to use this proposition to prove that $E^1$ convergence coincides with the notion of confluence when relators have only one leading term, or in our terminology: when all relators are reduction relators. We will first recall what confluence means. A \tbf{rewriting system} is a set together with \tbf{rewrite rules} which can be considered as directed arrows connecting elements of the set. This defines a directed graph. A rewriting system is \tbf{confluent} if for any two directed paths $\ga$ and $\ga'$ starting at a vertex $x$ in this graph, we can extend these paths into paths $\wt{\ga}$ and $\wt{\ga'}$ which end at the same element of this set. Assume that we have a presented filtered vector space $V$ as defined above. There is a natural rewrite system whose set is the set of all elements of $W$, and the rewrite rules are of the form $a v_{k_i} + x \mapsto a (v_{k_i} - r_i) + x$ where $L(r_i)=v_{k_i}$ and $x\in W$ whose decomposition into the basis doesn't contain $v_{k_i}$. In particular, spider presentations are filtered vector spaces, where the webs form the distinguished basis of $FSp$, the relators form the distinguished basis of $RSp$, and hence we obtain a canonical rewriting system. From this we can now prove \cref{intro:confluence} from the introduction which we restate below:

\Confluence*
\begin{remark}
In particular this holds for spider presentations.
\end{remark}
\begin{proof}
Notice that we only have reduction sequences consisting of two reduction relators, and no horizontal sequences.
\begin{itemize}
	\item $(\Leftarrow)$ Given two rewriting sequences, we can extend them to maximal rewriting sequences $x=a_nr_n +... +a_1r_1+w$ and $x'=a_m'r_m'+ +...+ a_1'r_1' + w$ which apply to the same web $w$ (these exist because we assumed that all descending chains stabilize). Then, $x-x' = a_nr_n + ... a_1r_1 - a_1'r_1' -...-a_m'r_m'$ is a relation in $d(R_V)$, and hence by spectral sequence convergence there exists a relation $\rho \in R_V$ such that $\ol{d(\rho)}=x-x'$ and $|\rho| = |d(\rho)|=|x-x'|$. But since there are no reduction relators that can be applied to webs in $x$ or $x'$ by maximality, $\rho$ must be trivial, and hence $x=x'$.
	\item $(\Rightarrow)$ By proposition \ref{Canceling} we just need to make sure that for every web $w$ and pair of reduction relators $r$ and $r'$ on $w$, there is a relation $\rho$ such that $|\rho|<|r-r'|$ and $d(\rho) = d(r-r')$, but by confluence, the maximal rewriting sequences of relators $r_1,...r_n$ and $r_1',...,r_m'$ must give $d(a_nr_n+...+a_1 r_1 +r +w)=d(a_m'r_m'+...a_1'r_1'+r'+w)$, and hence $d(r-r') = d(r_1+...+r_n-r_1'-...-r_m')$ so taking $\rho = r_1+...+r_n-r_1'-...-r_m$ gives us our result.
\end{itemize}
\end{proof}

\section{Commuting Relators and Notation}

For convenience, we will define some notation. If $r$ is a horizontal relator with $L(r) = w'-w$ where $w$ and $w'$ are webs, then we will write $r(w)=w'$, and say that $r$ \tbf{is a relator on} $w$. Each relator corresponds to a web, and hence has boundary edges. We'll call all non-boundary edges in a local relator \tbf{interior edges}. Since relators are locally defined, if two relators $r$ and $t$ on $w$ don't share any interior edges, there is a relator $r'$ on $t(w)$ which is identical to $r$ locally. In this case, we will both say that $r$ and $t$ \tbf{commute}, and $t^{-1}$ and $r$ commute. Thus applications of other relators will give an equivalence class of relators of the same type on different webs. We will sometimes write $t(r)=r'$ if we want to distinguish $r$ and $r'$. Moreover, in order to avoid notational bloat, if $r$ is not a relator on $w$, then $r(w)$ will be short hand for $\rho(w)$ for some $\rho \sim r$ which is on the web $w$. In the case of a sequence of relators, we need the following definition:

\begin{defn}
If $(s_i)$ is a sequence of relators, then we say $s_i$ and $s_j$ commute (where $i<j$) if $s_{j-1}...s_{i+2}s_{i+1}(s_i)$ exists and commutes with $s_j$
\end{defn}



We will next give a list of basic facts and lemmas that will let us simplify canceling sequences. We say two canceling sequences $s$ and $s'$ (with $|s|=|s'|$) are \tbf{equivalent} and write $s \sim s'$ if there exists a relation $r$ with $|r|<|s|=|s'|$, such that $d(r) = d(\sum s - \sum s')$. In particular we get
\begin{fact}
If $s \sim s'$, then $s$ is consistently reducible iff $s'$ is consistently reducible
\end{fact}
\begin{proof}
If $s'$ is consistently reducible, there exists an $r'$ such that $d(r')=d(\sum s')$ and $|r'| < |s'|$. Moreover, since $s \sim s'$, there also exists an $r$ such that $d(r) = d(\sum s - \sum s')$ and $|r| < |s|=|s'|$, so $d(r+r')=d(\sum s)$, and $|r+r'| \leq \max\lb r, r' \rb < |s|=|s'|$. So $r+r'$ is the relation proving that $s$ is consistently reducible. The reverse direction then follows from the symmetry in the statement.
\end{proof}

The main types of equivalence arise from commuting relators, canceling inverse relators, and relators that can be commuted out of the sequence, ie the following facts:

\begin{fact}
If $s$ is a canceling sequence such that $s_i$ and $s_{i+1}$ commute for some $i$, and $s'$ is the sequence where we reverse the order, (ie $s'=(...,s_{i-1},s_{i+1}',s_i',s_{i+2},...)$ where $s_{i}'s_{i+1}'(w) = s_{i+1}s_{i}(w)$), then $s \sim s'$.
\end{fact}
\begin{proof}
We know $s_{i} = w_{i+1}-w_i +r_i$, and $s_{i+1} = w_{i+2}-w_{i+1} +r_{i+1}$ (where $|r_i|<|s_i|$ for all $i$). $s_{i+1}'=w_{i} + (-s_i)(r_{i+1})$ and $s_{i}'=w_{i+1} + (-s_{i+1})(r_{i})$ (where $(-s_i)(r_{i+1})$ means applying the relator locally identical to $(-s_i)$ to each web term of $r_{i+1}$), so the difference between the sum of the two sequences is $r_i-(s_{i+1})(r_i)+r_{i+1}-(-s_{i})(r_{i+1})$. However, we know that considering each relator as a graphical replacement rule, it doesn't matter which order we replace. These two sequences differ from that only in the relators applied to the lower order terms, so these lower order relators give us our $r$.
\end{proof}

\begin{fact}
If $s$ is a canceling sequence such that $s_{i+1}=-s_i$ for some $i$, and $s'$ is the sequence where they are omitted, (ie $s'=(...,s_{i-1},s_{i+2},...)$), then $s \sim s'$.
\end{fact}
\begin{proof}
Indeed, $d(\sum s) = d(\sum s')$, so they are immediately equivalent.
\end{proof}

\begin{lemma}
\label{commute out}
If $s=(s_0,s_1,...,s_{n-1},s_n,s_{n+1})$ is a reduction sequences such that $s_n$ and $s_{n+1}$ commute, then $s \sim s'=(s_0,s_1,...,s_{n-1},s_{n+1}')$ where $s_{n+1}'=s_n(s_{n+1})$.
\end{lemma}
\begin{proof}
This follows since there is an $r \in RSp$ such that $d(r) = s_{n+1}+s_n-s_{n+1}'-s_n'$, but $s_n'$ is a lower order relator
\end{proof}

\begin{lemma}
\label{horiz comm}
If $s=(s_i)$ is a horizontal sequence, and $t$ is a relator which commutes with each relator in $s$, then the sequence $(t(s_i))$ on $r(w)$ is consistently reducible iff $s$ is consistently reducible. 
\end{lemma}
\begin{proof}
$s \sim (t,-t,s_1,...,s_n) \sim (-t,s_1,...,s_n,t) \sim (-t,t,t(s_1),...,t(s_n)) \sim (t(s_1),...,t(s_n))$. The reverse direction follows by applying $(-t)$ to $t(s)$
\end{proof}

We say a reduction sequence $s$ \tbf{decomposes into} reduction sequences $s'$ and $s''$ if $s \sim (s_0,...,s_{n+1})$, and $s'=(s_0,....,s_k,r)$ and $s''=(-r,s_{k+1}....,s_{n+1})$ for some reduction relator $r$.
\begin{fact}
If a reduction sequence $s$ decomposes into $s'$ and $s''$, where $s'$ and $s''$ reduce consistently, then $s$ also reduces consistently
\end{fact}
\begin{proof}
The corresponding lower order relation $r$ for $(s_0,...,s_{n+1})$ is just the sum of the relation $r$ for $s'$ and $r''$ for $s''$. Since $(s_0,...,s_{n+1}) \sim s$, we get that $s$ also reduces consistently.
\end{proof}

Finally, each web $w$ in $FSp$ has a planar drawing $D$ in $\R^2$. Given a web and a planar drawing, let $\ep=\ep_D \in \R_{>0}$ be the smaller of half the minimum distance between vertices of the web in $D$, and half the length of any edge. Given a subgraph of $w$, the \tbf{$\ep$-neighborhood} of the subgraph will be the smallest simply-connected region of $\R^2$ containing an $\ep$-ball around each point in $D$. 
\begin{remark}
We choose half the distance because we want to take any loop or edge not in $D$ but whose boundary vertices are in $D$, and cut them into two edges
\end{remark}

\chapter{Spiders of Quantum Groups}
\linespread{1}\selectfont
In this chapter we provide the definition of the various spiders which we'll be working with in this paper. The two sections are independent, so if you mostly care about the $A_3$ spider, then it is unnecessary to read about product spiders. In the section on the $A_3$ spider, we also prove some key results which will be used in the later chapters. In particular \cref{Global} is a global criterion for when an $A_3$ web is reducible. We now fix a ground field $K=\Q(q)$

\section{Semisimple Algebras and Product Spiders}
\subsection{The \texorpdfstring{$A_1$ and $A_2$}{and} Spiders}
The first two spiders we will mention are the $A_1$ and $A_2$ spiders. They have already been thoroughly studied, but we mention them here because we will use them to create product spiders which provide some relatively simple examples and counterexamples. 

The $A_1$ spider has a single self-dual generating object (this corresponds to the defining representation), and no generating morphisms, and hence the set of webs is graphically just the set of planar matchings (possibly containing closed loops). There is a single relator saying that we can remove a loop from a webs at the cost of multiplying the web by the quantum number $[2]_q=q+q^{-1}$, which we denote graphically by:

\begin{equation}
\label{A1:loop}
\begin{tikzpicture}[baseline=-.3ex,scale=.5] 
\draw[web] (0,0) circle (1cm);
\end{tikzpicture} \; =\; [2]_q
\end{equation}

This is just the Temperley-Lieb category (cf. \cite{kuperberg:rank}). 

As was mentioned in section \ref{PoTC:Intro}, the $A_2$ spider has two generating objects $V$ and $V^*$ (corresponding to the defining representation and its dual). In addition it has two generating morphisms $VVV$ and $V^*V^*V^*$ (corresponding to the contraction maps $V \ot V \xto{\wedge} \wedge^2 V \cong V^*$) which we draw graphically as: 

$$
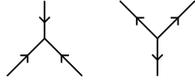
\begin{tikzpicture}[baseline=-2ex,scale=.5]
\draw[midto,web] (0,0) -- (1,1);
\draw[midto,web] (2,0) -- (1,1);
\draw[midto,web] (1,2) -- (1,1);

\draw[midto,web] (4,1) -- (4,0);
\draw[midto,web] (4,1) -- (5,2);
\draw[midto,web] (4,1) -- (3,2);
\end{tikzpicture}$$

Finally, there are three generating relators:

\begin{equation}
\begin{tikzpicture}[baseline=-.3ex,scale=.5] 
\draw[midto,web] (0,0) circle (1cm);
\end{tikzpicture} \; =\; [3]_q
\end{equation}
\begin{equation}
\begin{tikzpicture}[baseline=-.3ex,scale=.75] 
\draw[midfrom,web] (-1,0) -- (0,0);
\draw[web, midto]  (0,0) to[out=30,in=150] (2,0);
\draw[web, midto]  (0,0) to[out=-30,in=-150] (2,0);
\draw[midfrom,web] (2,0) -- (3,0);
\end{tikzpicture}\; =\; [2]_q
\begin{tikzpicture}[baseline=-.3ex,scale=.75] 
\draw[midto,web] (-1,0) -- (1,0);
\end{tikzpicture}
\end{equation}
\begin{equation}
\begin{tikzpicture}[baseline=-2ex,scale=.75]
\draw[midfrom,web] (-1.68,-1.5) -- (-1,-1);
\draw[midto,web] (-1.68,0.5) -- (-1,0);
\draw[midto,web] (0.68,-1.5) -- (0,-1);
\draw[midfrom,web] (0.68,0.5) -- (0,0);
\draw[midto,web] (-1,-1) -- (0,-1);
\draw[midto,web] (-1,-1) -- (-1,0);
\draw[midto,web] (0,0) -- (0,-1);
\draw[midto,web] (0,0) -- (-1,0);
\end{tikzpicture} = 
\begin{tikzpicture}[baseline=-2ex,scale=.55]
\draw [web, midfrom]  (-1.68,-1.5) to[out=60,in=-60] (-1.68,0.5);
\draw [web, midto]   (-.3,-1.5) to[out=120,in=240] (-.3,0.5);
\end{tikzpicture} +
\begin{tikzpicture}[baseline=-.5ex,scale=.65]
\draw [web, midfrom]  (-1.68,-.75) to[out=30,in=150] (-.3,-.75);
\draw [web, midto]   (-.3,0.5) to[out=210,in=-30] (-1.68,0.5);
\end{tikzpicture}
\end{equation}

Note that these two spiders have only reduction relators, so $E^1$ convergence is just usual confluence.

\subsection{Product Spiders}
If we have a semisimple but not simple algebra, there aren't enough relations between triples of fundamental representations to generate the category. This is because there are also morphisms of the form $A \ot B \natiso B \ot A$ if $A$ and $B$ are irreducibles for different simple components of our algebra, and these do not decompose into simpler morphisms as in case of the $A_3$ spider. Given two presentations of tensor categories with trivalent and/or tetravalent vertices as in chapter 3, $(\sV_1,\sT_1,D_1,R_1)$ and $(\sV_2,\sT_2,D_2,R_2)$, we define the product spider presentation as: $(\sV_1 \cup \sV_2,\sT_1\cup\sT_2\cup (\sV_1 \sV_2 D(\sV_1) D(\sV_2)), D=D_1\op D_2,R_1\cup R_2 \cup R_{prod})$ where $D_1 \op D_2(V)=D_i(V)$ for $V \in \sV_i$, and $R_{prod}$ is the set of diagrams of the form:
\begin{equation}
\label{product:bigon}
\begin{tikzpicture}[baseline=0ex,scale=.5]
\draw[red] (0,0) -- (3,0);
\draw[blue] (0.5,-1) to[out=90,in=180] (1.5,1);
\draw[blue] (1.5,1) to[out=0,in=90] (2.5,-1);
\end{tikzpicture} = 
\begin{tikzpicture}[baseline=0ex,scale=.5]
\draw[red] (0,0) -- (3,0);
\draw[blue] (.5,-1) to[out=30,in=150] (2.5,-1);
\end{tikzpicture}
\end{equation}
\begin{equation}
\label{product:Y}
\begin{tikzpicture}[baseline=2ex,scale=.5]
\draw[red] (0,0) -- (1,1);
\draw[red] (2,0) -- (1,1);
\draw[red] (1,2) -- (1,1);
\draw[blue] (0,.75) to[out=45,in=135] (2,.75);
\end{tikzpicture} = 
\begin{tikzpicture}[baseline=2ex,scale=.5]
\draw[red] (0,0) -- (1,1);
\draw[red] (2,0) -- (1,1);
\draw[red] (1,2) -- (1,1);
\draw[blue] (0,.75) to[out=-45,in=225] (2,.75);
\end{tikzpicture}
\end{equation}
\begin{equation}
\label{product:triangle}
\begin{tikzpicture}[baseline=1.5ex,scale=1] 
\draw[red] (0,0) -- (1,0); 
\draw[red] (1,0) -- (2,0);
\draw[red] (2,0) -- (3,0);

\draw[blue] (0.5,-.5) -- (1,0); 
\draw[blue] (1,0) -- (1.5,.5);
\draw[blue] (1.5,.5) -- (2,1);

\draw[green] (2.5,-.5) -- (2,0); 
\draw[green] (2,0) -- (1.5,.5);
\draw[green] (1.5,.5) -- (1,1);
\end{tikzpicture}
 \;=\;
\begin{tikzpicture}[baseline=-1ex,scale=1,transform shape, rotate=180] 
\draw[red] (0,0) -- (1,0); 
\draw[red] (1,0) -- (2,0);
\draw[red] (2,0) -- (3,0);

\draw[blue] (0.5,-.5) -- (1,0); 
\draw[blue] (1,0) -- (1.5,.5);
\draw[blue] (1.5,.5) -- (2,1);

\draw[green] (2.5,-.5) -- (2,0); 
\draw[green] (2,0) -- (1.5,.5);
\draw[green] (1.5,.5) -- (1,1);
\end{tikzpicture}
\end{equation}
\begin{equation}
\label{product:H}
\begin{tikzpicture}[baseline=-.3ex,scale=.6] 
\draw[red] (-1,-1) -- (0,0);
\draw[red] (-1,1) -- (0,0);
\draw[red] (0,0) -- (2,0);
\draw[red]  (2,0) -- (3,1);
\draw[red] (2,0) -- (3,-1);
\draw[blue] (-1,0) to[out=30,in=150] (3,0); 
\end{tikzpicture} =
\begin{tikzpicture}[baseline=-.3ex,scale=.6] 
\draw[red] (-1,-1) -- (0,0);
\draw[red] (-1,1) -- (0,0);
\draw[red] (0,0) -- (2,0);
\draw[red]  (2,0) -- (3,1);
\draw[red] (2,0) -- (3,-1);
\draw[blue] (-1,0) to[out=-30,in=210] (3,0); 
\end{tikzpicture} 
\end{equation}
\begin{remark}
It's possible to derive relator (\ref{product:H}) by applying relator (\ref{product:Y}) twice, however we will lose $E^1$ convergence of our spectral sequence.
\end{remark}
We will mainly be interested in the algebras $A_1^n \times A_2^m$ because both $A_1$ and $A_2$ have a unique choice of filtration, so we avoid many of the complications when dealing with $B_2$ or $G_2$ factors.

Following in part the terminology of \cite{kuperberg:rank}, we'll refer to (\ref{product:bigon}) as a $U$ relator/move, (\ref{product:Y}) as a $Y$ relator/move, and (\ref{product:H}) as an $H$ relator/move. We'll refer to (\ref{product:triangle}) as a triangle relator/move.

\section{The \texorpdfstring{$A_3$}{} Spider}
\subsection{Introduction}
Let us look at the spider for $A_3$. There are three generating objects: $V=V_{\om_1}$, $W=V_{\om_2}$, $V^*= V_{\om_3}$. As suggested by the notation $D(V)=V^*$ and $D(W)=W$. We will represent $V$ and $V^*$ as single directed edges, and $W$ as an undirected double edge.
$$
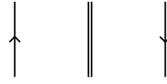
\begin{tikzpicture}[baseline=-2ex,scale=1]
\draw[midto,web] (0,0) -- (0,1);
\draw[double,web] (1,0) -- (1,1);
\draw[midfrom,web] (2,0) -- (2,1);
\end{tikzpicture}$$

The generating morphisms will correspond to $VVW$ and $WVV$ together with rotations:
$$
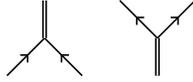
\begin{tikzpicture}[baseline=-2ex,scale=.5]
\draw[midto,web] (0,0) -- (1,1);
\draw[midto,web] (2,0) -- (1,1);
\draw[double,web] (1,1) -- (1,2);

\draw[double,web] (4,0) -- (4,1);
\draw[midto,web] (4,1) -- (5,2);
\draw[midto,web] (4,1) -- (3,2);
\end{tikzpicture}$$

We also have 8 relators up to duality and rotation corresponding to the 2 kinds of loops, the 2 kinds of bigons, 3 square relators, and 1 hexagon relator up to rotation and arrow reversal:
\begin{equation}
\label{Single Loop}
\begin{tikzpicture}[baseline=-.3ex,scale=.5] 
\draw[midto,web] (0,0) circle (1cm);
\end{tikzpicture} \;=\; \sqbinom{4}{1}_q
\end{equation}
\begin{equation}
\label{Double Loop}
\begin{tikzpicture}[baseline=-.3ex,scale=.5] 
\draw[double,web] (0,0) circle (1cm);
\end{tikzpicture} \;=\; \sqbinom{4}{2}_q
\end{equation}

\begin{equation}
\label{Pure Bigon}
\begin{tikzpicture}[baseline=-.3ex,scale=.75] 
\draw[double,web] (-1,0) -- (0,0);
\draw[web, midto]  (0,0) to[out=30,in=150] (2,0);
\draw[web, midto]  (0,0) to[out=-30,in=-150] (2,0);
\draw[double,web] (2,0) -- (3,0);
\end{tikzpicture}\; =\; [2]_q 
\begin{tikzpicture}[baseline=-.3ex,scale=.75] 
\draw[double,web] (-1,0) -- (1,0);
\end{tikzpicture}
\end{equation}

\begin{equation}
\label{Mixed Bigon}
\begin{tikzpicture}[baseline=-.3ex,scale=.75] 
\draw[midto,web] (-1,0) -- (0,0);
\draw[web, double]  (0,0) to[out=30,in=150] (2,0);
\draw[web, midfrom]  (0,0) to[out=-30,in=-150] (2,0);
\draw[midto,web] (2,0) -- (3,0);
\end{tikzpicture} \; =\;[3]_q 
\begin{tikzpicture}[baseline=-.3ex,scale=.75]
\draw[midto,web] (-1,0) -- (1,0);
\end{tikzpicture}
\end{equation}

\begin{equation}
\label{IH}
\begin{tikzpicture}[baseline=-.3ex,scale=.6] 
\draw[midto,web] (-1,-1) -- (0,0);
\draw[midto,web] (-1,1) -- (0,0);
\draw[double,web] (0,0) -- (2,0);
\draw[midfrom,web]  (2,0) -- (3,1);
\draw[midfrom,web] (2,0) -- (3,-1);
\end{tikzpicture}\; =\;
\begin{tikzpicture}[baseline=3ex,scale=.5] 
\draw[midto,web] (-1,-1) -- (0,0);
\draw[midto,web] (1,-1) -- (0,0);
\draw[double,web] (0,0) -- (0,2);
\draw[midfrom,web]  (0,2) -- (1,3);
\draw[midfrom,web] (0,2) -- (-1,3);
\end{tikzpicture}
\end{equation}

\begin{equation}
\label{Double Square}
\begin{tikzpicture}[baseline=-.3ex,scale=.5] 
\draw[double,web] (-1,-1) -- (1,-1);
\draw[midto,web] (-1,-1) -- (-1,1);
\draw[double,web] (-1,1)-- (1,1);
\draw[midfrom,web] (1,-1)-- (1,1);
\draw[midto,web] (-1,-1) -- (-2,-2);
\draw[midfrom,web] (-1,1) -- (-2,2);
\draw[midfrom,web] (1,-1) -- (2,-2);
\draw[midto,web] (1,1) -- (2,2);
\end{tikzpicture} \;=\; [2]_q
\begin{tikzpicture}[baseline=-4ex,scale=.75] 
\draw[web, midto]  (0,0) to[out=-30,in=-150] (2,0); 
\draw[web, midfrom]  (0,-1.5) to[out=30,in=150] (2,-1.5);
\end{tikzpicture}\;+\; 
\begin{tikzpicture}[baseline=4ex,scale=.75] 
\draw[web, midfrom]  (0,0) to[out=60,in=-60] (0,2); 
\draw[web, midto]  (1.5,0) to[out=120,in=-120] (1.5,2);
\end{tikzpicture}
\end{equation}

\begin{equation}
\label{Single Square}
\begin{tikzpicture}[baseline=-.3ex,scale=.5] 
\draw[double,web] (-1,-1) -- (1,-1);
\draw[midto,web] (-1,-1) -- (-1,1);
\draw[midfrom,web] (-1,1)-- (1,1);
\draw[midfrom,web] (1,-1)-- (1,1);
\draw[midto,web] (-1,-1) -- (-2,-2);
\draw[double,web] (-1,1) -- (-2,2);
\draw[midfrom,web] (1,-1) -- (2,-2);
\draw[double,web] (1,1) -- (2,2);
\end{tikzpicture} \;=\;
\begin{tikzpicture}[baseline=-.3ex,scale=.5] 
\draw[midfrom,web] (-1,-1) -- (0,0);
\draw[double,web] (-1,1) -- (0,0);
\draw[midto,web] (0,0) -- (2,0);
\draw[double,web]  (2,0) -- (3,1);
\draw[midfrom,web] (2,0) -- (3,-1);
\end{tikzpicture}\;+\; 
\begin{tikzpicture}[baseline=-4ex,scale=.75] 
\draw[web, double]  (0,0) to[out=-30,in=-150] (2,0); 
\draw[web, midfrom]  (0,-1.5) to[out=30,in=150] (2,-1.5);
\end{tikzpicture}
\end{equation}

\begin{equation}
\label{SS}
\begin{tikzpicture}[baseline=-.3ex,scale=.5] 
\draw[midto,web] (-1,-1) -- (1,-1);
\draw[midto,web] (-1,-1) -- (-1,1);
\draw[midfrom,web] (-1,1)-- (1,1);
\draw[midfrom,web] (1,-1)-- (1,1);
\draw[double,web] (-1,-1) -- (-2,-2);
\draw[double,web] (-1,1) -- (-2,2);
\draw[double,web] (1,-1) -- (2,-2);
\draw[double,web] (1,1) -- (2,2);
\end{tikzpicture} \;=\;
\begin{tikzpicture}[baseline=-.3ex,scale=.5]
\draw[midfrom,web] (-1,-1) -- (1,-1);
\draw[midfrom,web] (-1,-1) -- (-1,1);
\draw[midto,web] (-1,1)-- (1,1);
\draw[midto,web] (1,-1)-- (1,1);
\draw[double,web] (-1,-1) -- (-2,-2);
\draw[double,web] (-1,1) -- (-2,2);
\draw[double,web] (1,-1) -- (2,-2);
\draw[double,web] (1,1) -- (2,2);
\end{tikzpicture}
\end{equation}

\begin{equation}
\label{Kekule}
\begin{tikzpicture}[baseline=-.3ex,scale=.5] 
\draw[midto,web] (-1,-1) -- (-1.87,0); 
\draw[double,web] (-1.87,0)-- (-1,1);
\draw[midto,web] (-1,1)-- (.6,1);
\draw[double,web] (.6,1)--  (1.47 , 0);
\draw[midto,web]  (1.47 , 0)--  (.6 , -1);
\draw[double,web]  (.6 , -1)--  (-1,-1);

\draw[midto,web] (-1,-1) -- (-1.87,-2); 
\draw[midfrom,web] (-1.87,0)-- (-3.47,0);
\draw[midto,web] (-1,1)-- (-1.87,2);
\draw[midfrom,web] (.6,1)--  (1.47 , 2);
\draw[midto,web]  (1.47 , 0)--  (3.07 , 0);
\draw[midfrom,web]  (.6 , -1)--  (1.6,-1.87);
\end{tikzpicture} \; + \;
\begin{tikzpicture}[baseline=-.3ex,scale=.5]  
\draw[midfrom,web] (-0.87,-2) to[out=90,in=0] (-2.47,0);
\draw[midfrom,web] (-1.87,1.5) to[out=-30,in=210] (1.47 , 1.5);
\draw[midfrom,web]  (2.07 , 0) to[out=180,in=90]  (0.6,-1.87);
\end{tikzpicture}
 \;=\;
\begin{tikzpicture}[baseline=-.3ex,scale=.5] 
\draw[double,web] (-1,-1) -- (-1.87,0); 
\draw[midfrom,web] (-1.87,0)-- (-1,1);
\draw[double,web] (-1,1)-- (.6,1);
\draw[midfrom,web] (.6,1)--  (1.47 , 0);
\draw[double,web]  (1.47 , 0)--  (.6 , -1);
\draw[midfrom,web]  (.6 , -1)--  (-1,-1);

\draw[midto,web] (-1,-1) -- (-1.87,-2); 
\draw[midfrom,web] (-1.87,0)-- (-3.47,0);
\draw[midto,web] (-1,1)-- (-1.87,2);
\draw[midfrom,web] (.6,1)--  (1.47 , 2);
\draw[midto,web]  (1.47 , 0)--  (3.07 , 0);
\draw[midfrom,web]  (.6 , -1)--  (1.6,-1.87);
\end{tikzpicture}
\; + \;
\begin{tikzpicture}[baseline=-.3ex,scale=.5,transform shape, rotate=180]  
\draw[midto,web] (-0.87,-2) to[out=90,in=0] (-2.47,0);
\draw[midto,web] (-1.87,1.5) to[out=-30,in=210] (1.47 , 1.5);
\draw[midto,web]  (2.07 , 0) to[out=180,in=90]  (0.6,-1.87);
\end{tikzpicture}
\end{equation}
Where recall that $[n]_q=q^{n} +q^{n-2} +...+q^{2-n}+ q^{-n}$ are the quantum integers and $\sqbinom{n}{k}_q = \frac{[n]_q...[n+1-k]_q}{[k]_q...[1]_q}$ are the quantum binomial coefficients.
\begin{remark}
If you prefer minimal sets of relators, you could get by without the hexagon relator (\ref{Kekule}), two of the square relators (\ref{Double Square}) and (\ref{SS}), and one of the bigon relators (\ref{Pure Bigon}) but this will come at the cost of having only $E^3$ convergence
\end{remark}
However, this is not the presentation that will be easiest to work with. Instead, given a web $w$, we will define a new graph $S(w)$ by contracting all double edges that aren't on the boundary or part of a loop, so for example, a hexagon relator face will look like:
\begin{equation*}
w=\begin{tikzpicture}[baseline=-.3ex,scale=.5] 
\draw[double,web] (-1,-1) -- (-1.87,0); 
\draw[midfrom,web] (-1.87,0)-- (-1,1);
\draw[double,web] (-1,1)-- (.6,1);
\draw[midfrom,web] (.6,1)--  (1.47 , 0);
\draw[double,web]  (1.47 , 0)--  (.6 , -1);
\draw[midfrom,web]  (.6 , -1)--  (-1,-1);

\draw[midto,web] (-1,-1) -- (-1.87,-2); 
\draw[midfrom,web] (-1.87,0)-- (-3.47,0);
\draw[midto,web] (-1,1)-- (-1.87,2);
\draw[midfrom,web] (.6,1)--  (1.47 , 2);
\draw[midto,web]  (1.47 , 0)--  (3.07 , 0);
\draw[midfrom,web]  (.6 , -1)--  (1.6,-1.87);
\end{tikzpicture} \longrightarrow
S(w)=\begin{tikzpicture}[baseline=1.5ex,scale=1] 
\draw[midto,web] (0,0) -- (1,0); 
\draw[midto,web] (1,0) -- (2,0);
\draw[midto,web] (2,0) -- (3,0);

\draw[midfrom,web] (0.5,-.5) -- (1,0); 
\draw[midfrom,web] (1,0) -- (1.5,.5);
\draw[midfrom,web] (1.5,.5) -- (2,1);

\draw[midto,web] (2.5,-.5) -- (2,0); 
\draw[midto,web] (2,0) -- (1.5,.5);
\draw[midto,web] (1.5,.5) -- (1,1);
\end{tikzpicture}
\end{equation*}

Now, given $S(w)$, it is possible to reconstruct $w$ by using the directedness of the single edges, with one exception: the two faces of the $I=H$ relator (\ref{IH}) and its dual relator give the same contracted graph:
\begin{equation*}
\begin{tikzpicture}[baseline=-.3ex,scale=.6] 
\draw[midto,web] (-1,-1) -- (0,0);
\draw[midto,web] (-1,1) -- (0,0);
\draw[double,web] (0,0) -- (2,0);
\draw[midfrom,web]  (2,0) -- (3,1);
\draw[midfrom,web] (2,0) -- (3,-1);
\end{tikzpicture} \longrightarrow
\begin{tikzpicture}[baseline=-.3ex,scale=.6] 
\draw[midto,web] (-1,-1) -- (0,0);
\draw[midfrom,web] (0,0) -- (1,1);
\draw[midto,web] (1,-1) -- (0,0);
\draw[midfrom,web] (0,0) -- (-1,1);
\end{tikzpicture} \longleftarrow
\begin{tikzpicture}[baseline=3ex,scale=.5] 
\draw[midto,web] (-1,-1) -- (0,0);
\draw[midto,web] (1,-1) -- (0,0);
\draw[double,web] (0,0) -- (0,2);
\draw[midfrom,web]  (0,2) -- (1,3);
\draw[midfrom,web] (0,2) -- (-1,3);
\end{tikzpicture}
\end{equation*}

In some sense this a good thing, since those two webs are equal to each other in $Sp$. However we need to check that consistent reducibility of all canceling sequences on $S(w)$ will also imply consistent reducibility of all canceling sequences of $w$. We extend the map $S$ by linearity to any element $x \in FSp$. The relators we'll put on $S(w)$ will just be any relator of the form $S(w) \to w' \xto{r} r(w') \to S(r(w'))$ where $w'$ is any web which is equivalent to $w$ under the $I=H$ relators (\ref{IH}). We don't want to change the grading, so we give $S(w)$ the grading induced from $w$, so tetravalent vertices count for twice as much as (any remaining boundary) trivalent vertices. We need to be a little bit careful about how we define adjacency in the contracted web. We will say two faces $f$ and $f'$ of a contracted web $S(w)$ are \tbf{adjacent} if they are separated by an edge in every $w'$ such that $S(w')=S(w)$. In particular two faces separated by an $I=H$-relator-type tetravalent vertex won't be adjacent. This definition is made so that two relators commute if and only if they are not adjacent. 

Ultimately, we want to prove that if this presentation has $E^1$ convergence, then the original also has $E^1$ convergence. The following lemma will allow us to convert a sequence of relators in one presentation, to a sequence of relators in the other (this just takes the form of omitting $I=H$ relators, or adding needed ones back in).

\begin{lemma}
\label{Sequence equivalence}
If $s$ is a sequence of relators on $w$, there exists a corresponding sequence of relators $s'$ on $S(w)$ such that $S(\sum s) = \sum s'$, and conversely, for every sequence $s'$ on $S(w)$ there exists a sequence $s$ on $w'$ (for some $w'$ equivalent to $w$ under I=H relators) such that $S(\sum s) = \sum s'$
\end{lemma}
\begin{proof}
Let $s=(s_i)_{i=1}^n$ be a (not necessarily canceling) sequence of relators for $S(w)$, we claim that we can recursively construct a sequence of relators $s'=(s'_j)_{j=1}^m$ on some $w'$ which is equivalent to $w$ under the action of $I=H$ relators and such that $\sum s = S(\sum s')$. First, we know $s_1$ corresponds to a relator $r_1$ which can be applied to $w'$, so we define the partial sequence $s^1=(r_1)$. $S(\sum s^1) = s_1$ by definition.  Now, assume we've defined a sequence $s^k=(s^k_1,...,s^k_{n_k})$ such that $S(\sum s^k) = s_1+...+s_k$. $s_{k+1}$ corresponds to a relator $r_{k+1}$ which can be applied to $w_k'=\rho_{n_{k+1}} ... \rho_1 (s^k(w))$ for some I=H relators $\rho_i$, and define $s^{k+1}$ as the concatenation of $s^k$ and $(\rho_1,...,\rho_{n_{k+1}},r_{k+1})$. Since $S(\rho_1+...\rho_{n_{k+1}})=S(w'_k-w_k)=0$ by construction, and $S(r_{k+1})=s_{k+1}$ since they only differ by $I=H$ relators, this gives us $S(\sum s^{k+1})=S(r_{k+1})+S(\sum s^k)=s_{k+1}+s_k+...+s_1$. Therefore $s^n$ fits our requirements for $s'$

Let $s=(s_i)$ be a (not necessarily canceling) sequence of relators for $w$, we claim that we can construct a sequence of relators $s'=(s'_i)$ on $S(w)$ such that $S(\sum s) = \sum s'$. This direction is even more straightforward: we first take the subsequence $(s_{i_k})$ of relators which aren't I=H relators, then we claim that $s'=(S(s_{i_k}))$ is a sequence of relators and satisfies our criteria. Indeed, $s_{i_{k-1}}....s_2s_1(w)$ and $s_{i_{k}-1}....s_2s_1(w)$ differ only by the $I=H$ relators for every $k$, so $S(s_{i_{k-1}}....s_1(w))=S(s_{i_{k}-1}....s_1(w))$, and hence $S(s_{i_k})$ is a relator on $S(s_{i_{k-1}}....s_1(w))$, which means $(S(s_{i_k}))$ is actually a sequence of relators. Moreover, $S(r)=0$ for every I=H relator, so omitting the I=H relators doesn't change the sums after applying $S$.
\end{proof}

Next we use this lemma to prove that it is enough to prove $E^1$ convergence for $S(w)$ 
\begin{prop}
\label{excise}
If every canceling sequence of $S(w)$ is consistently reducible, then every canceling sequence of $w$ is consistently reducible
\end{prop}
\begin{proof}
Let $s$ be a canceling sequence on $w$. By part 1 of lemma \ref{Sequence equivalence}, there exists a corresponding canceling sequence $s'$ on $S(w)$. Now by hypothesis, there exists a relation $r'$ such that $|r'|<|s'|$ and $d(r') = d(\sum s')$. Now take any lift $r$ of $r'$ to a relation on $w$, such that $|r|=|r'|$. Then $d(r)=d(r')=d(\sum s)$ up to applications of I=H relators (since the kernel of $S$ is generated by I=H relators). In other words, $r-\sum s = \sum_i a_i r_i$ for some $I=H$ relators $r_i$. Now, since $I=H$ relators have no lower order terms, if $|r_k|>|\sum_i a_i r_i|$ for some $k$, then all the relators $r_i$ with $|r_i|=|r_k|$ must add to zero in the partial sum. Therefore, we can assume omit any such terms, and assume $|r_i| \leq |r-\sum s| < |s|$ and so $r-\sum_i a_i r_i$ is a relation that makes $s$ consistently reducible.
\end{proof}

We want to make one additional simplification. The double edges at the boundary will force us to list different relators which are the same in the original spider, but whose boundary edges contain different quantities of double edges. To avoid this, we will embed any web with double edges on the boundary into a web with only single edges on the boundary, which we construct as follows: if $w$ is an (uncontracted) web, let $w'$ be any web containing $w$ obtained by adding a trivalent vertex to every boundary double edge. Similar to before, we need to prove that it suffices to show $E^1$ convergence in $S(w')$:

\begin{prop}
\label{extend strand}
If every canceling sequence of the contracted web $S(w')$ (constructed above) is consistently reducible, then every canceling sequence of $S(w)$ is consistently reducible
\end{prop}
\begin{proof}
The key point is that in the contracted web $S(w')$ we haven't introduced any new relators (although we might have created some I=H relators in the uncontracted web $w'$). Let $s=(s_i)$ be a canceling sequence on $w$, then since $w\sbs w'$, there is a corresponding canceling sequence $s'=(s_i')$ on $w'$. It's consistently reducible by hypothesis, so there is an $r' \in RSp'$ such that $d(r') = d(\sum s')$ but $|r'| < |s'|$. Now, every relator in the contracted spider corresponds to an interior face, but the single edges added to $w'$ are adjacent only to boundary faces, and hence every relator in the contracted spider is disjoint from the new single edges. Therefore, it corresponds to a relation $r$ on $S(w)$ with the properties $d(r) = d(\sum s)$ and $|r| = |r'|-k < |s'| -k = |s|$ where $k$ is the number of boundary double edges in $w$.
\end{proof}

From this proposition it is enough to prove consistent reducibility on webs with only single boundary edges. Hence, we have the following derived relators (together with the two loop relators):
\begin{equation}
\label{Pure Bigon C}
\begin{tikzpicture}[baseline=-.3ex,scale=.75] 
\draw[web] (-1,.5) -- (0,0);
\draw[web] (-1,-.5) -- (0,0);
\draw[web, midto]  (0,0) to[out=30,in=150] (2,0);
\draw[web, midto]  (0,0) to[out=-30,in=-150] (2,0);
\draw[web] (2,0) -- (3,.5);
\draw[web] (2,0) -- (3,-.5);
\end{tikzpicture}\; =\; [2]_q 
\begin{tikzpicture}[baseline=-.3ex,scale=.6] 
\draw[web] (-1,-1) -- (0,0);
\draw[web] (0,0) -- (1,1);
\draw[web] (1,-1) -- (0,0);
\draw[web] (0,0) -- (-1,1);
\end{tikzpicture}
\end{equation}
\begin{equation}
\label{Mixed Bigon C}
\begin{tikzpicture}[baseline=-.3ex,scale=.5] 
\draw[midfrom,web] (0,0) circle (1cm);
\draw[midto,web] (-2,1) -- (0,1);
\draw[midto,web] (0,1) -- (2,1);
\end{tikzpicture} \; =\;[3]_q 
\begin{tikzpicture}[baseline=-.3ex,scale=.75]
\draw[midto,web] (-1,0) -- (1,0);
\end{tikzpicture}
\end{equation}
\begin{equation}
\label{Double Square C}
\begin{tikzpicture}[baseline=-.3ex,scale=.75] 
\draw[midto,web] (-1,.5) -- (0,0);
\draw[midfrom,web] (-1,-.5) -- (0,0);

\draw[web, midfrom]  (0,0) to[out=30,in=150] (2,0);
\draw[web, midto]  (0,0) to[out=-30,in=-150] (2,0);

\draw[midto,web] (2,0) -- (3,.5);
\draw[midfrom,web] (2,0) -- (3,-.5);
\end{tikzpicture}\;=\; [2]_q
\begin{tikzpicture}[baseline=4ex,scale=.75] 
\draw[web, midfrom]  (0,0) to[out=60,in=-60] (0,2); 
\draw[web, midto]  (1.5,0) to[out=120,in=-120] (1.5,2);
\end{tikzpicture}\;+\; 
\begin{tikzpicture}[baseline=-4ex,scale=.75] 
\draw[web, midto]  (0,0) to[out=-30,in=-150] (2,0); 
\draw[web, midfrom]  (0,-1.5) to[out=30,in=150] (2,-1.5);
\end{tikzpicture}
\end{equation}

\begin{equation}
\label{Single Square C}
\begin{tikzpicture}[baseline=-1ex,scale=1,transform shape, rotate=180] 
\draw[web] (0,0) -- (1,0); 
\draw[midto,web] (1,0) -- (2,0);
\draw[web] (2,0) -- (3,0);

\draw[web] (0.5,-.5) -- (1,0); 
\draw[midto,web] (1,0) -- (1.5,.5);
\draw[midto,web] (1.5,.5) -- (2,1);

\draw[web] (2.5,-.5) -- (2,0); 
\draw[midfrom,web] (2,0) -- (1.5,.5);
\draw[midfrom,web] (1.5,.5) -- (1,1);
\end{tikzpicture}\;=\;
\begin{tikzpicture}[baseline=-.3ex,scale=.5] 
\draw[midfrom,web] (0,-1) -- (0,0);
\draw[web] (0,0) -- (-1,0);
\draw[web] (0,0) -- (0,1);
\draw[midto,web] (0,0) -- (2,0);
\draw[web] (2,0) -- (3,0);
\draw[web] (2,0) -- (2,1);
\draw[midfrom,web] (2,0) -- (2,-1);
\end{tikzpicture}\;+\; 
\begin{tikzpicture}[baseline=-2ex,scale=.75] 
\draw[web] (-1,-.5) -- (0,0);
\draw[web] (0,0) -- (1,.5);
\draw[web] (1,-.5) -- (0,0);
\draw[web] (0,0) -- (-1,.5);
\draw[web, midfrom]  (-1,-1) to[out=30,in=150] (1,-1);
\end{tikzpicture}
\end{equation}
\begin{equation}
\label{SS C}
\begin{tikzpicture}[baseline=-.3ex,scale=.5] 
\draw[midto,web] (-1,-1) -- (1,-1);
\draw[midto,web] (-1,-1) -- (-1,1);
\draw[midfrom,web] (-1,1)-- (1,1);
\draw[midfrom,web] (1,-1)-- (1,1);
\draw[web] (-1,-1) -- (-1,-2);
\draw[web] (-1,-1) -- (-2,-1);
\draw[web] (-1,1)-- (-2,1);
\draw[web] (-1,1)-- (-1,2);
\draw[web] (1,-1) -- (1,-2);
\draw[web] (1,-1) -- (2,-1);
\draw[web] (1,1)-- (2,1);
\draw[web] (1,1)-- (1,2);
\end{tikzpicture} \;=\;
\begin{tikzpicture}[baseline=-.3ex,scale=.5]
\draw[midfrom,web] (-1,-1) -- (1,-1);
\draw[midfrom,web] (-1,-1) -- (-1,1);
\draw[midto,web] (-1,1)-- (1,1);
\draw[midto,web] (1,-1)-- (1,1);
\draw[web] (-1,-1) -- (-1,-2);
\draw[web] (-1,-1) -- (-2,-1);
\draw[web] (-1,1)-- (-2,1);
\draw[web] (-1,1)-- (-1,2);
\draw[web] (1,-1) -- (1,-2);
\draw[web] (1,-1) -- (2,-1);
\draw[web] (1,1)-- (2,1);
\draw[web] (1,1)-- (1,2);
\end{tikzpicture}
\end{equation}

\begin{equation}
\label{Kekule C}
\begin{tikzpicture}[baseline=1.5ex,scale=1] 
\draw[midto,web] (0,0) -- (1,0); 
\draw[midto,web] (1,0) -- (2,0);
\draw[midto,web] (2,0) -- (3,0);

\draw[midfrom,web] (0.5,-.5) -- (1,0); 
\draw[midfrom,web] (1,0) -- (1.5,.5);
\draw[midfrom,web] (1.5,.5) -- (2,1);

\draw[midto,web] (2.5,-.5) -- (2,0); 
\draw[midto,web] (2,0) -- (1.5,.5);
\draw[midto,web] (1.5,.5) -- (1,1);
\end{tikzpicture} \; + \;
\begin{tikzpicture}[baseline=-.3ex,scale=.5,transform shape, rotate=180]  
\draw[midto,web] (-0.87,-2) to[out=90,in=0] (-2.47,0);
\draw[midto,web] (-1.87,1.5) to[out=-30,in=210] (1.47 , 1.5);
\draw[midto,web]  (2.07 , 0) to[out=180,in=90]  (0.6,-1.87);
\end{tikzpicture}
 \;=\;
\begin{tikzpicture}[baseline=-1ex,scale=1,transform shape, rotate=180] 
\draw[midfrom,web] (0,0) -- (1,0); 
\draw[midfrom,web] (1,0) -- (2,0);
\draw[midfrom,web] (2,0) -- (3,0);

\draw[midto,web] (0.5,-.5) -- (1,0); 
\draw[midto,web] (1,0) -- (1.5,.5);
\draw[midto,web] (1.5,.5) -- (2,1);

\draw[midfrom,web] (2.5,-.5) -- (2,0); 
\draw[midfrom,web] (2,0) -- (1.5,.5);
\draw[midfrom,web] (1.5,.5) -- (1,1);
\end{tikzpicture}
\; + \;
\begin{tikzpicture}[baseline=-.3ex,scale=.5]  
\draw[midfrom,web] (-0.87,-2) to[out=90,in=0] (-2.47,0);
\draw[midfrom,web] (-1.87,1.5) to[out=-30,in=210] (1.47 , 1.5);
\draw[midfrom,web]  (2.07 , 0) to[out=180,in=90]  (0.6,-1.87);
\end{tikzpicture}
\end{equation}

Where the unoriented edges correspond to more than one choices of orientation. In particular, notice that up to lower order terms, relators (\ref{Pure Bigon C}), (\ref{Mixed Bigon C}), (\ref{Double Square C}), (\ref{Single Square C}) and (\ref{Kekule C}) can all be interpreted as homotopies of paths, while (\ref{SS C}) only changes orientations. This will be exploited in the following section to construct a global criterion for reducibility.

\subsection{Strands and a Global Criterion for Reducibility}
Our goal in this section is to develop a tool which will interpret the horizontal relators for $A_3$ in terms of homotopies of paths, and give a global criterion for when a web is reducible. From now on, we will abuse notation and refer to $S(w)$ as $w$, as defined in the previous section. If it's not clear whether we're discussing the contracted or uncontracted web we will explicitly state which we are considering.

\begin{defn}
Given a spider presentation with tetravalent vertices, a \tbf{(directed) strand} is a directed path in the web such that any two edges incident at a tetravalent vertex in the path are opposite each other (as depicted below as the dashed line) while an \tbf{undirected strand} is defined the same but without the requirement that the path is directed
\end{defn}
\begin{equation}
\begin{tikzpicture}[baseline=-.3ex,scale=.6] 
\draw[midto,web] (-1,-1) -- (0,0);
\draw[midto,web] (0,0) -- (1,1);
\draw[midto,web] (1,-1) -- (0,0);
\draw[midto,web] (0,0) -- (-1,1);
\draw[dashed, blue] (-1.1,-1) -- (0.9,1);
\end{tikzpicture}
\end{equation}
\begin{remark}
Beyond the suggestiveness of the presentation, strands also have a geometric motivation as (weak) geodesic strips in the building. This will be discussed in length in section 11.
\end{remark}

Notice that every single edge in the $A_3$ is part of a unique strand. We want to argue that relations are ultimately about homotopies of strands. However, since the square relator (\ref{SS}) can't be interpreted this was, we need to deal with that case first. We will prove that the horizontal square relator is unnecessary for creating a reducible face, and hence we will be able to construct a criterion for reducibility using strands alone. First we will need the following two lemmas about the $A_3$ spider which together tell us that we get a reducible face when two different types of relators collide, and that that is the only way a new reducible face will appear.

\begin{lemma}[Purity]
\label{collision}
If a triangle relator face (\ref{Kekule C}) and a horizontal square relator face are adjacent (in the sense of section $6$, ie when corresponding relators don't commute), then applying either relator will result in a reducible face.
\end{lemma}
\begin{proof}
It's not hard to see from orientations that such an adjacency must take the form of:
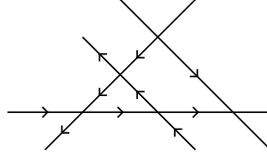
\begin{equation*}
\begin{tikzpicture}[baseline=1.5ex,scale=1] 
\draw[midto,web] (0,0) -- (1,0); 
\draw[midto,web] (1,0) -- (2,0);
\draw[midto,web] (2,0) -- (3,0);

\draw[midfrom,web] (0.5,-.5) -- (1,0); 
\draw[midfrom,web] (1,0) -- (1.5,.5);
\draw[midfrom,web] (1.5,.5) -- (2,1);

\draw[midto,web] (2.5,-.5) -- (2,0); 
\draw[midto,web] (2,0) -- (1.5,.5);
\draw[midto,web] (1.5,.5) -- (1,1);

\draw[midfrom,web] (3,0) -- (2,1);
\draw[web] (3,0) -- (3.5,-.5); 
\draw[web] (2,1) -- (1.5,1.5);
\draw[web] (2,1) -- (2.5,1.5);
\draw[web] (3,0) -- (3.5,0);
\end{tikzpicture} 
\end{equation*}
where it's possible that the opposite edge of the square relator is the same edge as one of the inner triangle edges, as in:
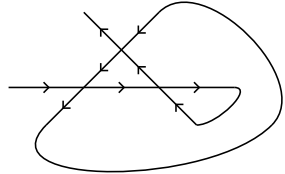
\begin{equation*}
\begin{tikzpicture}[baseline=1.5ex,scale=1] 
\draw[midto,web] (0,0) -- (1,0); 
\draw[midto,web] (1,0) -- (2,0);
\draw[midto,web] (2,0) -- (3,0);

\draw[midfrom,web] (0.5,-.5) -- (1,0); 
\draw[midfrom,web] (1,0) -- (1.5,.5);
\draw[midfrom,web] (1.5,.5) -- (2,1);

\draw[midto,web] (2.5,-.5) -- (2,0); 
\draw[midto,web] (2,0) -- (1.5,.5);
\draw[midto,web] (1.5,.5) -- (1,1);

\draw[web] (3,0) to[out=0,in=0] (2.5,-.5);
\draw[web] (2,1) to[out=45,in=45] (3.5,-.5);
\draw[web] (3.5,-.5) to[out=225,in=225] (0.5,-.5);
\end{tikzpicture} 
\end{equation*}
Now it's just a matter of checking the four cases: applying the square relator to the top makes a triangle reduction relator (\ref{Single Square C}), applying it to the bottom also makes a triangle reduction relator (\ref{Single Square C}), applying the triangle relator on the top makes a triangle reduction relator (\ref{Single Square C}), and finally applying a triangle relator to the bottom makes a new bigon relator (\ref{Double Square C}).
\end{proof}

\begin{lemma}
\label{purity lemma}
Let $w$ be an $A_3$ web such that $r(w)$ has a reducible face for some relator $r$, but $w$ does not. In $w$ there must be a horizontal square relator face (\ref{SS C}) which is adjacent to a triangle relator face (\ref{Kekule C}), and moreover $r$ corresponds to one of those two faces.
\end{lemma}
\begin{remark}
In particular, applying the relator corresponding to the second face will also create a reducible face by lemma \ref{collision}.
\end{remark}
\begin{proof}
First, let $r$ be a triangle relator. The reducible faces in the \tit{uncontracted} web all have four or fewer sides. Since the uncontracted hexagon relator doesn't change the size of any polygon, and no $I=H$ relator can be adjacent to a hexagon relator, the triangle relator must be adjacent to what was a square in the uncontracted web, however the only non-reducible squares in the uncontracted web are square relators. Therefore, since the adjacency must be by a single edge (since a horizontal square relator face has no interior double edges), the corresponding triangle is also adjacent to the square in the contracted web

On the other hand, if $r$ is a square relator, we see that it doesn't change the adjacencies in the \tit{contracted} web, and all the reducible faces in the web are triangles or smaller, so it must be adjacent to a triangle. The only type of non-reducible triangle is the horizontal relator face, which gives us our result.
\end{proof}

These lemmas give use the following proposition which essentially says that we only need to use one kind of relator in order to create a reducible face. This will help to reduce the number of possibilities we will need to deal with when proving $E^1$ convergence.

\begin{prop}
\label{A3 Purity}
If $w$ is a reducible $A_3$ spider, then there is a minimal sequence of relators $(r_i)$ with no horizontal square relators such that $r_n...r_1(w)$ has a reducible face. Similarly, there also exists a minimal sequence of relators with no horizontal triangle relators.
\end{prop}
\begin{proof}
We can find a minimal sequence of relators $(r_1,...,r_n)$ which gives $w$ a reducible face. If there are no horizontal square relators, we're done. Otherwise, there is a last horizontal square relator $r_k$. If $k<n$, and the triangle relator $r_{k+1}$ doesn't commute with $r_k$, then applying either relator will give a reduction face by lemma \ref{collision}. Hence there is a reducible face on $r_{k-1}....r_1(w)$, which contradicts minimality of the sequence. Hence we can commute the square relator $r_k$ to happen last. Now, since $r_nr_{n-1}....r_1(w)$ has a reducible face,  by lemma \ref{purity lemma} we know that there is a triangle relator face adjacent to the square face in $r_{n-1}....r_1(w)$, and applying that relator in place of the square relator, we get a sequence with one fewer square relator which ends in a web with a reducible face. Applying this process repeatedly we can remove all square relators from our minimal sequence.

The argument for removing all horizontal triangle relators follows identically.
\end{proof}

Now that know that we only need to use a single type of relator to show reducibility, and that relator corresponds to homotopies of strands, we will get a global criteria for when a web is reducible determined by the placement of strands. We will call any (topological) graph minor composed of strands a \tbf{configuration of strands}. If this graph minor corresponds to the face of a reduction relator, we will call it a \tbf{reducible configurations of strands}. We'll say a relator $r$ on $w$ doesn't \tbf{affect a configuration of strands} or doesn't \tbf{destroy a configuration of strands} if there exists an isomorphic (as a graph) configuration of strands in $r(w)$ which is equal to the original configuration outside an $\ep$-neighborhood of the relator faces.

Before we prove the global reducibility result, we will need to prove a lemma which shows that strands can only cross reducible configurations in very controlled ways:
\begin{lemma}[Directionality]
\label{A3 Directionality}
Let $w$ be an $A_3$ web, and take a reducible configuration of strands which is minimal in the sense that the bounded region in $w$ doesn't contain a reducible configuration. Then every strand that crosses this configuration must pass through the same two boundary edges, and can't intersect any other strand in the bounded region.
\end{lemma}
\begin{proof}
This just requires checking each of the reduction relators. If a strand crosses a loop or a monogon, then it forms a smaller bigon, so the loop or monogon wouldn't be minimal.  If a strand crosses a bigon relator of type (\ref{Pure Bigon C}), then it forms a reducible triangle configuration on either side. Strands can cross a bigon relator of type (\ref{Mixed Bigon C}), but only in one direction so as to not form a reducible triangle configuration. Therefore, if any two of these strands cross in the middle of the bigon relator, they will form a reducible triangle configuration with the side of the bigon. Similarly, if any strand crosses a reducible triangle relator configuration, and it cuts off one of the corners directed all inward or outward, then that strand together with the corner will form a reducible triangle configuration. Hence, all of the crossing strands must pass through the two sides of the triangle which cross cyclically, and the crossing strands must be directed in the same direction. Moreover, they can't intersect in the interior without creating a smaller reducible triangle configuration with one of the sides of the triangle.
\end{proof}

Finally, we can prove our global criterion for when a web is reducible.

\begin{thm} 
\label{Global}
A (contracted) $A_3$ web $w$ is reducible if and only if it has a reducible configuration of strands, ie if and only if it has a graph minor of one of the following forms:
\begin{enumerate}
	\item A closed (directed) strand (in the shape of \ref{Single Loop} or \ref{Double Loop})
	\item A (directed) strand that intersects itself (in the shape of \ref{Mixed Bigon C})
	\item Two (directed) strands that intersect more than once (in the shape of \ref{Pure Bigon C} or \ref{Double Square C})
	\item Three (directed) strands that pairwise intersect, but whose corresponding triangle isn't oriented cyclically (in the shape of \ref{Single Square C})
\end{enumerate}
Moreover, if there exists a configuration of undirected strands of any of these types, it is also reducible (in particular, a reducible triangle configuration can be oriented in any way besides a cyclic orientation).
\end{thm}
\begin{remark}[1] As a corollary of this, we know that applying square relators respects this criterion, however, it isn't true that applying a square relator can't affect a reducible configuration. Rather, this theorem shows that if a square relator does destroy a reducible configuration, there must still be a reducible configuration in the new web.
\end{remark}
\begin{remark}[2]
The undirected and directed versions will be used for two different purposes. We'll use the undirected version to show that certain webs (such as closed webs) are reducible. We'll apply the directed version to cases where we know the web is reducible. In those cases, knowing that the strands are directed will be useful.
\end{remark}
\begin{proof}
($\Rightarrow$) First we prove that every reducible web has a reducible configuration of strands. Since every reducible web can be given a reducible face via the (horizontal) triangle relator (\ref{Kekule C}) by proposition \ref{A3 Purity}, and every reducible face corresponds to a reducible configuration of strands, it suffices to prove that if you have one of these reducible configurations of strands and apply a triangle relator, the resulting web will still have a reducible configuration of strands, and the same orientation. Since horizontal triangle relators don't change which maximal directed strands intersect or how many times, so it's clear they can't remove loops and we immediately get that if there is a strand which crosses itself, or two strands which cross twice, then this will remain true after applying the triangle relator.

For the fourth criterion, we know that after applying the relator we still have three mutually intersecting strands, and since the orientation of this triangle is determined by its unchanged boundary, it must still be non-cyclically oriented. Hence, we still have a reducible configuration of strands.

($\Leftarrow$) Assume that the $A_3$ has at least one (directed) reducible configuration, and take a minimal one (meaning there are no reducible configurations contained in the bounded region). By lemma \ref{A3 Directionality}, we know all of the strands that cross the reducible configuration are parallel. Moreover, since all the reduction relators have at most $3$ sides, the crossing strands must make a triangle with one of the corners. Hence, we can repeatedly apply these triangle relators. By minimality, they are all non-reducible triangle relators, so after applying all these relators the reducible configuration corresponds to a single face, and we can applying the reduction relator.

For the final statement, we want to prove that if there are undirected strands in one of reducible configurations, then it must not be minimal, and so by induction there must be a (directed) reducible configuration contained in it. Clearly, any strand crossing a loop or a monogon has to create a smaller bigon (or self-intersect inside to create an smaller monogon). In the case of the bigon with non-cyclically oriented corners, any strand crossing it must create a non-cyclically oriented triangle with that corner. If a strand crosses a bigon with both corners oriented cyclically, but changes the orientation of one of the bigon strands, then the corresponding triangle with that corner is likewise non-cyclically oriented as depicted below:
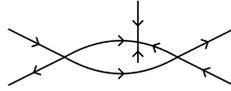
\begin{equation*}
\begin{tikzpicture}[baseline=-.3ex,scale=.75] 
\draw[midto,web] (-1,.5) -- (0,0);
\draw[midfrom,web] (-1,-.5) -- (0,0);

\draw[web, midto]  (0,0) to[out=30,in=150] (2,0);
\draw[web, midto]  (0,0) to[out=-30,in=-150] (2,0);

\draw[midto,web] (2,0) -- (3,.5);
\draw[midfrom,web] (2,0) -- (3,-.5);

\draw[midto,web] (1.3,1) -- (1.3,.21);
\draw[midfrom,web] (1.3,.21) -- (1.3,-.1);
\draw[midto,web] (1.62,0.18) -- (1.6,0.19);
\end{tikzpicture}
\end{equation*}
Lastly, in the case of a triangle, we know that there are some number of strands crossing it (otherwise it would be directed), so each must cut off a corner to form a new triangle configuration (again barring the case of self-intersection, or a strand leaving the same way it came and creating a bigon). If any one of them cuts off a non-cyclically oriented corner, then this smaller triangle is reducible. If all of them cut off a cyclically oriented corner, then at least one of them must intersect a side in a vertex where all orientations point inward or all point outward (otherwise the triangle would be made out of directed strands). Then, this says that the triangle cut off by the strand has a non-cyclically oriented corner, and hence it isn't cyclically oriented, giving us our result.
\end{proof}

\chapter{Proving \texorpdfstring{$E^k$}{} Convergence}
\linespread{1}\selectfont
This is the technical core of this dissertation where we prove our main results regarding $E^1$ and $E^2$ convergence. Just as the last chapter, the section on product spiders can be read independently of the section on the $A_3$ spider. However, the subsections in each section are not independent. In particular we will use the results about the $A_1^n$ and $A_1 \times A_2$ spiders in the proof of $E^2$ convergence for the $A_1^n \times A_2$ spider.

\section{Product Spiders Using Equinumeration}
\subsection{\texorpdfstring{$E^1$ Convergence for $A_1^n$}{convergence for}}
Webs of type $A_1^n$ are given by drawings of graphs corresponding to binary matchings, where we can only match boundary vertices corresponding to the same type. Since there are no trivalent vertices, we only get $U$ relators (of type \ref{product:bigon}) and triangle relators (of type \ref{product:triangle}), along with the inherited loop relators (of type \ref{A1:loop}) (which we copy here)
\begin{equation*}
\begin{tikzpicture}[baseline=0ex,scale=.5]
\draw[red] (0,0) -- (3,0);
\draw[blue] (0.5,-1) to[out=90,in=180] (1.5,1);
\draw[blue] (1.5,1) to[out=0,in=90] (2.5,-1);
\end{tikzpicture} = 
\begin{tikzpicture}[baseline=0ex,scale=.5]
\draw[red] (0,0) -- (3,0);
\draw[blue] (.5,-1) to[out=30,in=150] (2.5,-1);
\end{tikzpicture}
\end{equation*}
\begin{equation*}
\begin{tikzpicture}[baseline=1.5ex,scale=1] 
\draw[red] (0,0) -- (1,0); 
\draw[red] (1,0) -- (2,0);
\draw[red] (2,0) -- (3,0);

\draw[blue] (0.5,-.5) -- (1,0); 
\draw[blue] (1,0) -- (1.5,.5);
\draw[blue] (1.5,.5) -- (2,1);

\draw[green] (2.5,-.5) -- (2,0); 
\draw[green] (2,0) -- (1.5,.5);
\draw[green] (1.5,.5) -- (1,1);
\end{tikzpicture}
 \;=\;
\begin{tikzpicture}[baseline=-1ex,scale=1,transform shape, rotate=180] 
\draw[red] (0,0) -- (1,0); 
\draw[red] (1,0) -- (2,0);
\draw[red] (2,0) -- (3,0);

\draw[blue] (0.5,-.5) -- (1,0); 
\draw[blue] (1,0) -- (1.5,.5);
\draw[blue] (1.5,.5) -- (2,1);

\draw[green] (2.5,-.5) -- (2,0); 
\draw[green] (2,0) -- (1.5,.5);
\draw[green] (1.5,.5) -- (1,1);
\end{tikzpicture}
\end{equation*}
\begin{equation*}
\begin{tikzpicture}[baseline=-.3ex,scale=.5] 
\draw[web, blue] (0,0) circle (1cm);
\end{tikzpicture} \; =\; [2]_q
\end{equation*}
It is also not too hard to show that two webs give the same invariant vectors if and only if the corresponding matchings are the same. Additionally, we can see that a web does not have a minimal number of faces if and only if two strands cross each other more than once or there is a loop. Because of these two facts, it's relatively easy to show that the $A_1^n$ spider has an $E^1$ convergent spectral sequence without using the generalized diamond lemma (ie proposition \ref{Canceling}). 

\begin{thm}
\label{E1:A1n}
The $A_1^n$ spider with the product presentation has an $E^1$ convergent spectral sequence
\end{thm}
\begin{proof}
To prove this, we will show that: 
\begin{enumerate}
	\item Any web that has a closed loop or two strands that cross more than once can be reduced using the relators
	\item If two webs without closed loops or double-crossing strands correspond to the same matching of boundary vertices, then there exists a sequence of horizontal relators  sending one of the webs to the other
\end{enumerate}
Indeed, by the representation theory we know that the dimension of the invariant spaces are the product of the dimensions of the invariant spaces of each $A_1$ factor, which in turn is equal to the number of planar matchings. Therefore, as long as every choice of planar matchings for each $A_1$ factor has a corresponding minimal vertex web, and any two minimal vertex webs with the same $A_1$ planar matching are the same, we get that the dimension of Hom spaces of $gr_* Sp$ and $(gr_* FSp)/(\lan gr_* R\ran)$ are equal, which proves $E^1$ convergence. We now precede to prove the two claims:

1. First, let's look at the case where there is a minimal reducible configuration that is a loop. If there are no strands crossing it, then we have an immediate reduction. If a strand crosses it, then there is a smaller bigon, which contradicts our assumption of minimality. Hence, we can assume that the minimal configuration is a bigon. Now induct on the number of strands crossing the bigon. Any crossing strand creates a triangle with the two vertices of the bigon (since the bigon is minimal). We will use the following lemma which allows us to remove strands from any such a triangle 
\begin{lemma}
\label{A1n:Evacuate}
Given three pairwise intersecting strands in a $A_1^n$ web such that the bounded region is reduced, there exists a sequence of triangle relators (in an $\ep$-neighborhood of the bounded region) which removes all crossing strands.
\end{lemma}
\begin{proof}
We induct on the number of crossings in the bounded region. The configuration has no crossings strands exactly when the bounded region has 3 crossings. Any crossing strand makes a triangle with one of the corners of the outer triangle, so we can apply the inductive hypothesis to remove the strands from the interior triangle. This doesn't increase the number of crossings for the outer triangle (since triangle relators don't change the number of crossings), so after applying the relator associated to the now empty interior triangle we reduce the number of interior crossings for the outer triangle.
\end{proof}
Since we assumed the bigon was a minimal reducible configuration of strands, we can apply the lemma to remove strands from the triangle contained in the bigon, and then apply the corresponding triangle relator. This reduces the number of crossings which gives us our result by induction. 

2. Now assume we have two webs, $w$ and $w'$ without reducible configurations. The strands and crossings of strands of the two webs $w$ and $w'$ are in bijection. Pick an ordering of strands. We will recursively make $w$ into $w'$ on the first $n$ strands. By induction, assume that the topological minors made up of $\lb s_1,...,s_n \rb$ are the same for $w$ and $w'$, and define $w_n$ to be the minor made up of $\lb s_1,...,s_n \rb$.

Look at the two strands corresponding to $s_{n+1}$ on $w$ and $w'$. These each give paths on $w_n$ which we'll denote $\pi$ and $\pi'$ which together bound a disk. We need to show that we can homotope $\pi$ to $\pi'$ using only triangle moves. By induction we can assume that $s_{n+1}$ and $s_{n+1}'$ only cross once, and it is enough to assume that all of the $s_i$ for $i \in 1,...,n$ cross both $\pi$ and $\pi'$. Now orient each strand so that it goes from $\pi$ to $\pi'$. This gives a poset on crossings in the disk which is defined as follow: let $v,v'$ be crossings, then $v \leq v'$ if there is an oriented path from $v$ to $v'$. There are no directed loops: indeed take a minimal counter example. Then the loop has no self-intersection by minimality and hence bounds a disk with all inward pointing vectors oriented to the right (without loss of generality) of each edge. These edges correspond to strands, and the intersection of the corresponding (right-sided) half-spaces is a subset of the region bounded by the loop on one hand, but on the other hand we know all of these half-spaces contains $\pi(1)=\pi'(1)$ which gives a contradiction.

Now let $v$ be a crossing that is minimal in the above poset. The corresponding strands are $s_i$ and $s_j$ for some $i$ and $j$. By minimality, the triangle bounded by $s_i$, $s_j$, and $\pi$ has no crossing strands, and hence we can apply the corresponding relator. By induction this shows that $\pi$ can be homotoped to $\pi'$, which completes the proof of the second claim.

Hence by the previous remarks, $\gr_\bu FSp/ \lan \gr_\bu R \ran \surj gr_\bu Sp$ is a bijection, so the spectral sequence converges on the first page.
\end{proof}

In fact, in the proof we proved the following stronger result which we'll need for inductive arguments in $A_1^n \times A_2$ and the geometric results of chapter 6:

\begin{coro}
\label{A1n:singlestrand}
If $w$ and $w'$ are $A_1^n$ webs which differ only by a single strand $s$, then there exists a sequence of relators sending $w$ to $w'$ all of which involve $s$ (and hence the projection onto the $A_1^{n-1}$ factor which doesn't contain $s$ is invariant under each relator in the sequence)
\end{coro}

\subsection{\texorpdfstring{$E^1$ Convergence for $A_1 \times A_2$}{ Convergence for }}
In this section we will show that $A_1 \times A_2$ has an $E^1$ convergent spectral sequence. Recall that this spider has the following relators (where when drawing product spiders we will not draw the orientation of every edge as long as the orientation of the corresponding strand is indicated)
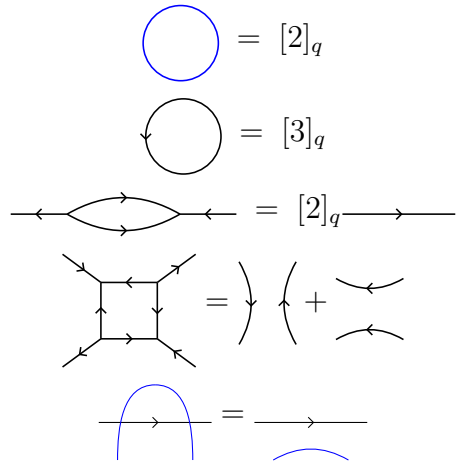
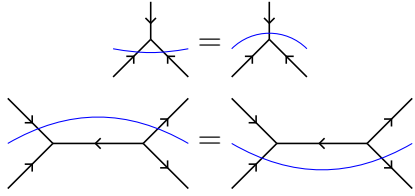
\begin{equation*}
\begin{tikzpicture}[baseline=-.3ex,scale=.5] 
\draw[web, blue] (0,0) circle (1cm);
\end{tikzpicture} \; =\; [2]_q
\end{equation*}
\begin{equation*}
\begin{tikzpicture}[baseline=-.3ex,scale=.5] 
\draw[web, midto] (0,0) circle (1cm);
\end{tikzpicture} \; =\; [3]_q
\end{equation*}
\begin{equation*}
\begin{tikzpicture}[baseline=-.3ex,scale=.75] 
\draw[midfrom,web] (-1,0) -- (0,0);
\draw[web, midto]  (0,0) to[out=30,in=150] (2,0);
\draw[web, midto]  (0,0) to[out=-30,in=-150] (2,0);
\draw[midfrom,web] (2,0) -- (3,0);
\end{tikzpicture}\; =\; [2]_q
\begin{tikzpicture}[baseline=-.3ex,scale=.75] 
\draw[midto,web] (-1,0) -- (1,0);
\end{tikzpicture}
\end{equation*}
\begin{equation*}
\begin{tikzpicture}[baseline=-2ex,scale=.75]
\draw[midfrom,web] (-1.68,-1.5) -- (-1,-1);
\draw[midto,web] (-1.68,0.5) -- (-1,0);
\draw[midto,web] (0.68,-1.5) -- (0,-1);
\draw[midfrom,web] (0.68,0.5) -- (0,0);
\draw[midto,web] (-1,-1) -- (0,-1);
\draw[midto,web] (-1,-1) -- (-1,0);
\draw[midto,web] (0,0) -- (0,-1);
\draw[midto,web] (0,0) -- (-1,0);
\end{tikzpicture} = 
\begin{tikzpicture}[baseline=-2ex,scale=.55]
\draw [web, midfrom]  (-1.68,-1.5) to[out=60,in=-60] (-1.68,0.5);
\draw [web, midto]   (-.3,-1.5) to[out=120,in=240] (-.3,0.5);
\end{tikzpicture} +
\begin{tikzpicture}[baseline=-.5ex,scale=.65]
\draw [web, midfrom]  (-1.68,-.75) to[out=30,in=150] (-.3,-.75);
\draw [web, midto]   (-.3,0.5) to[out=210,in=-30] (-1.68,0.5);
\end{tikzpicture}
\end{equation*}
\begin{equation*}
\begin{tikzpicture}[baseline=0ex,scale=.5]
\draw[midto] (0,0) -- (3,0);
\draw[blue] (0.5,-1) to[out=90,in=180] (1.5,1);
\draw[blue] (1.5,1) to[out=0,in=90] (2.5,-1);
\end{tikzpicture} = 
\begin{tikzpicture}[baseline=0ex,scale=.5]
\draw[midto] (0,0) -- (3,0);
\draw[blue] (.5,-1) to[out=30,in=150] (2.5,-1);
\end{tikzpicture}
\end{equation*}
\begin{equation*}
\begin{tikzpicture}[baseline=2ex,scale=.5]
\draw[web,midto] (0,0) -- (1,1);
\draw[web,midto] (2,0) -- (1,1);
\draw[web,midto] (1,2) -- (1,1);
\draw[blue] (0,.75) to[out=-10,in=190] (2,.75);
\end{tikzpicture}
= 
\begin{tikzpicture}[baseline=2ex,scale=.5]
\draw[web,midto] (0,0) -- (1,1);
\draw[web,midto] (2,0) -- (1,1);
\draw[web,midto] (1,2) -- (1,1);
\draw[blue] (0,.75) to[out=45,in=135] (2,.75);
\end{tikzpicture}
\end{equation*}
\begin{equation*}
\begin{tikzpicture}[baseline=-.3ex,scale=.6] 
\draw[web,midto] (-1,-1) -- (0,0);
\draw[web,midto] (-1,1) -- (0,0);
\draw[web,midfrom] (0,0) -- (2,0);
\draw[web,midto]  (2,0) -- (3,1);
\draw[web,midto] (2,0) -- (3,-1);
\draw[blue] (-1,0) to[out=30,in=150] (3,0); 
\end{tikzpicture} =
\begin{tikzpicture}[baseline=-.3ex,scale=.6] 
\draw[web,midto] (-1,-1) -- (0,0);
\draw[web,midto] (-1,1) -- (0,0);
\draw[web,midfrom] (0,0) -- (2,0);
\draw[web,midto]  (2,0) -- (3,1);
\draw[web,midto] (2,0) -- (3,-1);
\draw[blue] (-1,0) to[out=-30,in=210] (3,0); 
\end{tikzpicture} 
\end{equation*}
Since the $A_1$ strands of $A_1 \times A_2$ correspond to $A_2$ cut paths, we can cite the following result:
\begin{lemma}[\cite{kuperberg:rank} Lemma 6.5]
\label{A1A2:Cutpath}
If $\al$ and $\beta$ are cut paths connecting boundary faces $p$ and $q$ of a basis web $w$ and $\al$ is minimal, then the weight of $\al$ is less than or equal to (and not incomparable to) the weight of $\beta$. If $\beta$ is also minimal, the two parts of $w$ cut by $\beta$ are the same as those of $w$ cut by $\al$ up to $H$-moves
\end{lemma}

In fact what the proof shows is that even if $\beta$ is not minimal, we can get from $\beta$ to $\al$ using a sequence of $U$ (type \ref{product:bigon}), $Y$ (type \ref{product:Y}) or $H$ (type \ref{product:H}) moves. To prove $E^1$ convergence we will essentially just apply this fact repeatedly. We will also implicitly use the confluence of the $A_2$ spider in assuming that non-elliptic webs form a basis.

\begin{prop}
The $A_1 \times A_2$ spider with the product presentation has an $E^1$ convergent spectral sequence
\end{prop}
\begin{proof}
We will show that minimal vertex webs are those whose $A_2$ component is non-elliptic, and whose $A_1$ strands follow minimal cut paths. 

We first show that webs with elliptic $A_2$ faces are reducible. Let $w$ be an $A_1 \times A_2$ web whose $A_2$ subweb has an elliptic face, ie a square, bigon, or loop. Fix one such face $f$. If there are no $A_1$ strands crossing the face or contained in the face, then we can apply the corresponding $A_2$ relator, so let's assume there is at least one crossing strand $s$. If one of the crossing strands is a closed $A_1$ loop contained in the face, find a minimal closed $A_1$ loop, and apply that relator. Otherwise, looking at the subweb which contains the $A_2$ subweb and only a fixed crossing $A_1$ strand $s$.
\begin{itemize}
	\item In the case where $f$ is an $A_2$ loop: $s$ forms a bigon with the $A_2$ strand.
	\item In the case where $f$ is an $A_2$ bigon: $s$ must form a $Y$ with both boundary vertices
	\item In the case where $f$ is an $A_2$ square: $s$ forms either a $Y$ or an $H$ with the boundary depending on whether it crosses adjacent edges or opposite edges.
\end{itemize}
Now returning to $w$, since $A_1$ strands can't intersect each other, we can find a minimal strand such that there are no strands between that strand and a $U$, $Y$, or $H$ relator face, and thus apply that relator to remove that strand from the elliptic face. Inducting on the number of crossing strands, we can remove all crossing strands, and then apply the relator.

Now assume the $A_2$ factor of $w$ has no elliptic faces, but there is at least one $A_1$ strand $s$ that doesn't follow a minimal cut path. Look at the minor composed of the $A_2$ subweb, and $s$. By the proof of lemma \ref{A1A2:Cutpath}, there exists a sequence of length-non-increasing relators sending $s$ to a minimal cut path $p_s$. Back in $w$, $p_s$ may cross other $A_1$ strands, but we'll argue by induction on the number of such crossings that we can make those strands not cross $p_s$.

Find a strand $s'$ which crosses $p_s$ minimally at $A_2$ faces $f_1$ and $f_2$ in the sense that $s'$ crosses $s$ at $f_1$, and then again at $f_2$, and there are no crossings of $s'$ and $p_s$ between $f$ and $f'$. We know $p_s$ is a minimal cut path, so every segment is also minimal length, so by replacing the segment of $s'$ between $f$ and $f'$ so as to follow $p_s$, we would get a cut path of equal or less length. Again by the proof of lemma \ref{A1A2:Cutpath}, we can do this by using length-non-increasing relators. Since, the area bounded by $s'$ and $p_s$ is strictly contained in the area bounded by $s$ and $p_s$, so we can apply induction to move $s'$ so as to not intersect $s$ and or $p_s$. Applying induction on the number of strands crossing $p_s$, we can make it so that no strands cross $p_s$ in $w$, but then using lemma \ref{A1A2:Cutpath} again we can move $s$ to $p_s$. Finally, applying induction on the number of non-minimal cut path strands we get that any minimal-vertex web has a non-elliptic $A_2$ component, and every $A_1$ strands follows a minimal cut path.

Next we will show that if we have two such webs $w$ and $w'$ whose $A_2$ component and $A_1$ components are the same, we can find a sequence of $H$ moves sending $w$ to $w'$. There is a correspondence between $A_1$ strands in $w$ and $w'$, and we can consider any strand in either follows a cut path in the $A_2$ component of $w$. We will induct on the number of times a cut-path coming from a strand in $w$ transversely intersects, or follows than diverges from a cut-path coming from a (possibly different) strand in $w'$. Pick an arbitrary pair of consecutive such crossing of cut-paths $\pi$ and $\pi'$ from $w$ and $w'$ (consecutive crossings must exist because the number of crossings of any two strands is even), which is minimal in the sense that there are no other $A_1$ strands in the bounded region. By the proof of lemma \ref{A1A2:Cutpath} we can move the corresponding segment of $w$ to the corresponding segment of $w'$. Applying induction we can make it so that any strands that whose cut-paths intersect must follow the same cut-path, and in particular corresponding strands are the same.

Elliptic webs are a basis for the $A_2$ invariant space, and planar matchings are a basis for $A_1$ invariant space. Invariants of the product are then products of invariants, and therefore we get equinumeracy between the minimal-vertex webs and the basis of the invariant space, and hence there can be no more relations among the minimal vertex webs. So the first page is the same as the representation category, which gives us $E^1$ convergence.
\end{proof}

\begin{remark} 
In the above proof we used (the proof of) the fact that a non-elliptic dual $A_2$ web has coherent geodesics as in \cite{fkk:buildings}. Conversely, given a spider $H$ such that $H \times A_1$ spider has $E^1$ convergence with the product presentation (and any grading finer than the weight poset length of strands), then it's not too hard to show that the equivalence class of a minimal web must have \textbf{coherent geodesics}. By this we mean that for every pair of boundary points, there is a representative web with a geodesic of length $\ell$ between those two points, and every path between those points in every web in the equivalence class has length greater than (in particular comparable to) $\ell$.
\end{remark}

\subsection{\texorpdfstring{$E^2$ Convergence for $A_1^n \times A_2$}{ Convergence for }}
We now move on to the rank $\geq 4$ spiders $A_1^n \times A_2$ for $n\geq 2$. These spiders have the same relators as $A_1 \times A_2$, plus the following additional two relators:
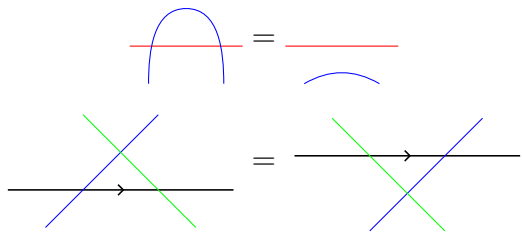
\begin{equation*}
\begin{tikzpicture}[baseline=0ex,scale=.5]
\draw[red] (0,0) -- (3,0);
\draw[blue] (0.5,-1) to[out=90,in=180] (1.5,1);
\draw[blue] (1.5,1) to[out=0,in=90] (2.5,-1);
\end{tikzpicture} = 
\begin{tikzpicture}[baseline=0ex,scale=.5]
\draw[red] (0,0) -- (3,0);
\draw[blue] (.5,-1) to[out=30,in=150] (2.5,-1);
\end{tikzpicture}
\end{equation*}
\begin{equation*}
\begin{tikzpicture}[baseline=1.5ex,scale=1] 
\draw[web] (0,0) -- (1,0); 
\draw[midto,web] (1,0) -- (2,0);
\draw[web] (2,0) -- (3,0);

\draw[blue] (0.5,-.5) -- (1,0); 
\draw[blue] (1,0) -- (1.5,.5);
\draw[blue] (1.5,.5) -- (2,1);

\draw[green] (2.5,-.5) -- (2,0); 
\draw[green] (2,0) -- (1.5,.5);
\draw[green] (1.5,.5) -- (1,1);
\end{tikzpicture}
 \;=\;
\begin{tikzpicture}[baseline=-1ex,scale=1,transform shape, rotate=180] 
\draw[web] (0,0) -- (1,0); 
\draw[midfrom,web] (1,0) -- (2,0);
\draw[web] (2,0) -- (3,0);

\draw[blue] (0.5,-.5) -- (1,0); 
\draw[blue] (1,0) -- (1.5,.5);
\draw[blue] (1.5,.5) -- (2,1);

\draw[green] (2.5,-.5) -- (2,0); 
\draw[green] (2,0) -- (1.5,.5);
\draw[green] (1.5,.5) -- (1,1);
\end{tikzpicture}
\end{equation*}
and as long as $n\geq 3$, we also get an extra type of triangle relator:
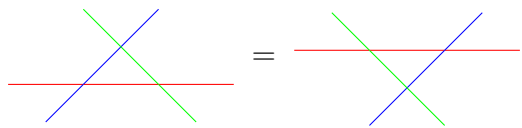
\begin{equation*}
\begin{tikzpicture}[baseline=1.5ex,scale=1] 
\draw[red] (0,0) -- (1,0); 
\draw[red] (1,0) -- (2,0);
\draw[red] (2,0) -- (3,0);

\draw[blue] (0.5,-.5) -- (1,0); 
\draw[blue] (1,0) -- (1.5,.5);
\draw[blue] (1.5,.5) -- (2,1);

\draw[green] (2.5,-.5) -- (2,0); 
\draw[green] (2,0) -- (1.5,.5);
\draw[green] (1.5,.5) -- (1,1);
\end{tikzpicture}
 \;=\;
\begin{tikzpicture}[baseline=-1ex,scale=1,transform shape, rotate=180] 
\draw[red] (0,0) -- (1,0); 
\draw[red] (1,0) -- (2,0);
\draw[red] (2,0) -- (3,0);

\draw[blue] (0.5,-.5) -- (1,0); 
\draw[blue] (1,0) -- (1.5,.5);
\draw[blue] (1.5,.5) -- (2,1);

\draw[green] (2.5,-.5) -- (2,0); 
\draw[green] (2,0) -- (1.5,.5);
\draw[green] (1.5,.5) -- (1,1);
\end{tikzpicture}
\end{equation*}

In the last chapter we will show  that $A_1^n \times A_2$ with the product presentation does \tit{not} have an $E^1$ convergent spectral sequence. Nevertheless, we can prove that it has an $E^2$ convergent spectral sequence by combining the results for $A_1 \times A_2$ and $A_1^n$. The key idea of the proof is that we can move each $A_1$ strand into the correct position with respect to the $A_2$ factor using $H$, $Y$, and $U$ moves. The other $A_1$ strands can interfere with each $A_1 \times A_2$ relator step, but only enough to force one extra crossing, which then can be fixed at the end of the $A_1 \times A_2$ relator step preventing unbounded growth. For the sake of this case, we will extend what we mean by a sequence of relators to include sequences which increase or decrease the number of vertices in the web. Since all relators in this spider have exactly two terms, this should not create too much confusion. 

\begin{prop}
\label{A1nA2:E2}
The $A_1^n \times A_2$ spider with the product presentation has an $E^2$ convergent spectral sequence
\end{prop}
The difference between the $n=1$ case and the $n>1$ case is that $A_1$ strands from different $A_1$ factors can intersect. We will first prove the following lemma which will limit how much the extra $A_1$ strands can interfere with relators:

\begin{lemma}
\label{A1nA2:Evacuation}
Let $w$ be a $A_1^n \times A_2$ web, and let $\pi_i$ be the projection onto the $i$th $A_1 \times A_2$ component. If $\pi_i(w)$ has a $U$ (type \ref{product:bigon}), $Y$ (type \ref{product:Y}), or $H$ (type \ref{product:H}) relator face there is a sequence of relators $r=(r_k)$ on $w$ in an $\ep$-neighborhood of the corresponding bounded region such that either $r(w)$ has a reducible face, or the relator face has no crossing strands. We also have the following properties:
\begin{enumerate}
	\item If the relator is an $H$ relator, $r$ can be chosen such that $|r_k...r_1(w)| \leq |w| +1$ for all $k<n$ and $|r(w)|=|w|$. If it is a $U$ or $Y$ relator, then $r$ can chosen to consist only of horizontal relators.
	\item If the relator is an $H$ relator, every relator includes a segment of the $A_1$ strand of the configuration, or an $A_1$ strand which parallel to it in an $\ep$-neighborhood of the configuration
\end{enumerate}
\end{lemma}
\begin{proof}
We'll argue case by case that either there is a sequence of triangle relators in $w$ which remove all other strands from the bounded region, or we create a reducible face. We first restrict to an $\ep$-neighborhood of the configuration (which may cut some strands into multiple pieces if the original strand crossed the configuration more than once). Moreover, if there is a closed loop contained in the configuration, then in an $\ep$-neighborhood of that loop we have a $A_1^n$ web, and hence can apply \cref{E1:A1n} to the subweb to find as sequence of relators reducing it. So assume that there are no loops.
\begin{enumerate}
	\item If it is a $U$ relator then the bounded relation has the same relators as an $A_1^n$ web, so we can apply \cref{E1:A1n} to the subweb to find as sequence of relators reducing it.
	\item If it is a $Y$ relator we can assume each crossing strand  makes a triangle or $Y$ relator configuration with one of the corners, otherwise if it crossed the same edge twice we could apply the $A_1^n$ result again. Find a strand that is minimal with respect to a corner, in that no other $A_1$-strand crosses strictly between it and the corner. If it makes a triangle relator configuration with a corner, we can restrict to the region bounded by that strand and the corner where there are no trivalent vertices. Applying the $A_1^n$ result again we can remove the strand from the original $Y$ configuration. By induction, we can remove all such strands from the $Y$ configuration. On the other hand, if it makes a $Y$ configuration with a corner, then we can immediately apply induction to reduce the web.  
	\item If it is an $H$ relator with $A_1$ strand $s$, by the same arguments as above, we can assume that every strand crosses opposite edges of the bounded square region. Moreover, we can assume that it doesn't cross the two opposite edges of the $A_2$ factor, otherwise we could apply induction to remove it from the configuration. \tit{This is the point where we have to increase the number of crossings.} By the above there are no strands crossing the two opposite $A_2$ edges, so we can apply an inverse $Y$ relator to $s$ and one of the $A_2$ corners (increasing the number of crossings by $1$). Then $s$ forms a $Y$ configuration with the other $A_2$ corner for example as depicted in (\ref{increasingcrossings}). Applying the above result for $Y$ relators, we can reduce this subweb without increasing the number of vertices, giving us our result.
\begin{equation}
\label{increasingcrossings}
\begin{tikzpicture}[baseline=-.3ex,scale=.6]
\draw[midto,web] (-1,-1) -- (0,0);
\draw[midto,web] (-1,1) -- (0,0);
\draw[midfrom,web] (0,0) -- (2,0);
\draw[midto,web]  (2,0) -- (3,1);
\draw[midto,web] (2,0) -- (3,-1);
\draw[red] (.9,-1) -- (.9,1);
\draw[blue] (-1,0) to[out=30,in=150] (3,0);
\end{tikzpicture}\mapsto 
\begin{tikzpicture}[baseline=-.3ex,scale=.6]
\draw[midto,web] (-1,-1) -- (0,0);
\draw[midto,web] (-1,1) -- (0,0);
\draw[midfrom,web] (0,0) -- (2,0);
\draw[midto,web]  (2,0) -- (3,1);
\draw[midto,web] (2,0) -- (3,-1);
\draw[red] (.9,-1) -- (.9,1);
\draw[blue] (-1,0) to[out=-45,in=240] (.9,.2) to[out=60,in=150] (3,0);
\end{tikzpicture}
\end{equation}
\end{enumerate}
\end{proof}

\begin{proof}[Proof of \Cref{A1nA2:E2}]
Similar to $A_1 \times A_2$ we have a natural set of minimal vertex webs: those whose $A_2$ component is non-elliptic, whose $A_1$ strands follow all minimal cut paths, and any two $A_1$ strands cross at most once. We will show that all non-minimal-vertex webs can be reduced without increasing the number of vertices by more than $1$, and that any two minimal-vertex webs $w$ and $w'$ are equivalent via a sequence of webs which have at most one more vertex than $w$ and $w'$. 

Let's first prove that any web not of the above form is reducible. Let $w$ be a web \tit{whose $A_2$ component is elliptic} and hence has an elliptic face (which we recall is a square, bigon, or loop). Restrict to an $\ep$-neighborhood of this face, and induct on the number of vertices. If no strand intersects with the $\ep$-neighborhood of this face then we're done, so assume there is such a strand. Look at each $A_1 \times A_2$ component. By $E^1$ convergence of $A_1 \times A_2$, there is a sequence of relators $(r_i)$ which removes all strands from the elliptic face without increasing the number of vertices. The $(r_i)$ are of type $H$, $Y$, and $U$ relators. 

Applying \cref{A1nA2:Evacuation} to the relator face of $r_1$ in $w$ we get a sequence of relators either leading to a web $w'$ with $|w|=|w'|$ or a web with fewer vertices (and in either case not increasing the number of vertices by more than one). By induction, we can assume that it is the latter, in which case $w'$ possesses a relator corresponding to $r_1$, and the sequence of relators did not change the projection onto the chosen $A_1 \times A_2$ factor since all the relators in the sequence are contained in an $\ep$-neighborhood of the relator face of $r_1$, and any strands from the chosen $A_1$ factor which crossed the relator would make applying $r_1$ on the $A_1 \times A_2$ component impossible. Next, we apply this relator, and then apply induction on the length of $(r_i)$ to get a sequence of relators leading to a web where we may remove a strand from the elliptic $A_2$ face. Since the sequence of relators didn't increase the number of vertices at the end, removing a strand reduces the number of edges in an $\ep$-neighborhood of the elliptic face, and hence by induction we can remove all strands from the elliptic face. Applying the corresponding reduction relator then gives us our result.

If there is \tit{an $A_1$ strand that doesn't follow a minimal cut-path}, we know that in the corresponding $A_1 \times A_2$ component there is a sequence of relators which reduces the web. By the same argument as in the non-elliptic case, we get a sequence of relators in $w$, which gives us our result.

Now assume that the $A_2$ factor is non-elliptic, but \tit{there is an $A_1$ loop}. We can cut the web along the outside of the loop strand to get a subweb. Then the $A_1$ strand is a non-minimal cut path, and so by \cref{A1A2:Cutpath} we can find a sequence of relators in the $A_1 \times A_2$ factor which contracts the $A_1$ strand. Then by \cref{A1nA2:Evacuation} and the same argument as before, we can remove all trivalent vertices from the bounded region, and hence reduce the web.

If there are \tit{$A_1$ strands that cross more than once}, look at the strand-segments going from one crossing to the next. Cut along the outside of the bounded region to get a subweb where both strands can be considered cut paths on the $A_2$ factor. There is a sequence of relators in the $A_1 \times A_2$ component that send one cut path to the other again by \ref{A1A2:Cutpath}, so again by \cref{A1nA2:Evacuation} we can remove all trivalent vertices from the bounded region, and then apply the $A_1^n$ convergence result to the remaining bounded region. This completes the case where $w$ is reducible.

Lastly, we examine the case where we have two webs $w$ and $w'$ whose $A_1$ and $A_2$ components are all the same, and are minimal vertex webs. We want to argue that we can go from $w$ to $w'$ without increasing the number of vertices by more than $1$. The argument will be similar to before, except we need to be careful about which components we change. We first order the strands $(s_i)$ in the web $w$. We'll look at the projection $\pi_{\leq k}$ onto the $A_2$ component and first $k$ strands of $A_1^n \times A_2$. 

We'll induct on the number of faces in our web, and secondly on the largest $k$ such that $\pi_{\leq k}(w)=\pi_{\leq k}(w')$. We'll prove that we can move $s_{k+1}$ in $w$ to follow the same $A_2$ cut path as the version from $w'$ with every relator being $\pi_{\leq k}$ invariant, and without increasing the number of vertices by more than one. If $w$ has only one face, then the statement is vacuously true since the web is the empty graph. So assume that $\pi_{\leq k}(w)=\pi_{\leq k}(w')$ where $k$ can be zero. Via the isomorphism $\pi_{\leq k}(w)=\pi_{\leq k}(w')$ the strands $s_{k+1}$ on $w$ and $w'$ correspond to (minimal) cut paths on $\pi_{\leq k}(w)$. Using induction it suffices to assume that the two $A_2$ cut-paths follow the boundary, and intersect only their endpoints. By \cref{A1A2:Cutpath}, there exists at least one $H$ move $h$ on the $A_2$ component which reduces the number of trivalent vertices between the two cut-paths. By \cref{A1nA2:Evacuation} we can find a sequence of relators $r=(r_i)$ in $w$ such that we can apply that $H$ relator in $r(w)$, and moreover all relators are $\pi_{\leq k}$ invariant unless one of the first $k$ $A_1$ strands is parallel to $s_{k+1}$ in an $\ep$-neighborhood of the relator configuration $h$. So let's assume that there is such a parallel strand $s$, and let's also pick that $s$ to intersect the $H$-relator configuration $h$ closest to $s_{k+1}$ along one of the $A_2$ edges in the set of all such parallel strands. For example, in the diagram (\ref{A1nA2:closest}) below, we would require there to be no $A_1$ strands crossing the (without loss of generality) right-side $A_2$ strand segment between the top (blue) $A_1$ strand, and the bottom (red) $A_1$ strand.

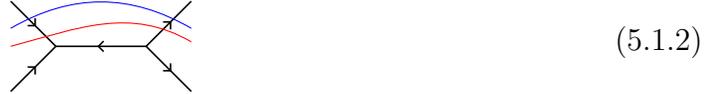
\begin{equation}
\label{A1nA2:closest}
\begin{tikzpicture}[baseline=-.3ex,scale=.6] 
\draw[web, midto] (-1,-1) -- (0,0);
\draw[web, midto] (-1,1) -- (0,0);
\draw[web, midfrom] (0,0) -- (2,0);
\draw[web, midto]  (2,0) -- (3,1);
\draw[web, midto] (2,0) -- (3,-1);
\draw[blue] (-1,.4) to[out=30,in=150] (3,.4); 
\draw[red] (-1,0) to[out=15,in=150] (3,.1); 
\end{tikzpicture}
\end{equation}

The strands $s$ and $s_{k+1}$ can't be globally non-intersecting on $w$, otherwise $s$ and $s_{k+1}$ would intersect twice on $w'$, and hence they must intersect somewhere in $w$ outside the $\ep$-neighborhood, for example as in diagram (\ref{A1nA2::outside}) below:

\begin{equation}
\label{A1nA2::outside}
\begin{tikzpicture}[baseline=-.3ex,scale=.6] 
\draw[web, midto] (-1,-1) -- (0,0);
\draw[web, midto] (-1,1) -- (0,0);
\draw[web, midfrom] (0,0) -- (2,0);
\draw[web, midto]  (2,0) -- (3,1);
\draw[web, midto] (2,0) -- (3,-1);
\draw[blue] (-2,-.8) to[out=60,in=210] (-1,.4) to[out=30,in=150] (3,.4); 
\draw[red] (-2,-.4) -- (-1,0) to[out=15,in=150] (3,.1); 
\end{tikzpicture}
\end{equation}

Now the cut paths on the $A_2$ component corresponding to $s$ and $s_{k+1}$ intersect both at the $A_2$ face corresponding to their crossing, along with the $A_2$ face which contains the $H$ relator. The subweb bounded by the two strands is strictly smaller than $w$, so by the vertex induction we can move $s_{k+1}$ to follow $s$ on this segment in $w$ with all relators being invariant under $\pi_{\leq k}$. However, by doing this we may have to move $s_{k+1}$ over some $A_2$ trivalent vertices, which could mess up the algorithm. However, we've already restricted to the case where all $A_2$ vertices are bounded between the cut-paths corresponding to $s_{k+1}$ on $w$ and $w'$, so any such relator reduces the number of bounded vertices, which would allow us to apply induction.


Then since there are no vertices in the region bounded by the two strands $s$ and $s_{k+1}$, we can use triangle relators to move the crossing of $s$ and $s_{k+1}$ to occur in an $\ep$-neighborhood of the $H$-relator configuration, and hence make it non-parallel in the $\ep$-neighborhood. Since all of these relators involve $s_{k+1}$, they are $\pi_{\leq k}$-invariant too. So by repeating this process for each parallel strand, we can apply the $H$ relator by \cref{A1nA2:Evacuation}. This reduces the number of vertices bounded by the two cut-paths corresponding to $s_{k+1}$ on the $A_2$ factor, so we can apply induction on the number of vertices. Hence, we can assume $s_{k+1}$ gives the same $A_2$ cut-path for $w$ and $w'$, and hence the region bounded by the cut-paths in $\pi_{\leq k}(w)=\pi_{\leq k}(w')$ contains no trivalent vertices. Therefore, this bounded region has the same relators as an $A_1^n$ spider, so by \cref{A1n:singlestrand} there is a sequence of relators in $\pi_{\leq k+1}(w)$ moving $s_{k+1}$ to the same place as the $s_{k+1}$ in $\pi_{\leq k+1}(w')$ which is invariant under $\pi_{\leq k}$. We can then use \cref{A1n:Evacuate} to extend this to a sequence of relators on $w$ (possibly at the expense of moving some other strands $s_i$ with $i>k+1$), which completes the proof that we can go from $w$ to $w'$ as desired.

Finally we need to connect this to the spectral sequence. We have given a generating set of webs for the second page $E^2$ and shown that they are uniquely determined by their projections onto the components. Therefore, the number of such webs is equal to the product of the number of corresponding $A_1$ webs and $A_2$ webs for each component, but we know from the representation theory that this is the dimension of the invariant space. Hence, $E^2 = E^\infty$.

\end{proof}

\section{Simple Spiders Using Canceling Sequences}

\subsection{A Criterion for \texorpdfstring{$E^1$ Convergence in Rank $3$}{ Convergence in Rank }}
We will now change gears a bit and create a list of properties that are sufficient for a spider to have $E^1$ convergence. We'll see that we've already proven some of these in the case of $A_3$. To do this, we generalize the notion of types of relators. For us, a \tbf{type function} is just an arbitrary map $\tau: Rel \to \lb 0, 1 \rb$. For the contracted $A_3$ spider, we will send the triangle relators to $\ol{0}$ and the square relators to $\ol{1}$ 

To prove the theorem, we break first it into three cases: horizontal sequences on reducible webs, horizontal sequences on reduced webs, and reduction sequences. For the first case, we'll argue that the sequence is equivalent to one which has a reducible face on some web, and then reduce it to the reduction sequence case. The second case will basically follow immediately from one of the criteria (and so is postponed until later). For the last case, we'll control the sequence enough to make sure that the reduction relator is adjacent to exactly one of the horizontal relators, then change the sequence so that the last horizontal relator becomes a reduction relator on this face. This then commutes with all other relators, and hence may be omitted by lemma \ref{commute out}.

\begin{thm}
\label{CC}
Any spider presentation with trivalent and/or tetravalent vertices, whose relators correspond to faces, and which fulfills the following criteria has a spectral sequence that converges at the $E^1$ page:
\begin{enumerate}
	\item Relators: Each face is a leading order term of at most one local relator, and has at most two leading order terms
	\item Purity: $Rel$ admits a type function $\tau$ such that if two consecutive horizontal relators $s_1$ and $s_2$ of different type in a relator sequence don't commute, then applying $-s_1$ or $s_2$ will create a reducible face (and in particular the corresponding reduction sequence $(r,-s_1,s_2,r')$ will be consistently reducible by local confluence below).
	\item Global Reducibility: On any web without a reducible face, no more than one type of the horizontal relator can destroy a minimal reducible configuration of strands, and every web with a reducible configuration of strands is reducible.
	\item Simply-connected: Any horizontal sequence on a reduced web which has relators with lower order terms is equivalent to a horizontal sequence where two consecutive relators are inverses of each other.
	\item Controlled degeneration: Any two reducible configurations whose intersection is disconnected, must have a reducible configuration contained in the bounded region.
	\item Local confluence: Any canceling sequence on a web with two faces is consistently reducible.
\end{enumerate}
\end{thm}
\begin{remark}
Let us quickly discuss which of these relators seem necessary.

In order to get the relator condition, we may need to contract edges corresponding to $I=H$ relators, but otherwise the first part seems to hold for all examples we care about. The second part fails in general for all higher rank spiders (besides some semisimple cases).

The purity criterion is fundamental to the proof. It is not satisfied by the $A_1^2 \times A_2$ spider whose webs correspond to placing two $A_1$ webs and one $A_2$ web transversely on top of each other. As expected, this spider fails to have any finite presentation with the $E^1$ property.

The global reducibility criterion is used for bookkeeping in the proof. It's not clear whether the proof could be generalized without it.

The simply-connected condition is stronger than we need, but seems to hold for all rank 3 spiders, but not for higher rank spiders.

The controlled degeneration criterion appears necessary. It is satisfied by neither the annular $A_3$ spider nor the $A_2 \times A_2$ spider. Correspondingly, these two spiders do not have finite presentations with $E^1$ convergence

Finally, local confluence is certainly necessary for $E^1$ convergence, although a failure here can be repaired by adding in relators.
\end{remark}

Before getting to the proof, we need a generalization of lemma \ref{A3 Purity}:
\begin{lemma}
\label{purity}
Given the hypotheses of theorem \ref{CC}, for any reducible web $w$, there is a minimal sequence of horizontal relators $(s_i)$ such that $s_n...s_1(w)$ has a reducible face, the type of every horizontal relator in the sequence is the same, and this type can be chosen to be either one of the two types. Moreover, every reduction sequence decomposes into a reduction sequence $s'$ which contains only one type of horizontal relator, along with reduction sequences of the form $(r_0,s_1,s_2,r_3)$ where $s_1$ and $s_2$ are non-commuting horizontal relators
\end{lemma}
\begin{remark}
The $(r_0,s_1,s_2,r_3)$ correspond to only two faces, and so reduce consistently by local confluence.
\end{remark}
\begin{proof}
The proof of the first claim is the same as in proposition \ref{A3 Purity}. For the second claim, we have a reduction sequence $s=(s_0,s_1,...,s_n,s_{n+1})$ where $s_1$ is a relator on a web $w$. If $s_n$ and $s_{n+1}$ commute then we can remove the $s_n$ by lemma \ref{commute out}, so we can assume that they do not commute. Then there must be a relator $s_n'$ of a different type than $s_n$ which when be applied to $w_n$ creates a reducible face. Then $s$ decomposes into $s'=(s_0,s_1,...,s_n',r)$ and $(-r,-s_n',s_{n},s_{n+1})$ for some reduction relator $r$. We claim that we can make all the horizontal relators in $s$ be of type $\tau(s_n)$ or $\tau(s_n')$. Let $s_k$ have the largest $k$ of relators of type different from $\tau(s_n)$. We can commute it to a later position, so we can assume that $s_k$ and $s_{k+1}$ don't commute, or $k=n$. In the former case, we get a decomposition of $s$ into $(s_0,s_1,...,s_k,r_1)$, $(-r_1,-s_k,s_{k+1},r_2)$, and $(-r_2,s_{k+2}..,s_n,s_{n+1})$ for some reduction relators $r_1$ and $r_2$. $(s_0,s_1,...,s_k,r_1)$ and $(-r_2,s_{k+2}..,s_n,s_{n+1})$ have shorter length, so we can apply induction. On the other hand, if $k=n$, then if $s_k=s_n$ and $s_{n+1}$ commute, then it is equivalent to a sequence without $s_n$ by lemma \ref{commute out}. Otherwise, we get a decomposition of $s$ into $(s_0,s_1,...,s_n'',r_3)$ and $(-r_3,-s_n'',s_n,s_{n+1})$ for some horizontal relator $s_n''$ of the same type as the original $s_n$, and $r_3$ some reduction relator. $(s_0,s_1,...,s_n',r_3)$ has more type $\tau(s_n)$ relators than $s$, so we can repeat this algorithm until $s$ has only one type of relator. Doing the same thing to $s'=(s_0,s_1,...,s_n',r)$ will give us a sequence of the other type.
\end{proof}

\begin{proof}[Proof of Theorem \ref{CC}] 
First let's assume without loss of generality that type $0$ relators don't affect the reducible configurations.

By proposition \ref{Canceling}, it is sufficient to prove that all horizontal and reduction sequences reduce consistently. We induct first on the skein grading, then on the number of horizontal relators (for those webs with the same number of vertices). First we will reduce to the reduction sequence case by proving that for any horizontal canceling sequence $s$, there is a reduction sequence $s'$ such that $d(\sum s) = d(\sum s')$ with the same number of horizontal relators. Then we will prove the reduction sequence case.

If $s=(s_i)$ is a \tbf{horizontal sequence on a reducible web}, assume by contradiction that $s$ doesn't reduce consistently and that it has the shortest length of any minimal canceling sequence with this property. By purity, we know that we can decompose the sequence into types, unless some relators of different types don't commute. In that case, after commuting two such relators adjacent to each other, applying either of those non-commuting relators will create a reducible face. In other words, we can take $s_1$ and $s_n$ to not commute, and hence there exists a reduction relator $r$ such that $(r,s_1,...,s_n,-r)$ is a reduction sequence with the same number of horizontal relators, and thus we can reduce to the reduction sequence case. If, on the other hand, the sequence has relators of only one type, by lemma \ref{purity} we know that for each $k$ there is a minimal sequence of relators $(t_i^k)_{i=1}^{m_k}$ on $w_k=s_k...s_1(w)$ which are of a different type than $s_i$ and lead to a reducible face. Now take $k$ such that $m_k$ is minimal. We claim that each $t_i^k$ commutes with all $s_j$. Otherwise, let $t_i^k$  (where $0<i<m_k$) be the first relator which doesn't commute with all relators in $s$, and let $s_j$ be the first relator in $s$ which doesn't commute with $t_i^k$. Take each $t'^k_i$ (where $0 \leq 0 \leq i$) to be defined as the relator equivalent to $t^k_i$ on $w_j$. Then $t'^k_{i-1}...t'^k_1(w_{j+1})$ has a reducible face since $t'^k_{i}$ and $s_j$ don't commute, so this contradicts the fact that $m_k$ is minimal. Hence, the $t_i^k$ commute with all $s_i$, and so applying the sequence to all the relators, we get a horizontal sequence $s'$ such that $w_k$ has a reducible face, and is equivalent to the original sequence by lemma \ref{horiz comm}. This is a reduction sequence $(r,s_{k+1}',...s_n',s_1',...s_{k}', -r)$ with the same lower order terms as $s$.

If $s=(s_i)$ is a \tbf{horizontal sequence on a non-reducible web} then by the simply-connectedness criterion, it is equivalent to a sequence with consecutive inverse relators. Canceling the inverse relators, and applying induction, we get our result.

Finally, if $s=(s_0,s_1,...,s_n,s_{n+1})$ is a \tbf{reduction sequence}, then by applying lemma \ref{purity}, we can decompose it into a sequence composed purely of relators of the type which doesn't alter reducible configurations together with small sequences which are consistently reducible. On $w_{n} := s_{n}...s_1(w)$, $w_0:=w$ (where $w$ is the leading term of $s_0$), we have a reducible face along with its associated configuration of strands $C$. Symmetrically, we also have such a reducible configuration of strands $C'$ on $w$. By the global reducibility property, we know that all the webs $w_i$ continue to have these two configurations of strands. We will use two different arguments to show that $C$ and $C'$ at two different place.

Define functions $f,f': \lb 0,...,n \rb \to N$, where $f$ is the minimum number of relators required to create a reducible face in an $\ep$ neighborhood of $C$ on the web $w_{n}$, and similarly for $f'$ and $C'$. By global reducibility, these numbers are finite. If for any $i =1,...,n$, $f(i)<min\lb i,n-i \rb$, then there exists a sequence $(r_1,...,r_{f(i)})$ such that $r_{f(i)}...r_1(w_i)$ has a reduction face $r$, and hence we can decompose $s$ into $(s_0,...,s_i,r_1,...,r_{f(i)},r)$ and $(-r,-r_{f(i)},...,-r_{1},s_{i+1},...,s_{n+1})$ which are both shorter reduction sequences and hence we can apply induction. So by symmetry we get $min\lb f(i), f'(i) \rb \geq min\lb i,n-i \rb$, and the sequence $s$ itself gives us $ f'(i) \leq i$ and $f(i) \leq n-i$, we get $f'(i)=i$ and $f(i)=n-i$. In particular, we know that $\ep$ neighborhoods around $C$ and $C'$ both change with every relator, and thus $C$ and $C'$ must be adjacent.

We also know that on $w_{n}$, if there are no strands crossing $C$ then it is reduction relator face, and we can decompose $s$. On the other hand, on $w_n$, the number of crossing strands must be zero, and therefore $s_n$ involves moving a strand off of $C$. So this strand on $w_{n-1}$ cuts $C$ into two faces, with $s_n$ applied to one of the two which we denote by $f$. By lemma \ref{purity}, there is another horizontal relator $s_n'$ on the second face $f'$ of $C$ (in $w_{n-1}$), and we can find a new sequence from $s$ by replacing $s_n$ by $s_n'$, getting a new sequence $(s_0,....,s_n',s_{n+1}')$ where $\tau(s_n) \neq \tau(s_n')$. Since $s_n'$ is a different type than the other $s_i$ for $0<i<n+1$, it must commute with all of these $s_i$ or else we would get a reducible face, and be able to decompose this sequence. If it also commuted with $s_0$, then we could commute it out by fact \ref{commute out}, and shorten the sequence, so assume that $(-s_1)...(-s_{n-1})(s_n')$ and $s_0$ do not commute. Thus the face corresponding to $C'$ (the face of $s_0$) is adjacent to the face of $(-s_1)...(-s_{n-1})(s_n')$ on $w$. Hence $C'$ and $C$ are adjacent or the same (one can't be contained in the other since reduction relators correspond to faces, and faces have no internal edges). However, if $C$ and $C'$ were the same, then $s_n'$ and $s_1$ wouldn't have commuted which violates our assumption, and hence $w_{n-1}$ will look something like the following (with potentially more adjacencies and strands crossing $C'$):

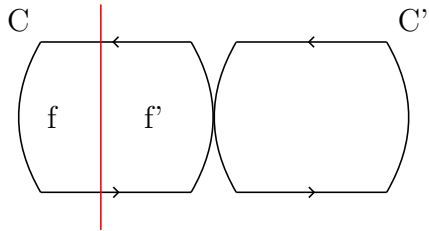
$$
\begin{tikzpicture}[baseline=-.3ex,scale=1] 
\node at (-.85, 0)  (f)     {f}; 
\node at (.5, 0)  (f')     {f'}; 
\draw[midto,web] (-1,-1) -- (1,-1); 
\draw[web] (-1,-1) to[out=120,in=240] (-1,1)node[anchor=south east] {C};
\draw[midfrom,web] (-1,1) -- (1,1);
\draw[web] (1,-1) to[out=60,in=-60] (1,1);
\draw[web,red] (-.2,-1.5) -- (-.2,1.5); 
\draw[midto,web] (1.6,-1) -- (3.6,-1); 
\draw[web] (1.6,-1) to[out=120,in=240] (1.6,1);
\draw[midfrom,web] (1.6,1) -- (3.6,1)node[anchor=south west] {C'};
\draw[web] (3.6,-1) to[out=60,in=-60] (3.6,1);
\end{tikzpicture}
$$

Where $f$ corresponds to $s_n$ which is the type $0$, and $f'$ corresponds to $s_n'$ which is type $1$, but by the previous argument also know $s_n$ must affect $C'$, and hence the face $f$ must also be adjacent to $C'$. Since $f$ and $f'$ are disjoint, we get that $C$ and $C'$ are adjacent at disconnected points.

Then by controlled degeneration, there are relators $t_i$ of a different type than $s_i$ such that $t_k....t_1(w)$ has a reducible face. If $t_i$ is the first relator which doesn't commute with all the $s_j$, then by lemma \ref{horiz comm}, we get the sequence $(s_i') \equiv (s_i)$ where $s_i'=t_{j-1}...t_1(s_i)$, but there is a reduction relator on $t_{j-1}...t_1(w_j)$ by purity, and hence we can decompose the sequence. If they all commute, then we can apply them all and get a reducible relator $r$ contained in the region bounded between $C$ and $C'$. Decomposing $s$ as $(s_0,r)$ and $(-r,s_1,...,s_n)$, we see the latter is in a neighborhood around $C'$ together with the bounded region. This neighborhood has at least one fewer face, and hence a smaller skein grading, so we can apply induction.

We have now reduced all possibilities to reduction sequences coming from purity, and reduction sequences consisting of two reduction relators, hence giving us our result.

\end{proof}

\subsection{\texorpdfstring{$E^1$ Convergence for $A_3$}{ Convergence for }}
The purpose of this section is to prove that the $A_3$ spider defined earlier satisfies the criteria of theorem \ref{CC}.

\begin{thm}
The (both contracted and uncontracted) $A_3$ spider define above has a spectral sequence which converges at $E^1$
\end{thm}
\begin{proof}
We will apply the criteria of theorem \ref{CC} to the contracted $A_3$ spider, and then the result will follow for the uncontracted spider by proposition \ref{excise}.

\begin{enumerate}
	\item The relator criterion follows immediately from the presentation.
	
	\item We proved purity in lemma \ref{collision}. 

	\item Global reducibility follows from the proof of the global reducibility criterion (theorem \ref{Global}) as soon as we prove that any triangle relator that destroys a reducible configuration must introduce a reducible face.

	\item For simply-connectedness, assume we have a horizontal sequence $s=(s_i)$ of triangle horizontal relators on a reduced contracted web $w$, let $w_i=s_i...s_1(w)$, and as usual we will induct on the number of relators in $s$. We first apply proposition \ref{extend strand} so that we can assume that $w$ has only single edges on the boundary. So the maximal undirected strands correspond to matchings of boundary edges, and hence each one divides the web into two regions. Find an undirected strand $\al$ such that no other strand is completely contained on one side (call this region $D$). Define a function $f: \N \to \N$ where $f(i)$ is the number of tetravalent vertices in $D$ on the web $w_i$. If it is constant, then the sequence decomposes into a horizontal sequences in $D$ and $D^c$ which commute and so we can apply induction. Reordering the sequence we can assume that $f(0)=f(n)$ is the minimum of the function. We will also induct on $\sum_i f(i)$ for those sequences of equal length.
	
Let $k$ be the first integer such that $f(k)$ is the maximum, and let $\ell>k$ be the first such integer such that $f(\ell+1)<f(\ell)$. Let $\ga$ and $\ga'$ be the strands together with $\al$ that correspond to $s_k$. If $\ell \neq k+1$ then and the relator $s_{k+1}$ fails to commute with $s_k$ this would mean that $s_k$ makes a strand cross the configuration of these three strands, and hence cross two of three strands. If it would cross all three, it would give four mutually intersecting strands which means the web has a reducible configuration for orientation reasons. However, by minimality of $D$, we know it must cross $\al$. Therefore $s_{k+1}$ would have to change $f$, which contradicts our assumption that $\ell \neq k+1$ and $f(k)$ is maximum. Hence we can commute $s_{k+1}$ and $s_k$. Repeating this argument to remove all relators between $s_k$ and $s_\ell$, we can assume $\ell = k+1$. Now, $s_{k+1}=s_\ell$ removes a vertex from $D$ so it can't add a vertex to the region bounded by $\al$, $\ga$, and $\ga'$ which is contained in $D$. So either $s_{k+1} = -s_k$ or the two commute. In the first case we can cancel to get a shorter sequence. In the second case we can commute them change $f$ by $f(i) \mapsto f(i)-2$ and the rest are unchanged allowing us to apply induction.

	\item For controlled degeneration, we can just list the possible ways two reducible faces can be adjacent in a degenerate way, then check them case by case. We only need to check the case where there are four or more boundary edges (in the non-contracted spider), since each adjacency will be via two boundary edges. So there are two reduction relators of this type: the cyclically oriented bigon (\ref{Double Square C}), and the reducible triangle (\ref{Single Square C}) giving the following three possibility: 
 	
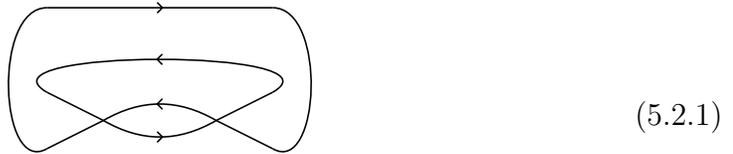
\begin{equation}
\begin{tikzpicture}[baseline=-.3ex,scale=.75] 
\draw[web] (-1,.5) -- (0,0);
\draw[web] (-1,-.5) -- (0,0);

\draw[web, midfrom]  (0,0) to[out=30,in=150] (2,0);
\draw[web, midto]  (0,0) to[out=-30,in=-150] (2,0);

\draw[web] (2,0) -- (3,.5);
\draw[web] (2,0) -- (3,-.5);
\draw[web, midto]  (3,.5) to[out=30,in=150] (-1,.5);
\draw[web]  (3,-.5) to[out=-30,in=0] (3,2);
\draw[web, midfrom]  (3,2) -- (-1,2);
\draw[web]  (-1,2) to[out=180,in=210] (-1,-.5);
\end{tikzpicture}
\end{equation}
In the above case, the middle region gives a reducible configuration, so there is a sequence of the square horizontal relators inside this region leading to a reduction relator.
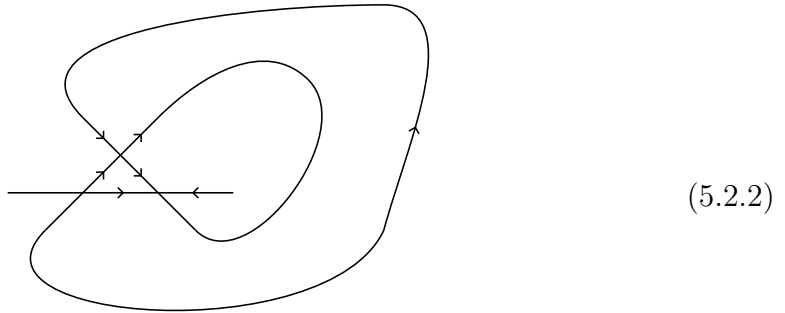
\begin{equation}
\begin{tikzpicture}[baseline=-1ex,scale=1] 
\draw[web] (0,0) -- (1,0); 
\draw[midto,web] (1,0) -- (2,0);
\draw[midfrom,web] (2,0) -- (3,0);

\draw[web] (0.5,-.5) -- (1,0); 
\draw[midto,web] (1,0) -- (1.5,.5);
\draw[midto,web] (1.5,.5) -- (2,1);

\draw[web] (2.5,-.5) -- (2,0); 
\draw[midfrom,web] (2,0) -- (1.5,.5);
\draw[midfrom,web] (1.5,.5) -- (1,1);

\draw[web]  (2.5,-.5) to[out=-45,in=-45] (4,1.5);
\draw[web]  (4,1.5) to[out=135,in=45] (2,1);
\draw[midto,web] (0.5,-.5) to[out=225,in=-115] (5,-.5) to[out=75,in=0] (5,2.5) to[out=180,in=135] (1,1);
\end{tikzpicture}
\end{equation}
In the above case, the undirected horizontal strand must cross the same strand twice in order to leave the bounded region, hence giving us our reducible configuration.

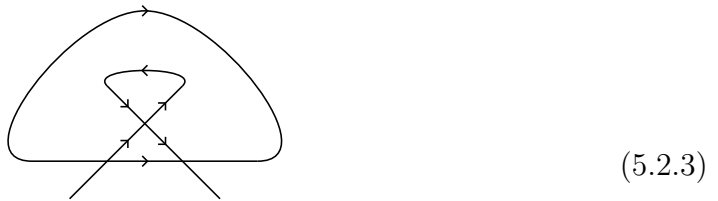
\begin{equation}
\begin{tikzpicture}[baseline=-1ex,scale=1] 
\draw[web] (0,0) -- (1,0); 
\draw[midto,web] (1,0) -- (2,0);
\draw[web] (2,0) -- (3,0);

\draw[web] (0.5,-.5) -- (1,0); 
\draw[midto,web] (1,0) -- (1.5,.5);
\draw[midto,web] (1.5,.5) -- (2,1);

\draw[web] (2.5,-.5) -- (2,0); 
\draw[midfrom,web] (2,0) -- (1.5,.5);
\draw[midfrom,web] (1.5,.5) -- (1,1);

\draw[midfrom,web] (1,1) to[out=135,in=45] (2,1);
\draw[midto,web] (0,0) to[out=180,in=180] (1.5,2) to[out=0,in=0] (3,0);
\end{tikzpicture}
\end{equation}
Again, the bounded region in the above case is a reducible relator, in this case (\ref{Mixed Bigon C}).

	\item The most labor intensive calculation is proving local confluence. Since reduction sequences can only occur when there are two relators on every web, we know that both faces must have a relator.

The first set of calculations are a monogon (\ref{Mixed Bigon C}) adjacent to other relators, which must be adjacent via tetravalent vertex oriented cyclically. We'll sometimes put a black dot on the boundary to indicate a base point to make the more complicated webs comparable.
\begin{equation}
\begin{split}

\end{split}
\end{equation}
Next are pairs of relators including a bigon of type (\ref{Pure Bigon C}). They can be adjacent either via a single edge (the tetravalent vertex won't give adjacency because they are point inward and outward which don't correspond to adjacency in the non-contracted web). Note that adjacency with a horizontal triangle relator was already done in calculation (\ref{KekulePureBigon}). The single edges of the bigons are the same up to automorphisms giving us:
\begin{equation} 
\begin{split}
%
\end{split}
\end{equation}

The reducible triangle face has two types of edges up to automorphism (the ones adjacent to the cyclic vertex, and the ones opposite) giving us:

\begin{equation}
\begin{split}
%
\end{split}
\end{equation}

Next we have the other type of bigon, and again notice that we've already taken care of that bigon's adjacency to a square relator via \ref{SquarePureBigon}. 

\begin{equation}
\begin{split}
%
\end{split}
\end{equation}

Finally, we have the most complicated computations involving reducible triangle relators adjacent to each other, or adjacent to horizontal relators.

This first one has both cyclic vertices opposite the adjacency edge, and is symmetric via a reflection about the x-axis:
\begin{equation}
\begin{split}
%
\end{split}
\end{equation}
Now these two are equal after applying the horizontal triangle relator. Notice that if we hadn't included that relator, this calculation would have indicated it's existence. The next one is asymmetric with the bottom having the cyclic vertex adjacent to the connecting edge.
\begin{equation}
\begin{split}
%
 
\end{split}
\end{equation}
Next we have the case where both triangle relators have their cyclic vertex adjacent to the adjacency edge:
\begin{equation}
\begin{split}
%
\end{split}
\end{equation}

The last possibility for two adjacent reducible triangle relators is adjacency via the cyclic vertex, corresponding to a double edge.

\begin{equation}
\begin{split}
%
\end{split}
\end{equation}

and we see that the horizontal square relator falls out. Now there remains the case of reducible triangles adjacent to horizontal relators, but we already took care of this in calculations  (\ref{KekuleBigonDouble}), (\ref{KekulePureBigon}), and purity.

Finally, we have the case of a pair of horizontal relators:

\begin{align*}
%
\end{align*}
\end{enumerate}

Therefore we have confirmed all the criteria in the theorem, which shows us that the $A_3$ spider with the Morrison-Kim presentation is $E^1$ convergent. 
\end{proof}

\chapter{The Euclidean Building and Geometric Satake}
\linespread{1}\selectfont
This chapter has two related functions. The first is to apply the combinatorial results about the spiders from the previous sections to prove interesting geometric results regarding combinatorial disks with a fixed boundary in the Euclidean building. In particular, we prove some results relating web bases to the geometric Satake basis. The second function is to provide a geometric interpretation of some of the technical lemmas which on one hand motivates them, and on the other hand may point to the correct generalizations for future work. The second section on the $A_1^n$ spider is independent of the final two sections regarding the $A_3$ spider.

\section{Background}
\subsection{Buildings and \texorpdfstring{$CAT(0)$}{} Geometry}
As demonstrated in \cite{fkk:buildings}, webs have a natural geometric interpretation. Let $G$ be an algebraic group (eg: $SL_{n+1}$), then as a set the \tbf{affine Grassmannian} is $Gr_G =G(\C((t)))/G(\C[[t]])$, the actual geometric structure is as an ind-variety which is a categorical limit of varieties. In fact, if $G$ is the group we care about, we will be more interested in $Gr_{G^\vee}$ where $G^\vee$ is the Langlands dual group which is isomorphic to $G$ if $G$ is type $A$, $D$, $E$, $F$ or $G$, but if $G$ is $B_n$ then $G^\vee$ is $C_n$ and vice versa. We won't get too deep into this algebraic geometry point of view since we will be mainly working with a related geometric object: a polyhedral complex called the Euclidean building which we define below. The key connection is that there exists a injective map from the affine Grassmannian into the vertices of the Euclidean building whose image is the so-called \tbf{special vertices} of the Euclidean building (which we can just take as the definition). In the case of $A_n$, every vertex is special, so we won't need to worry about his distinction. 

There exists an evaluation map at $t=0$ $ev_0: G(\C[[t]]) \surj  G(\C)$. Let $B$ be any Borel subgroup of $G(\C)$ (eg: the upper triangular matrices in $SL_{n+1}(\C)$). Then $\wt{B}=ev_0^{-1}(B) \sbs G(\C[[t]])\sbs G(\C((t)))$ is an example of what is called a \tbf{Iwahori subgroup} which functions similarly to a Borel subgroup (more rigorously it is part of a $BN$-pair). More generally, all conjugates of $\wt{B}$ in $G(\C((t)))$ will be called an Iwahori subgroup. In our running example of upper triangular matrices in $SL_{n+1}$, $\wt{B}$ is the subgroup of determinant-$1$ matrices with $\C[[t]]$-coefficients, where all entries below the diagonal have coefficients in $t\C[[t]]$. 

Next we define an analogue of parabolic subgroups called parahoric subgroups. A \tbf{standard parahoric subgroup} is a subgroup of $G(\C((t)))$ which contains $\wt{B}$ and is contained in a finite union of double cosets $\wt{B}g\wt{B}$. We need this second condition unlike in the parabolic case to avoid including much larger subgroups and making our space too big. A \tbf{parahoric subgroup} is then any conjugate of a standard parahoric subgroup. In particular, Iwahori subgroups are examples of parahoric subgroups. We order these subgroups by reverse inclusion so that Iwahori subgroups are the "largest" in this ordering. This then gives an abstract polyhedral complex where a Iwahori subgroup $\wt{B}$ gives a top dimensional cell $C_{\wt{B}}$, subcells of $C_{\wt{B}}$ correspond to parahoric subgroups which contain $\wt{B}$, and points are given by maximal parahorics. We call this complex the \tbf{Euclidean building} and denote it $X$.

This complex satisfies the axioms of a Tits building for an affine Coxeter group, and has many nice properties some of which we list here.
\begin{enumerate}
	\item The Euclidean building is a locally finite polyhedral complex
	\item There exists a collection of subcomplexes homeomorphic to $\R^n$ called apartments such that any pair of cells are contained in an apartment.
	\item If $A$ and $A'$ are apartments containing a pair of cells as above, then there exists a combinatorial isomorphism $A \natiso A'$ such that the points in the two cells are fixed under this map.
	\item The link of a special vertex of the Euclidean building is isomorphic to the \tbf{spherical building} which by definition is the polyhedral complex whose cells are parabolic subgroups of $G$
\end{enumerate}

Beyond the naturally induced CW complex topology, it admits a metric where each top dimensional cell is isometrically isomorphic to a Weyl alcove of $G(\C)$, and each apartment is isometric to the weight space of $G$ (and hence isometric to $\R^n$ as a metric space). One of the key features is that it satisfies a generalization of the non-positively curved property of some smooth manifolds. This property is called the $CAT(0)$ property, and is defined as follows:

We will follow \cite{bh:curvature}. Let $Y$ be a metric space. In this context, a \tbf{geodesic} is a length minimizing path between two fixed points (which we should point out does not coincide with the definition for Riemannian manifolds). A \tbf{geodesic triangle} $\Delta(p,q,r)$ for points $p,q,r \in Y$ is a choice of geodesics $[p,q]$, $[q,r]$, and $[r,p]$ connecting each pair of points. The \tbf{comparison triangle} $\ol{\Delta}(p,q,r)$ is any geodesic triangle in $\R^2$ with vertices $\ol{p}$, $\ol{q}$, and $\ol{r}$ whose side lengths are equal to the lengths of $[p,q]$, $[q,r]$, and $[r,p]$. If $x$ is on $[p,q]$, then we denote by $\ol{x}$ the \tbf{comparison point} of $x$ which is the point on the geodesic $[\ol{p},\ol{q}]$ such that $d(\ol{x},\ol{p})=d(x,p)$ and $d(\ol{x},\ol{q})=d(x,q)$. We say $Y$ is $CAT(0)$ if for every geodesic triangle $\Delta(p,q,r)$, and $x \in [p,q]$, $d(x,r)<d(\ol{x},\ol{r})$, or intuitively: if triangles in $Y$ are thinner than triangles in $\R^2$. $CAT(0)$ spaces then inherit many of the properties of non-positively curved manifolds:

\begin{prop}[cf. \cite{bh:curvature}]
\begin{enumerate}
	\item For any two points in a $CAT(0)$ space, there is a unique geodesic segment joining the two points
	\item Any $CAT(0)$ space is contractible 
\end{enumerate}
\end{prop}

One of the main theorems about the geometry of buildings is:

\begin{thm}[cf. \cite{bh:curvature}]
The Euclidean building (with top dimensional cells metrized as Weyl alcoves) is a $CAT(0)$ space, hence contractible
\end{thm}

There's another kind of distance we can put on the Euclidean building. Given any two points $v$ and $w$, there exists an apartment $A$ containing both points. This apartment is isomorphic to the weight space, and we can choose the isomorphism $\ph_{A,v,w}$ to send $v$ to $0$, and $w$ into the Weyl chamber. Then $d_{wt}(v,w)=\ph_{A,v,w}(w)$ does not depend on the choices, and hence gives a weight-valued distance function \cite{kuperberg:rank}, which satisfies properties analogous to a real valued distance functions. We'll mainly only be interested in the distance between vertices.

\subsection{Dual Webs}
We now return to spiders, again following \cite{fkk:buildings}. A key insight in this paper was to look at the \textbf{dual web} $D(w)$, defined as follows. $w$ is a planar graph, and hence we can construct a dual graph whose points are the faces of $w$ and two vertices are connected by an edge if the corresponding faces are adjacent. Moreover, each vertex of the web corresponds to a $2$-cell in the weight space, so we can form a polyhedral complex with cells of dimension $\leq 2$. We denote this complex $D(w)$ and call it the dual web. See \cite{fkk:buildings} for more details. 

If we fix a polygon $P$ where each edge is labeled by a fundamental weight $\vec{\la}$, then we can define a space $Q(P)$ of based maps from the vertices of $P$ into the affine Grassmannian $Gr=G^\vee(\C((t)))/G^\vee(\C[[t]])$ such that if two vertices of $P$ are connected by an edge of length $\la$, then in the image they must be distance $\la$ from each other. We call this a \textbf{configuration space}, and in the same way, we can define the configuration space $Q(D(w))$ of maps from the vertices of $D(w)$ into $Gr$ (notice we are ignoring the $2$ dimensional structure). We will refer to the elements of this space as \textbf{configurations}. Furthermore, by restricting a configuration of $D(w)$ to the boundary, we get a configuration of $P$, hence giving a map $Q(D(w)) \to Q(P)$, which then extends to a map of homology: $H_\ast(Q(D(w)),\C)  \to H_\ast (Q(P))$. In fact, $Q(P)$ is the Satake fiber $F(\vec{\la})$ mentioned in the introduction, and hence $H_{top}(Q(P),\C)$ is naturally isomorphic to the invariant space of $V_{\la_1} \ot ... \ot V_{\la_n}$.

\section{The \texorpdfstring{$A_1^n$ spider and a $CAT(0)$}{ spider and a } Subcomplex of the Building}
One of the issues that we will try to address in this chapter is the failure of minimal vertex dual webs to have coherent geodesics. The other is the fact that there are more than one choice of equivalent minimal webs. 

If we look at a generic map from a dual web to the building $D(w) \to X$ the first issue amounts to the fact that, for fixed boundary points $v$ and $w$, there may not be any (weak) geodesics between $v$ and $w$ which stay in $D(w)$. We will first look at $A_1^n$ as a test case. Since it is geometrically a cube complex, this makes it easy to prove that it is $CAT(0)$. As a proposed solution to both problems, we will take all (dual) webs which are equivalent to $w$, and argue that they naturally combine into a $n$-dimensional subcomplex of the building which has coherent geodesics, is geometrically $CAT(0)$, and naturally includes all dual webs equivalent to $w$ as a subcomplex.

We will first restrict to a finite subcomplex of the building. Recall that the Euclidean building of a product of Lie groups is the product of the buildings. Let $w$ be a minimal vertex $A_1^n$ web. Then for the $i$th $A_1$ component of $A_1^n$ we get a minimal vertex $A_1$ web, whose dual graph is a tree which we denote $\Ga_i$. Next we give it the path-length metric where each edge is length $1$. Each face of $w$ is then naturally labeled by vertices in $Y := \prod\limits_{i=1}^n\Ga_i$, and this has the natural product path-length metric. If two vertices in $w$ are adjacent, then the corresponding vertices in $Y$ are adjacent because the label is only changing for one $A_1$ factor, and similarly squares in $D(w)$ correspond to squares in $Y$. Therefore, we get a combinatorial map $D(w) \to Y$. Moreover, $Y$ is the same for every equivalent web $w'$ since if the matchings for $w$ and $w'$ are the same, then all the $A_1$ components are too. Let $\om$ be the equivalence class of webs equal to $w$ in the spider $Sp$. We get a combinatorial map $\ph:\bigsqcup\limits_{w \in \om} D(w) \to Y$.

As an aside, let's relate $Y$ to the building $X$ directly. First choose an embedding $\psi: D(w) \to X$. Composing with the projections $\pi_i$ of $X$ onto the $i$th $A_1$ term we get embeddings $\psi_i$ of each $A_1$ component of $D(w)$ into the $A_1$ building. Then by the universal property of products we get a unique embedding $\wt{\psi}$ of $Y$ into the $A_1^n$ building such that $\psi_i=\pi_i \circ \psi$, where $\pi_i$ is the projection of the $A_1^n$ building onto the $i$th $A_1$ term. Therefore, we get a natural map of configuration spaces $Q(D(w)) \to Q(Y)$, which is an isomorphism with inverse obtained by restricting a configuration from $Y$ to a configuration from $D(w) \inj Y$.

We now precede to discuss our main geometric object of interest:

\begin{defn}
Given $w$, $Y$, and $\ph$ as above, the \textbf{pocket} of $w$ denoted $P=P_w$ is the subcomplex of $Y$ whose $0$-skeleton is the $0$-skeleton of $im(\ph)$ and whose cells of dimension $>0$ are all cells in $Y$ whose adjacent vertices are all contained in $im(\ph)$
\end{defn}

In fact, we get a bijection between $2$-skeletons as follows:

\begin{prop}
The map $\ph:\bigsqcup\limits_{w \in \om} D(w) \to P^{(2)}$ is surjective.
\end{prop}
\begin{remark}
We can construct a similar polyhedral complex in the other simple and semisimple cases, but in general we need to add in extra edges and faces.
\end{remark}
\begin{proof}
We get the result for the \tbf{$0$-skeleton} by definition, so we'll start with the \tbf{$1$-skeleton}. Let $v_1 \in D(w_1)$ and $v_2 \in D(w_2)$ where $w_1 \sim w_2$ are minimal webs, and $\ph(v_1)$ is adjacent to $\ph(v_2)$ via an edge $e \in P$. Denote by $f_i$ the face in $w_i$ corresponding to $v_i \in D(w_i)$. $e$ then corresponds to a strand $s$ in the two webs which without loss of generality we'll say is on the first $A_1$ component. If there is no such vertex $v$ on $D(w_1)$ such that $\ph(v)=\ph(v_2)$, then graphically this means that $s$ is not adjacent to the face $f_1$. 


We'll inductively show that we can move $s$ to cross $f_1$. First note that there exists a sequence of triangle relators $(r_i)$ sending $w_1$ to $w_2$ by the $E^1$ convergence of $A_1^n$ (\cref{E1:A1n}). 

In $w_1$ the face $f_1$ is on the opposite side of $s$ compared to the face labeled $f_2$ on $w_2$ since they are connected by an edge corresponding to $s$ in $P$. If $w$ is a web equivalent to $w_1$, let $\wt{w}$ be the web obtained from $w$ by omitting the strand $s$, and similarly let $\wt{P}$ be its pocket with map $\bigsqcup\limits_{w \in \om} D(\wt{w}) \to \wt{P}$. There is a projection $D(w) \surj D(\wt{w})$, and hence we can extend the pocket map to $\wt{\ph}:\bigsqcup\limits_{w \in \om} D(w) \to \wt{P}$. The key reason we defined all this is that $\wt{\ph}(v_1)=\wt{\ph}(v_2)$ since $v_1$ and $v_2$ differ only by the strand $s$.

If a face labeled $\wt{\ph}(v_1)=\wt{\ph}(v_2)$ exists on the web $r_k...r_1(w_1)$ for every $k$, then the strand must at some point cross one such face (since there is no single relator that allows a strand to jump over a face), and then we'd be done. So assume otherwise, and find the last web on $w_1':=r_k...r_1(w_1)$ such that the face labeled $\wt{\ph}(v_1)$ exists, and is on the same side of $s$ as $f_1$ is in $w_1$, and let $w_2':=r_\ell...r_1(w_1)$ be the first web after that where a face labeled $\wt{\ph}(v_1)$ exists (and hence $s$ is on the opposite side). 

Then on $w_1'$ and $w_2'$ the face labeled $\ph(v_1)$ and $\ph(v_2)$ must be triangles, since triangle relators only change the label on their associated face. Since the triangle face are on opposite sides of $s$, this shows that $s$ crosses all three strands $\lb s_1,s_2,s_3 \rb$ associated to this relator (otherwise one of the webs would be reducible). Up to rotation and reflection in $w_1'$ this looks something like:
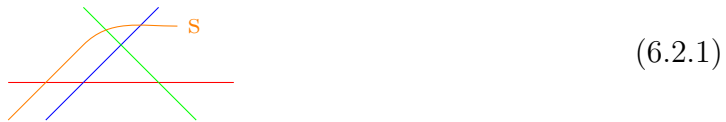
\begin{equation}
\begin{tikzpicture}[baseline=1.5ex,scale=1] 
\draw[red] (0,0) -- (1,0); 
\draw[red] (1,0) -- (2,0);
\draw[red] (2,0) -- (3,0);

\draw[blue] (0.5,-.5) -- (1,0); 
\draw[blue] (1,0) -- (1.5,.5);
\draw[blue] (1.5,.5) -- (2,1);

\draw[green] (2.5,-.5) -- (2,0); 
\draw[green] (2,0) -- (1.5,.5);
\draw[green] (1.5,.5) -- (1,1);

\draw[orange] (0,-.5) -- (1,.5) to[out=45,in=180] (2.25,.75); 
\node[right,orange] at (2.25,.75) {s};
\end{tikzpicture} 
\end{equation}

One can directly check that $s$ must make a triangle with one of the outer corners of the triangle $s_1s_2s_3$, which without loss of generality will be $ss_1s_2$. Now by \cref{A1n:singlestrand}, if there are no strands parallel to $s$ in an $\ep$-neighborhood of $ss_1s_2$, we can move $s$ onto the face label $\ph(v_1)$, and then we'd be done. So assume otherwise: let $s'$ be a parallel strand that is minimally close to $s$ in the sense that it crosses one of the two other strands in the outer triangle closest to $s$.
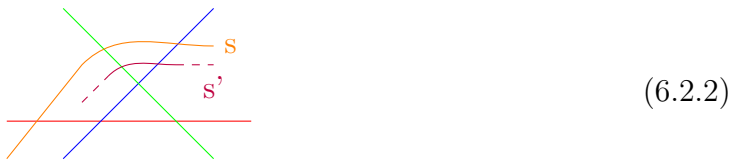
\begin{equation}
\begin{tikzpicture}[baseline=1.5ex,scale=1] 
\draw[red] (-.25,0) -- (1,0); 
\draw[red] (1,0) -- (2,0);
\draw[red] (2,0) -- (3,0);

\draw[blue] (0.5,-.5) -- (1,0); 
\draw[blue] (1,0) -- (1.5,.5);
\draw[blue] (1.5,.5) -- (2.5,1.5);

\draw[green] (2.5,-.5) -- (2,0); 
\draw[green] (2,0) -- (1.5,.5);
\draw[green] (1.5,.5) -- (.5,1.5);

\draw[orange] (-.25,-.5) -- (.75,.75) to[out=45,in=180] (2.5,1); 
\node[right,orange] at (2.5,1) {s};
\draw[purple] (1.1,.6) to[out=45,in=180] (2,.75);
\draw[purple, dashed] (.75,.25) -- (1.1,.6);
\draw[purple, dashed] (2,.75) -- (2.5,.75) ;
\node[below,purple] at (2.5,.75) {s'};
\end{tikzpicture} 
\end{equation}

$s'$ must cross $s$ globally, otherwise in $w_2'$, $s'$ would still be on the same side of $s$, and hence there would be no face labeled $\ph(v_1)$. Considering them as cut paths on the web of all other strands, this shows that they intersect in the dual web twice, and hence shifting one segment to the other would still be a minimal cut path. Therefore, by \cref{A1n:singlestrand}, there is a sequence of relators moving $s$ to $s'$ such that there is no vertices in the interior of the bounded region of $s$, $s'$, and one of the strands in the triangle relator. 

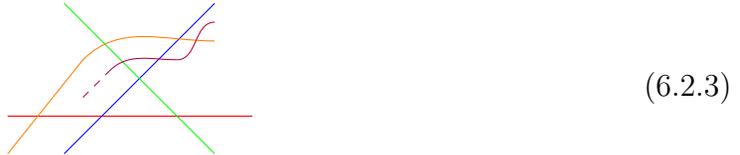
\begin{equation}
\begin{tikzpicture}[baseline=1.5ex,scale=1] 
\draw[red] (-.25,0) -- (1,0); 
\draw[red] (1,0) -- (2,0);
\draw[red] (2,0) -- (3,0);

\draw[blue] (0.5,-.5) -- (1,0); 
\draw[blue] (1,0) -- (1.5,.5);
\draw[blue] (1.5,.5) -- (2.5,1.5);

\draw[green] (2.5,-.5) -- (2,0); 
\draw[green] (2,0) -- (1.5,.5);
\draw[green] (1.5,.5) -- (.5,1.5);

\draw[orange] (-.25,-.5) -- (.75,.75) to[out=45,in=180] (2.5,1); 
\draw[purple] (1.1,.6) to[out=45,in=180] (2,.75);
\draw[purple, dashed] (.75,.25) -- (1.1,.6);
\draw[purple] (2,.75) to[out=0,in=180] (2.5,1.25) ;
\end{tikzpicture} 
\end{equation}

Applying triangle relators to cross all dividing strands, we can then move the crossing of $s$ and $s'$ so that $s'$ is no longer parallel to $s$ in the outer triangle $ss_1s_2$. Applying induction on the number of strands parallel in an $\ep$-neighborhood, we can move $s_1$ onto the face labeled $\ph(v_1)$ only using relators involving $s_1$, and hence create a pair of adjacent faces corresponding to $e$.

Now we'll look at the \tbf{$2$-skeleton}. Given a square in $P$, fix a vertex $v$, and adjacent edges $e_1$ and $e_2$. We have two associated strands $s_1$ and $s_2$ which cross in all webs equivalent to $w$. Similar to before we can construct a web $\wh{w}$ by omitting $s_1$ and $s_2$, a corresponding pocket $\wh{P}$, and a corresponding map $\wh{\ph}: \bigsqcup\limits_{w \in \om} D(w) \to \wh{P}$. Then the $3$ points which are adjacent to either $e_1$ or $e_2$ have the same image under $\wh{\ph}$. By the above proof, we can move $s_1$ and $s_2$ to cross the subweb $f=\wh{\ph}^{-1}(\wh{\ph}(v)) \cap D(w)$ associated to this vertex, however $s_1$ and $s_2$ may not cross on this subweb, in which case we wouldn't get the desired $2$-cell. However, we know the corresponding cut-paths on $\wh(w)$ do cross twice: once at the crossing of the strands, and once inside $f$. Since the web is reduced, the strands follow minimal cut-paths, which means that if we replace the segment of $s_1$ between these two crossings with the corresponding segment of $s_2$, we still get a reduced web. At this point, since there are no vertices between $s_1$ and $s_2$, we can apply triangle relators to move the crossing into $f$, giving us the desired result.
\end{proof}

As an immediate corollary of the second part of this proof, we get something stronger that will be used to tell us more about the link of vertices in $P$.

\begin{coro}
\label{A1n:SquareComplete}
Let $v \in P_w$, and let $v_1$ and $v_2$ be vertices adjacent to $v$ via edges which are associated to strands $s_1$ and $s_2$ which intersect in $w$. Then there exists a minimal vertex web $w'$ where the crossing of $s_1$ and $s_2$ is adjacent to faces with images $v$, $v_1$, and $v_2$ in $P$.
\end{coro}

$Y$ itself is $CAT(0)$ because it is a product of $CAT(0)$ spaces, but this subcomplex is not convex so we don't immediately get that $P$ is $CAT(0)$ (this was also an issue in $A_2$). However, $P$ can still be shown to be $CAT(0)$, which is the main result of this section. We will use a well known results about cube complexes to prove this, but first we need some related definitions. An abstract simplicial complex is called a \tbf{flag complex} if every clique in the $1$-skeleton spans a simplex in the complex \cite{bh:curvature}.

\begin{thm}
\label{A1n:cubeCAT}
A finite-dimensional simply-connected cube complex is $CAT(0)$ if and only if the link of each vertex is a flag complex \cite{bh:curvature}.
\end{thm}

We will use this to prove the following result from the introduction:

\CAT*
\begin{proof}
First we'll prove that the link of every vertex is a flag complex. Fix an arbitrary vertex $v$ in $P$, and take a set of pairwise adjacent vertices in the (geometric) link $lk_P(v)$ with corresponding edges $e_1,...,e_n$ in $P$. By assumption, the corresponding strands $s_1,...s_n$ pairwise intersect. There is a corresponding $n$-cube $C$ in the product space $Y$, and we'll show that this $n$-cube is contained in $P$ which follows almost immediately from \cref{A1n:SquareComplete}. To do this, we'll argue by induction that all $k$-cells adjacent to $v$ in $C \sbs Y$ are in $P$. Let $S$ be a $k+1$-cell in $C$ adjacent to $v$. All of the $k$-cells adjacent to $v$ are already in $P$, and the vertices in these $k$-cells opposite to $v$ are precisely the vertices of $S$ adjacent to vertex opposite of $v$ in $S$ which we'll label $v'$. We know from the above that all vertices besides $v'$ must be in $P$. Take a square adjacent to $v'$. The two edges that are not adjacent to $v'$ in this square are contained in $P$, and thus by \cref{A1n:SquareComplete}, $v'$ is in $P$, and so all of $C$ is in $P$ by construction. But this proves that the link of $v$ is a flag complex.

Now we need to show $P$ is simply-connected. Fix a boundary vertex $\ast$ as the base point, and let $p \in \pi_1(P,\ast)$ be a loop. By van Kampen it is sufficient to assume that $p$ is a combinatorial path $(v_i)_{i=1}^n$ in $P^{(2)}$, and let $d_i = d(\ast,v_i)$. We will apply induction on $\sum d_i$. Each edge either goes away from $\ast$ or towards it in one of the quotient trees, and hence $d_i \neq d_{i+1}$ for all $i$. Let $d_k$ be a maximal value, so in particular $d_k>d_{k-1}$ and $d_k>d_{k+1}$. The edges $(v_{k-1},v_k)$ and $(v_{k},v_{k+1})$ correspond to some strands $s$ and $s'$. If $s=s'$, then these are the same edge, so we can omit the pair and apply induction. If $s \neq s'$, then these strands must intersect, otherwise $d_{k+1}>d_k$ since all paths would have to cross $s$ in order to cross $s'$. But \cref{A1n:SquareComplete} says that the corresponding edges span a square in $P$, and by homotoping over this square we get a loop $p'$ whose new distances $(d_i')$ are the same except $d_k'=d_k-2$, and hence we can apply induction.
\end{proof}

To put everything together, we have complexes $D(w) \sbs P \sbs Y$ whose configuration spaces into the Euclidean building are all naturally isomorphic, such that $P$ and $Y$ are $CAT(0)$, and $P^{(2)}$ is exactly the union of the dual webs of all webs equivalent to $w$.

\section{The \texorpdfstring{$A_3$}{} Spider and the Geometric Satake Basis}
We will use the concept of coherent webs defined in \cite{fontaine:generating} to prove an upper triangularity result. Unlike the previous sections, we will be exploiting quite a few nice properties of the $A_3$ spider presentation to do this, and so many of these results are unlikely to generalize easily. 

Let $D(w)$ be the dual graph of $w$ defined in \cite{fkk:buildings} where each vertex $D(w)$ is a face of $w$. $D(w)$ has a weight-valued distance function given by adding up the weights of each edge in a path. In particular, cut paths in $w$ correspond to paths in $D(w)$ between boundary vertices, and minimal cut paths correspond to maximal geodesics, which we are just defined as minimal length path between two points in $D(w)$.

Following \cite{fontaine:generating} we define coherent $A_n$ webs:
\begin{defn}
\label{CoherenceDef}
We say a web $w$ is \text{coherent} at a boundary face $\ast$ if
\begin{enumerate}
	\item Any geodesic in $D(w)$ from $\ast$ to some vertex $a$ are equal length (hence there is a well defined notion of distance from $\ast$ and any vertex)
	\item Every vertex in $D(w)$ is crossed by some geodesic from $\ast$ to another boundary vertex
	\item If vertices $a$ and $b$ of $D(w)$ are separated by a edge of weight $\la$, then the difference in length from $\ast$ to $a$ and $\ast$ to $b$ is in the Weyl orbit $W \la$ of $\la$.
\end{enumerate}
\end{defn}
\begin{remark}
As we'll see in proposition \ref{Rel:Redundant}, the third axiom is redundant.
\end{remark}

This definition is essentially the exact hypotheses used in \cite{fkk:buildings} to prove upper unitriangularity for $SL_3(\C)$. Every coherent web then gives us an LS path by taking all cut paths from $\ast$ to each boundary point. Subsequently, we get the following theorem from \cite{fontaine:generating}

\begin{thm}
\label{coherent}
Any set of coherent webs which is in bijection with LS paths (of the corresponding type given by the boundary) forms a basis for the invariant space which is upper unitriangular with respect to the geometric Satake basis in terms of the natural partial ordering.
\end{thm}

To apply this, we will argue that any application of a hexagon relator or a reduction relator respects coherency on at least one of the resulting webs. First we want to discuss convexity properties of cut paths on webs.

Let $X$ be a $2$-D polyhedral complex with directed edges colored by a set $S$. Next give $\Z^S$ some partial ordering. We say $X$ is \tbf{combinatorially convex} at a vertex $v$ if for every pair of vertices $w,w' \in X$ connected by an edge $e$, there exists geodesics $\ga$ and $\ga'$ from $v$ to $w$ and $v$ to $w'$ respectively, such that there are no vertices in the open region bounded by $\ga$, $\ga'$, and $e$ (i.e. the homological inside).

\begin{prop}
\label{Rel:cvx}
Let $w$ be a web satisfying the first two axioms of a coherent web. Then $D(w)$ is combinatorially convex at $\ast$, and moreover the corresponding geodesics can be extended to maximal geodesics
\end{prop}
\begin{proof}
Let $v$ and $v'$ be adjacent vertices. By coherence, we can find a maximal geodesic $\ga$ from $\ast$ through $v$ and $\ga'$ from $\ast$ through $v'$. Now if $\ga'$ ever crosses to the side of $\ga$ opposite of $v'$ before $v'$, we can choose to follow $\ga$ instead. More precisely, if $\ga'$ crosses $\ga$ at points $x_1,x_2,...x_n$, then we can replace any segment $[x_i,x_{i+1}]_{\ga'}$ of $\ga'$ which is to the right of $\ga'$ with the corresponding segment $[x_i,x_{i+1}]_{\ga}$ of $\ga$. Hence we can take $\ga'$ to be on the same side of $\ga$ as $v'$ up to $v'$, and hence the two geodesic segments from $\ast$ to $v$ and $\ast$ to $v'$ together with the edge connecting $v$ and $v'$ bound a region.

If $w$ is a vertex in the open region bounded by the two geodesics, we can find a maximal geodesic $\rho$ passing through it from $\ast$. By replacing some segments as above, we can have $\rho$ always be (non-strictly) between the two geodesics $\ga$ and $\ga'$. Therefore, it must pass through either $v$ or $v'$, and thus we can replace either $\ga$ or $\ga'$ with it and reduce the number of bounded vertices. Applying induction, we have our result.
\end{proof}

We will use this fact to conclude that the Fontaine's third axiom is redundant:

\begin{prop}
\label{Rel:Redundant}
If an $A_n$ web satisfies the first two axioms of a coherent web, then it is coherent.
\end{prop}
\begin{remark}
One might wonder whether we can remove one of the other two axioms, but it isn't too hard to construct $A_3$ webs with one axiom but not other. I'll quickly sketch the constructions. First, the web corresponding to an $I=H$ relator fails axiom $1$ if we put $\ast$ so that it is not adjacent to the double edge. However, it has no interior vertices, so it fulfills axiom $2$. Next, take the square from the horizontal square relator. Attach a trivalent vertex to one of the boundary double edges, oriented such that there is no $I=H$ relator, and let $\ast$ be the newly added boundary face. This fulfills axiom $1$, but not axiom $2$ since no geodesic passes through the interior vertex.
\end{remark}
\begin{proof}
Let $v$ and $v'$ be faces in the web. By combinatorial convexity (Proposition \ref{Rel:cvx}), we know there exists geodesics $\ga$ and $\ga'$ to $v$ and $v'$ respectively such that the open bounded region contains no vertices, and moreover let us choose the pair of geodesics with the fewest number of edges between them. I claim that if $w$ and $w'$ are vertices on $\ga$ and $\ga'$ (respectively), such that they are adjacent by an edge $e$ of dominant weight $\la_{i}$ for some $i$, then $d(\ast,w) - d(\ast,w') \in W \la_i$. We'll prove this by induction on the number of edges in the open region bounded by $\ga$, $\ga'$, and $e$. If there is only one, then this follows from the fact that each single triangle is coherent. Otherwise, there is a last interior edge $e'$ of weight $\la_{j}=L_1+...+L_j$ in the bounded region. Without loss of generality, we can assume that $e$ and $e'$ share a vertex $w$ on $\ga$. In $\ga'$, let $\la_{k}$ be the weight of the edge connecting the other vertices $w'$ and what we'll call $w''$ of $e'$ and $e''$ respectively. And again, by duality symmetry, let the triangle $ee'e''$ be oriented clockwise.
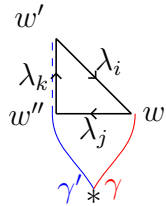
\begin{equation}
\begin{tikzpicture}[baseline=-.3ex,scale=1]  
	\node at (0, -.1)  (*)     {$\ast$}; 
	\node at (.2, 1.7)  (i)     {$\la_i$};
	\node at (0, .75)  (j)     {$\la_j$};
	\node at (-.8, 1.5)  (k)     {$\la_k$};
	\draw[blue] (0,0)node[anchor=east]{$\ga'$} to[out=120,in=-90] (-.55,1);
	\draw[red] (0,0)node[anchor=west]{$\ga$} to[out=60,in=-90] (.55,1);
	\draw[blue,dashed] (-.55,1) -- (-.55,2);
	
	\draw[web,midfrom] (-.5,1)node[anchor=east]{$w''$} -- (.5,1)node[anchor=west]{$w$};
	\draw[web,midfrom] (.5,1) -- (-.5,2)node[anchor=south east]{$w'$};
	\draw[web,midfrom] (-.5,2) -- (-.5,1);
\end{tikzpicture}
\end{equation}
We know the following:
\begin{enumerate}
	\item  $i+j+k=n+1$ by the presentation of the $A_n$ spider
	\item  $d(\ast,w)-d(\ast,w'') \in W \la_{n+1-j}$ by induction
	\item  $d(\ast,w)-d(\ast,w'') \in \la_{k} - \la_{n+1-i} + \Delta^+$. This is equivalent to $d(\ast,w)+\la_{n+1-i}-d(\ast,w') \in \Delta^+$ (ie is a positive root), which we know is the true because the path following $\ga$ then $\la_i$ to get to $w'$ and the path following $\ga'$ to $w'$ should have comparable lengths, and they can't be equal length otherwise we could have found geodesics with fewer bounded edges by choosing both geodesics to follow one side or the other.
\end{enumerate}
We know $\la_{k} - \la_{n+1-i}= -L_{k+1}-...-L_{n+1-i}$, and from the last two items we know it is in $\la_{k} - \la_{n+1-i} \in W \la_{n+1-i} - \Delta^+=W \la_{n+1-i} + \Delta^-$. Hence, we want to determine the ways $-L_{k+1}-...-L_{n+1-i}$ can be written as an element of $W \la_{n+1-i} + \Delta^-$. Since $W$ is the symmetric group, $W \la_{n+1-i}$ is the set of sums of $n+1-i$ distinct $L_s$'s. Also recall that $\Delta^+=\lb L_s - L_t : s < t \rb$. There are four kinds of elements in $W \la_{n+1-i} + \Delta^-$: 1) the ones where the $L_s$'s of the root and the $L_s$'s of the weight are disjoint, 2) the ones where only the negative $L_s$ from $\Delta^-$ overlaps, 3) the ones where only the positive $L_s$ from $\Delta^-$ overlaps, and 4) the ones that where both the positive and negative $L_s$ from $\Delta^-$ overlaps with the $L_s$'s from $W \la_{n+1-i}$. Since adding integer multiples of $L_1+...+L_{n+1}$ doesn't change the differences of coefficients of $L_s$'s, we know that (1), (3), and (4) can't happen because there exists an $L_s$ and $L_t$ whose coefficients differ by $2$, which don't exist in elements of $W \la_{n+1-i}$. Therefore it must be of type (2). 

The relevant elements are thus the sum of $n+1-i$ distinct $L_s$'s. The element must be the sum of $L_s$ distinct from $\lb L_{k+1},...,L_{n+1-i} \rb$ so that when we add $L_1+...+L_{n+1}$ to $\la_{k} - \la_{n+1-i}$ we get equality. So the element is $L_1+...+L_k + L_{n+1-i+1}+...+L_{n+1}$. Next, we want to figure out what the root could have been. Negative roots are of the form $L_s-L_t$ with $s>t$. If I subtract one and get something in $W \la_{n+1-i}$ we see that $t>k$, so $s>k$, and hence the element from $W \la_{n+1-i}$ has $\la_k$ as a summand, which tells us that $d(\ast,w) - d(\ast,w'') \in \la_k +W\la_{i}$, so $d(\ast,w) - d(\ast,w') \in W\la_{i}$
\end{proof}

We will now return to the $A_3$ spider. For simplicity, we will look at the contracted spider, and allow cut paths to cross through tetravalent vertices which correspond to $I=H$ relators, or to double edges in the non-contracted spider that can be crossed.

\begin{lemma}
\label{Rel:StrandCross}
If $w$ is a coherent web with respect to $\ast$, then a geodesic cut path from $\ast$ can cross a strand at most once (where crossing a tetravalent vertex, or a double edge adjacent to two strands counts as crossing both strands).
\end{lemma}
\begin{proof}
The ($\Z$-linear) partial ordering on dominant weights is generated by $\la_1 + \la_3 \leq 2 \la_2$, $\la_2 \leq 2\la_1$, and $\la_2 \leq 2\la_3$. Define a function $\ph_s:\Z[\la_1,\la_2,\la_3] \to \Z$ by $\ph_s(\la_1)=\ph_s(\la_3)=1$, and $\ph_s(\la_2)=2$. Then if $x \leq y$ for dominant weights $x$ and $y$, we get $\ph_s(x) \leq \ph_s(y)$. In terms of the contracted web, composing this map with the length of a cut path gives the number of undirected strands crossed by the cut path, where crossing a tetravalent vertex counts as crossing both strands. This proves that any geodesic cut path must cross a minimal number of undirected strands. 

Now assume by contradiction that there's a geodesic from $\ast$ which crosses a strand twice. If there are any internal faces in the corresponding bigon, then we can find a geodesic from $\ast$ passing through it. If this new geodesic crosses the other cut path twice, we can reduce the number of contained faces by replacing a segment of the original cut path with the new cut path. If it crosses the cut path and the strand we can replace the end of the original cut path with the new cut path and get fewer faces in between. Finally, if it crosses the strand twice, we can just replace the original cut path with the new cut path. Hence, we can assume that there are no faces between the cut path and the strand, or in other words there are just some number of other strands passing straight through the constructed bigon. This implies that we can shift the geodesic follow the other side of the strand and so cross fewer strands, and hence the original cut path wasn't a geodesic.
\end{proof}

In particular this shows that a geodesic cut path can't cross the same $I=H$ tetravalent vertex twice, which shows that geodesic cut paths correspond in $S(w)$ correspond to a geodesic cut path in some $w'$ with $S(w')=S(w)$, and moreover this $w'$ can be chosen to be the same for all cut paths coming from $\ast$: indeed $\ast$ is in one of the four quadrants cut out by the two (undirected) strands, and the only geodesics which can cross the tetravalent vertex are ones that go to the opposite quadrant.

\begin{lemma}
\label{Rel:summand}
Let $w$ be a contracted $A_3$ spider with a reduction relator or horizontal triangle relator $r$ on $w$. Then there is a unique summand web $w'$ in $r$ which is coherent, and who's maximal geodesic cut path distances from $\ast$ are the same as $w$
\end{lemma}
\begin{proof}
First we notice that the loop relators, and relators (\ref{Pure Bigon C}) and (\ref{Mixed Bigon C}) can't exist in a coherent web since any cut path crossing the face has a short cut by going around the face. This leaves relators (\ref{Double Square C}), (\ref{Single Square C}), and the triangle relator (\ref{Kekule C}). 

Next we'll prove that there is summand web $w'$ whose distances between $\ast$ and boundary faces are equal to $w$, notice that each of the remaining three local relators can only decrease distances between its boundary faces, and each geodesic cut path on a leading order term has a corresponding equal length geodesic cut path on one of the other terms:
\begin{equation}
\begin{tikzpicture}[baseline=-.3ex,scale=.75] 
\draw[midto,web] (-1,.5) -- (0,0);
\draw[midfrom,web] (-1,-.5) -- (0,0);

\draw[web, midfrom]  (0,0) to[out=30,in=150] (2,0);
\draw[web, midto]  (0,0) to[out=-30,in=-150] (2,0);

\draw[midto,web] (2,0) -- (3,.5);
\draw[midfrom,web] (2,0) -- (3,-.5);

\draw[blue, dashed] (1.2,-1) -- (1.2,1); 
\end{tikzpicture}\;=\; [2]_q
\begin{tikzpicture}[baseline=4ex,scale=.75] 
\draw[web, midfrom]  (0,0) to[out=60,in=-60] (0,2); 
\draw[web, midto]  (1.5,0) to[out=120,in=-120] (1.5,2);

\end{tikzpicture}\;+\; 
\begin{tikzpicture}[baseline=-4ex,scale=.75] 
\draw[web, midto]  (0,0) to[out=-30,in=-150] (2,0); 
\draw[web, midfrom]  (0,-1.25) to[out=30,in=150] (2,-1.25);
\draw[blue, dashed] (1.2,-1.5) -- (1.2,0); 
\end{tikzpicture}
\end{equation}

\begin{equation}
\begin{tikzpicture}[baseline=-1ex,scale=1,transform shape, rotate=180] 
\draw[web] (0,0) -- (1,0); 
\draw[midto,web] (1,0) -- (2,0);
\draw[web] (2,0) -- (3,0);

\draw[web] (0.5,-.5) -- (1,0); 
\draw[midto,web] (1,0) -- (1.5,.5);
\draw[midto,web] (1.5,.5) -- (2,1);

\draw[web] (2.5,-.5) -- (2,0); 
\draw[midfrom,web] (2,0) -- (1.5,.5);
\draw[midfrom,web] (1.5,.5) -- (1,1);
\draw[blue, dashed] (1.3,-.6) -- (1.6,1.1); 
\draw[green, dashed] (0.5,.35) -- (2.5,.35); 
\end{tikzpicture}\;=\;
\begin{tikzpicture}[baseline=-.3ex,scale=.5] 
\draw[midfrom,web] (0,-1) -- (0,0);
\draw[web] (0,0) -- (-1,0);
\draw[web] (0,0) -- (0,1);
\draw[midto,web] (0,0) -- (2,0);
\draw[web] (2,0) -- (3,0);
\draw[web] (2,0) -- (2,1);
\draw[midfrom,web] (2,0) -- (2,-1);
\draw[green, dashed] (-.75,-.7) -- (2.75,-.7); 
\end{tikzpicture}\;+\; 
\begin{tikzpicture}[baseline=-2ex,scale=.75] 
\draw[web] (-1,-.5) -- (0,0);
\draw[web] (0,0) -- (1,.5);
\draw[web] (1,-.5) -- (0,0);
\draw[web] (0,0) -- (-1,.5);
\draw[web, midfrom]  (-1,-1) to[out=30,in=150] (1,-1);

\draw[blue, dashed] (-.2,-1.1) -- (.2,1.1); 
\end{tikzpicture}
\end{equation}

\begin{equation}
\begin{tikzpicture}[baseline=-1ex,scale=1,transform shape, rotate=180] 
\draw[midfrom,web] (0,0) -- (1,0); 
\draw[midto,web] (1,0) -- (2,0);
\draw[web] (2,0) -- (3,0);

\draw[midfrom, web] (0.5,-.5) -- (1,0); 
\draw[midto,web] (1,0) -- (1.5,.5);
\draw[midto,web] (1.5,.5) -- (2,1);

\draw[web] (2.5,-.5) -- (2,0); 
\draw[midfrom,web] (2,0) -- (1.5,.5);
\draw[midfrom,web] (1.5,.5) -- (1,1);
\draw[red, dashed] (0.5,-0.25) -- (2.25,.6); 
\end{tikzpicture}\;=\;
\begin{tikzpicture}[baseline=-.3ex,scale=.5] 
\draw[midfrom,web] (0,-1) -- (0,0);
\draw[web] (0,0) -- (-1,0);
\draw[web] (0,0) -- (0,1);
\draw[midto,web] (0,0) -- (2,0);
\draw[midto,web] (2,0) -- (3,0);
\draw[midto,web] (2,0) -- (2,1);
\draw[midfrom,web] (2,0) -- (2,-1);
\draw[red, dashed] (3,.75) -- (1,.75) to[out=-180,in=45] (0,0) --(-.75,-.75); 
\end{tikzpicture}\;+\; 
\begin{tikzpicture}[baseline=-2ex,scale=.75] 
\draw[web] (-1,-.5) -- (0,0);
\draw[midto,web] (0,0) -- (1,.5);
\draw[midfrom,web] (1,-.5) -- (0,0);
\draw[web] (0,0) -- (-1,.5);
\draw[web, midfrom]  (-1,-1) to[out=30,in=150] (1,-1);

\end{tikzpicture}
\end{equation}

\begin{equation}
\begin{tikzpicture}[baseline=1.5ex,scale=1] 
\draw[midto,web] (0,0) -- (1,0); 
\draw[midto,web] (1,0) -- (2,0);
\draw[midto,web] (2,0) -- (3,0);

\draw[midfrom,web] (0.5,-.5) -- (1,0); 
\draw[midfrom,web] (1,0) -- (1.5,.5);
\draw[midfrom,web] (1.5,.5) -- (2,1);

\draw[midto,web] (2.5,-.5) -- (2,0); 
\draw[midto,web] (2,0) -- (1.5,.5);
\draw[midto,web] (1.5,.5) -- (1,1);

\draw[blue, dashed] (1.3,-.6) -- (1.6,1.1); 
\draw[red,dashed] (1.75,-.5) to[out=90,in=225] (2.2,.7);
\end{tikzpicture} \; + \;
\begin{tikzpicture}[baseline=-.3ex,scale=.5,transform shape, rotate=180]  
\draw[midto,web] (-0.87,-2) to[out=90,in=0] (-2.47,0);
\draw[midto,web] (-1.87,1.5) to[out=-30,in=210] (1.47 , 1.5);
\draw[midto,web]  (2.07 , 0) to[out=180,in=90]  (0.6,-1.87);
\end{tikzpicture}
 \;=\;
\begin{tikzpicture}[baseline=-1ex,scale=1,transform shape, rotate=180] 
\draw[midfrom,web] (0,0) -- (1,0); 
\draw[midfrom,web] (1,0) -- (2,0);
\draw[midfrom,web] (2,0) -- (3,0);

\draw[midto,web] (0.5,-.5) -- (1,0); 
\draw[midto,web] (1,0) -- (1.5,.5);
\draw[midto,web] (1.5,.5) -- (2,1);

\draw[midfrom,web] (2.5,-.5) -- (2,0); 
\draw[midfrom,web] (2,0) -- (1.5,.5);
\draw[midfrom,web] (1.5,.5) -- (1,1);

\draw[blue, dashed] (1.3,-.6) -- (1.6,1.1); 
\draw[red,dashed] (-.1,-.4) -- (1.4,1.1);
\end{tikzpicture}
\; + \;
\begin{tikzpicture}[baseline=-.3ex,scale=.5]  
\draw[midfrom,web] (-0.87,-2) to[out=90,in=0] (-2.47,0);
\draw[midfrom,web] (-1.87,1.5) to[out=-30,in=210] (1.47 , 1.5);
\draw[midfrom,web]  (2.07 , 0) to[out=180,in=90]  (0.6,-1.87);
\end{tikzpicture}
\end{equation}
We want to know whether there could be a geodesic on $w'$ with a shorter or incompatible length compared to the above paths $\rho$ going from $\ast$ to some boundary face $b$. We know that any cut path on $w'$ that crosses along the boundary has a corresponding cut path on $w$ of the same length, and hence has a distance equal to or less than the indicated path on $w'$. Therefore, the only thing that could go wrong is if there were a geodesic cut path $\ga$ in $w'$ between the same faces as $\rho$ which would have needed to go through the face in $w$ (and so doesn't correspond to a geodesic in $w$). $\ga$ decomposes into three paths: $\ga_1$ goes from $\ast$ to the boundary of $r$, $\ga_2$ is contained in $r$, and $\ga_3$ goes from the boundary of $r$ to another boundary face. $\ga_1$ and $\ga_3$ have corresponding cut paths in $w$ which must be geodesics. However, it's easy to see that every above geodesic which crosses through a reduction relator can be modified into a geodesic which goes along the boundary without crossing the face. Hence we can replace either $\ga_1$ or $\ga_3$ with the corresponding geodesic in $w'$, and therefore we can reduce to the case when $\ga$ and $\rho$ don't cross transversely. Since there aren't many, we can simply look at each case separately:
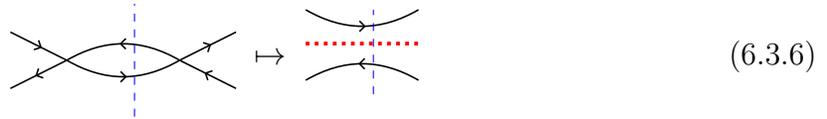
\begin{equation}
\begin{tikzpicture}[baseline=-.3ex,scale=.75] 
\draw[midto,web] (-1,.5) -- (0,0);
\draw[midfrom,web] (-1,-.5) -- (0,0);

\draw[web, midfrom]  (0,0) to[out=30,in=150] (2,0);
\draw[web, midto]  (0,0) to[out=-30,in=-150] (2,0);

\draw[midto,web] (2,0) -- (3,.5);
\draw[midfrom,web] (2,0) -- (3,-.5);

\draw[blue, dashed] (1.2,-1) -- (1.2,1); 
\end{tikzpicture} \; \mapsto \;
\begin{tikzpicture}[baseline=-4ex,scale=.75] 
\draw[web, midto]  (0,0) to[out=-30,in=-150] (2,0); 
\draw[web, midfrom]  (0,-1.25) to[out=30,in=150] (2,-1.25);
\draw[blue, dashed] (1.2,-1.5) -- (1.2,0); 
\draw[red, dotted,line width=.5mm] (0,-0.6) -- (2,-.6); 
\end{tikzpicture}
\end{equation}
In the bigon case, we can deform the dashed blue cut path on the leading term to cross any boundary faces (the left and right ones in the diagram). Therefore, $\ga_1$ and $\ga_3$ can be replaced with the corresponding segments of $\rho$, which is all of $\ga$, so this tells us that the length of the two strands are the same. Next we have the following:

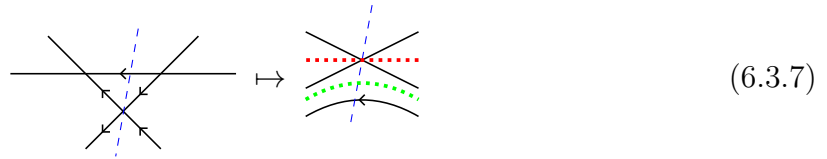
\begin{equation}
\begin{tikzpicture}[baseline=-1ex,scale=1,transform shape, rotate=180] 
\draw[web] (0,0) -- (1,0); 
\draw[midto,web] (1,0) -- (2,0);
\draw[web] (2,0) -- (3,0);

\draw[web] (0.5,-.5) -- (1,0); 
\draw[midto,web] (1,0) -- (1.5,.5);
\draw[midto,web] (1.5,.5) -- (2,1);

\draw[web] (2.5,-.5) -- (2,0); 
\draw[midfrom,web] (2,0) -- (1.5,.5);
\draw[midfrom,web] (1.5,.5) -- (1,1);
\draw[blue, dashed] (1.3,-.6) -- (1.6,1.1); 
\end{tikzpicture} \; \mapsto \; 
\begin{tikzpicture}[baseline=-2ex,scale=.75] 
\draw[web] (-1,-.5) -- (0,0);
\draw[web] (0,0) -- (1,.5);
\draw[web] (1,-.5) -- (0,0);
\draw[web] (0,0) -- (-1,.5);
\draw[web, midfrom]  (-1,-1) to[out=30,in=150] (1,-1);

\draw[blue, dashed] (-.2,-1.1) -- (.2,1.1); 
\draw[green, dotted,line width=.5mm] (-1,-.7) to[out=30, in=150] (1,-.7); 
\draw[red, dotted,line width=.5mm] (1,0)  -- (-1,0); 
\end{tikzpicture}
\end{equation}
where the horizontal green and red dotted lines are the shortened geodesics. However, as above we can alter the vertical blue geodesic to pass through the boundary faces of $r$ in $w$ which correspond to the lower horizontal green strand. Replacing $\ga_1$ and $\ga_3$ we get that the length of the lower horizontal green cut path and the blue cut paths have to be the same, and arguing that the higher red cut path can made to follow the blue cut path outside of this region, we get that it connects the same two boundary faces of the relator, but has longer length, than the original.

In the remaining cases, we don't get the ability to cross every boundary edge:

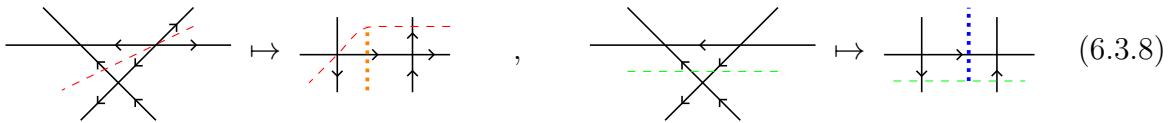
\begin{equation}
\begin{tikzpicture}[baseline=-1ex,scale=1,transform shape, rotate=180] 
\draw[midfrom,web] (0,0) -- (1,0); 
\draw[midto,web] (1,0) -- (2,0);
\draw[web] (2,0) -- (3,0);

\draw[midfrom, web] (0.5,-.5) -- (1,0); 
\draw[midto,web] (1,0) -- (1.5,.5);
\draw[midto,web] (1.5,.5) -- (2,1);

\draw[web] (2.5,-.5) -- (2,0); 
\draw[midfrom,web] (2,0) -- (1.5,.5);
\draw[midfrom,web] (1.5,.5) -- (1,1);
\draw[red, dashed] (0.5,-0.25) -- (2.25,.6); 
\end{tikzpicture}\; \mapsto \;
\begin{tikzpicture}[baseline=-.3ex,scale=.5] 
\draw[midfrom,web] (0,-1) -- (0,0);
\draw[web] (0,0) -- (-1,0);
\draw[web] (0,0) -- (0,1);
\draw[midto,web] (0,0) -- (2,0);
\draw[midto,web] (2,0) -- (3,0);
\draw[midto,web] (2,0) -- (2,1);
\draw[midfrom,web] (2,0) -- (2,-1);
\draw[red, dashed] (3,.75) -- (1,.75) to[out=-180,in=45] (0,0) --(-.75,-.75);
\draw[orange, dotted,line width=.5mm] (.8,-.95) -- (.8,.75); 
\end{tikzpicture} \qquad , \qquad
\begin{tikzpicture}[baseline=-1ex,scale=1,transform shape, rotate=180] 
\draw[web] (0,0) -- (1,0); 
\draw[midto,web] (1,0) -- (2,0);
\draw[web] (2,0) -- (3,0);

\draw[web] (0.5,-.5) -- (1,0); 
\draw[midto,web] (1,0) -- (1.5,.5);
\draw[midto,web] (1.5,.5) -- (2,1);

\draw[web] (2.5,-.5) -- (2,0); 
\draw[midfrom,web] (2,0) -- (1.5,.5);
\draw[midfrom,web] (1.5,.5) -- (1,1);
\draw[green, dashed] (0.5,.35) -- (2.5,.35); 
\end{tikzpicture}\; \mapsto \;
\begin{tikzpicture}[baseline=-.3ex,scale=.5] 
\draw[midfrom,web] (0,-1) -- (0,0);
\draw[web] (0,0) -- (-1,0);
\draw[web] (0,0) -- (0,1);
\draw[midto,web] (0,0) -- (2,0);
\draw[web] (2,0) -- (3,0);
\draw[web] (2,0) -- (2,1);
\draw[midfrom,web] (2,0) -- (2,-1);
\draw[green, dashed] (-.75,-.7) -- (2.75,-.7); 
\draw[blue, dotted,line width=.5mm] (1.25,-.7) -- (1.25,1.25);
\end{tikzpicture}
\end{equation}
Where orange and blue dotted lines represent $\ga_2$. In the first case, if the shortened cut path follows $\rho$ from the left we see that it can be shortened. If it comes from the right, we can replace it with an equal length segment going down and right avoiding the interior edge, and hence corresponding to a path in $w$. Then, in the second (i.e. right) case, we see that we can shorten the cut path by crossing one of the tetravalent vertex no matter which side we come from.

Finally the horizontal relator is actually simpler. In this case, the distances between boundary faces are unchanged in the two leading order terms, so the lengths of maximal geodesics are automatically unchanged.

So what we've proven is that for every geodesic cut path in $w$, we can find a summand $w'$ and equal length geodesic in $w'$. In order to get a well defined summand, we need to make sure that the summand is the same for every geodesic cut path in $w$. For the bigon and horizontal relator case this follows from there only being one summand chosen. In the reducible triangle case, it's enough to show that if there is a geodesic crossing the cyclically oriented vertex (as depicted in blue below) then there can't be geodesic cut paths of the other two types originating from $\ast$.
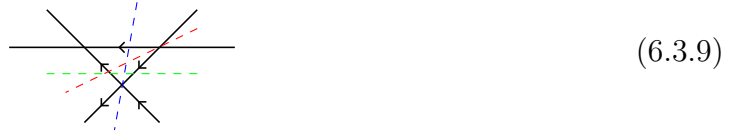
\begin{equation}
\begin{tikzpicture}[baseline=-1ex,scale=1,transform shape, rotate=180] 
\draw[web] (0,0) -- (1,0); 
\draw[midto,web] (1,0) -- (2,0);
\draw[web] (2,0) -- (3,0);

\draw[web] (0.5,-.5) -- (1,0); 
\draw[midto,web] (1,0) -- (1.5,.5);
\draw[midto,web] (1.5,.5) -- (2,1);

\draw[web] (2.5,-.5) -- (2,0); 
\draw[midfrom,web] (2,0) -- (1.5,.5);
\draw[midfrom,web] (1.5,.5) -- (1,1);
\draw[blue, dashed] (1.3,-.6) -- (1.6,1.1); 
\draw[green, dashed] (0.5,.35) -- (2.5,.35); 
\draw[red, dashed] (0.5,-0.25) -- (2.25,.6); 
\end{tikzpicture}
\end{equation}
but we see that all three types cross all 3 strands, so if the non-cyclic type (ie red or green) were to start where the cyclic type starts, then they would need to cross one of the 3 strands twice, and hence wouldn't be a geodesic. 

Next we want to prove that every face $f$ in $w'$ is still crossed by a geodesic cut path. In the case of reduction relators, the fact that any maximal geodesic cut path can be made to avoid the relator face means that we can find a cut path $p$ in $w'$ which corresponds to a geodesic in $w$, but by the above argument, any geodesic in $w'$ can be made to correspond to a geodesic in $w$, so $p$ is minimum length since the corresponding path in $w$ is minimum length.

In the horizontal relator case, any maximal geodesic cut path in $w'$ be associated to a geodesic in $w$ with the same length crossing all the same faces outside the relator face. So the only thing left to check is that there is a maximal geodesic cut path crossing the relator face. If there is a maximal geodesic cut path in $w$ connects opposite boundary faces of the relator, then this immediate, but notice that the red dashed cut path in the corresponding diagram below isn't associated to a cut path in $w'$ which crosses the face:
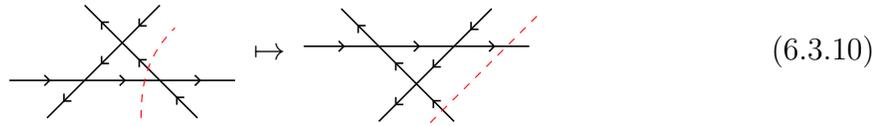
\begin{equation}
\begin{tikzpicture}[baseline=1.5ex,scale=1] 
\draw[midto,web] (0,0) -- (1,0); 
\draw[midto,web] (1,0) -- (2,0);
\draw[midto,web] (2,0) -- (3,0);

\draw[midfrom,web] (0.5,-.5) -- (1,0); 
\draw[midfrom,web] (1,0) -- (1.5,.5);
\draw[midfrom,web] (1.5,.5) -- (2,1);

\draw[midto,web] (2.5,-.5) -- (2,0); 
\draw[midto,web] (2,0) -- (1.5,.5);
\draw[midto,web] (1.5,.5) -- (1,1);

\draw[red,dashed] (1.75,-.5) to[out=90,in=225] (2.2,.7);
\end{tikzpicture}
 \; \mapsto \;
\begin{tikzpicture}[baseline=-1ex,scale=1,transform shape, rotate=180] 
\draw[midfrom,web] (0,0) -- (1,0); 
\draw[midfrom,web] (1,0) -- (2,0);
\draw[midfrom,web] (2,0) -- (3,0);

\draw[midto,web] (0.5,-.5) -- (1,0); 
\draw[midto,web] (1,0) -- (1.5,.5);
\draw[midto,web] (1.5,.5) -- (2,1);

\draw[midfrom,web] (2.5,-.5) -- (2,0); 
\draw[midfrom,web] (2,0) -- (1.5,.5);
\draw[midfrom,web] (1.5,.5) -- (1,1);

\draw[red,dashed] (-.1,-.4) -- (1.4,1.1);
\end{tikzpicture}
\end{equation} 
But we know $w$ is coherent, so by combinatorial convexity, there is a geodesic which crosses the opposite (ie top) face which is always adjacent to $\rho$, so if it doesn't cross the face it must be:
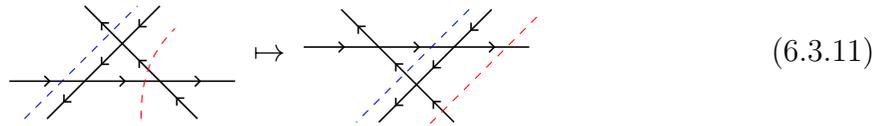
\begin{equation}
\begin{tikzpicture}[baseline=1.5ex,scale=1] 
\draw[midto,web] (0,0) -- (1,0); 
\draw[midto,web] (1,0) -- (2,0);
\draw[midto,web] (2,0) -- (3,0);

\draw[midfrom,web] (0.5,-.5) -- (1,0); 
\draw[midfrom,web] (1,0) -- (1.5,.5);
\draw[midfrom,web] (1.5,.5) -- (2,1);

\draw[midto,web] (2.5,-.5) -- (2,0); 
\draw[midto,web] (2,0) -- (1.5,.5);
\draw[midto,web] (1.5,.5) -- (1,1);

\draw[blue, dashed] (.2,-.5) -- (1.7,1); 
\draw[red,dashed] (1.75,-.5) to[out=90,in=225] (2.2,.7);
\end{tikzpicture}
 \; \mapsto \;
\begin{tikzpicture}[baseline=-1ex,scale=1,transform shape, rotate=180] 
\draw[midfrom,web] (0,0) -- (1,0); 
\draw[midfrom,web] (1,0) -- (2,0);
\draw[midfrom,web] (2,0) -- (3,0);

\draw[midto,web] (0.5,-.5) -- (1,0); 
\draw[midto,web] (1,0) -- (1.5,.5);
\draw[midto,web] (1.5,.5) -- (2,1);

\draw[midfrom,web] (2.5,-.5) -- (2,0); 
\draw[midfrom,web] (2,0) -- (1.5,.5);
\draw[midfrom,web] (1.5,.5) -- (1,1);

\draw[blue, dashed] (.8,-.5) -- (2.3,1); 
\draw[red,dashed] (-.1,-.4) -- (1.4,1.1);
\end{tikzpicture}
\end{equation}
but in $w'$ we can then find an equal length geodesic which crosses the relator face as depicted.

Therefore we've shown that we can find a well-defined summand where all maximal geodesic cut path distances from $\ast$ are the same, and the summand is still coherent.
\end{proof}

We can finally prove theorem \cref{intro:Satake}, which we restate here:

\Satake*
\begin{proof}
We will first prove that there exist reduced, coherent webs. We will use the existence of a coherent web for each LS path from \cite{fontaine:generating}. Let $w_1$ be such a coherent web. If it is reduced, we're done. Otherwise, there exists a sequence of hexagon relators which leads the web to be reducible. By \cref{Rel:summand}, the leading order term is still coherent. By again by \cref{Rel:summand} we can find a coherent summand web of the reduction relator which has the same cut path lengths between $\ast$ and the boundary vertices. We can then repeat this process until it is reduced. This gives us existence.

Since such a set of coherent webs are a basis of $\gr Sp$, and we have $E^1$ convergence, for every reduced web $w$ there is a sequence of horizontal relators which makes $w$ into one of these coherent webs which we'll call $w'$. Since the horizontal square and triangle relators commute on reduced webs, we can have the square relators occur before the triangle relators in the sequence. Call the intermediate web before applying any hexagon relators, $w''$. Since $w''$ and $w$ are obtained from each other via square relators which have no lower order terms, they are equal on the nose in $Sp$. Moreover, $w''$ is coherent since it is equivalent to $w'$ via hexagon relators. Therefore, if we decompose $w=w''$ in the Satake basis, it is upper unitriangular. So any basis of reduced webs in $\gr Sp$ correspond to a upper unitriangular basis. Since this is true for any base point, we get our result. 
\end{proof}

\section{Combinatorial Minimal Disks in the \texorpdfstring{$A_3$}{} Euclidean Building}
Geometrically, it is better to look at the so-called dual web. Take a non-collapsed web $w$, and then take the dual graph with directed edges labeled by weights. We can then assign a distance to every cellular path with values in the dominant weight cone as defined in \cite{fkk:buildings}. Denote this complex by $D(w)$. Then in any simplicial complex with weight valued metric, a \tbf{combinatorial geodesic path} is a cellular path whose length is minimal with respect to the natural ordering of dominant weights. These exactly correspond to the geodesic cut paths from before.

Following \cite{fkk:buildings}, take a polygon with fixed base point whose sides are labeled by dominant weights, and look at the configuration space of locally isometric maps from the $0$-skeleton $D(w)^0$ into the Euclidean building (over $\C$). For each LS path, there is a generic embedding such that the distances from the base point to the other vertices on the immersed polygon are exactly the weights in this path. By \cite{fontaine:generating}, for any coherent web $w$ corresponding to this LS path, we can extend the map of the polygon to a map of the coherent web. Conversely, if we have an immersed subcomplex of the building which is a disk whose boundary is the immersed polygon, we get a dual web with boundary given by the weights of the polygon. Hence, we can interpret coherent webs as surfaces in the building with a fixed boundary. We say that such a disk is a combinatorial minimal disk if the area is minimal among all such disks. We're interested in the geometric properties of these disks.

First, we reinterpret the relators geometrically:
\begin{prop}
\label{geo:horz}
Let $D$ be a topological disk which is an immersed subcomplex of the $A_3$ Euclidean building with generic boundary polygon $P$, and let the corresponding web $w$ have a horizontal square (\ref{SS}) or $I=H$ (\ref{IH}) relator $r$, then there is another such topological disk $D'$ with the same boundary whose corresponding web is $r(w)$
\end{prop}
\begin{remark}
Unlike the $I=H$ relator and horizontal square relator, the boundary of the hexagon relator can be embedded in many different ways, which makes it much less easy to work with. Luckily we can use are previous combinatorial results to get what we want without needing these hexagon relators.
\end{remark}
\begin{proof}
First let $r$ be a square relator, and $v$ be the vertex in $D$ corresponding to the relator face. The combinatorial link of $v$ in the building is an $A_3$ spherical building, whose vertices can be associated to subspaces of $\C^4$. The adjacent vertices in $D$ will correspond to two $1$-D subspaces, and two $3$-D subspaces, where each $3$-D subspaces contains both of the $1$-D subspaces. There exists at least one $2$-D subspace contained in the two $3$-D subspaces, and contained the two $1$-D subspaces (if all subspaces are distinct then there is actually a unique such subspace). The corresponding vector in the building is adjacent to all four vertices, giving us a square composed of four triangles. Replacing the original square relator face with this face gives us our disk $D'$.

Similarly, if $r$ is an $I=H$ relator, look at the link of one over the vertices adjacent to the double edge. We have a $1$-D, $2$-D, and $3$-D subspace such that the $1$-D is contained in the $2$-D, and the $2$-D is contained in the $3$-D, and thus the $1$-D is contained in the $3$-D, which tells us that there is an edge in the building connecting the two vertices which are unconnected in $D$. Replacing the original double edge with this double edge, we get $D'$.
\end{proof}

\begin{prop}
\label{geo:red}
Let $D$ be a topological disk which is an immersed subcomplex of the $A_3$ Euclidean building with boundary polygon $P$, and let the corresponding web $w$ have a reduction relator $r$, then there is another such topological disk $D'$ with lower area but the same boundary polygon.
\end{prop}
\begin{proof}
This follows mainly from direct computation in the link of each vertex which is a spherical building:
\begin{enumerate}
	\item If $r$ is a loop, then we can just omit the corresponding edge.
	\item If $r$ is the bigon we see that the two adjacent faces are separated by two edges of the same type. This implies that the corresponding edges are the same in the building, and hence if we omit the vertex corresponding to $r$ we still have a disk
	\item If $r$ is the square relator with two interior double edges, we look at the link of the corresponding vertex. We have two $2$-D subspaces and two $1$-D subspaces (up to duality). If both $1$-D subspaces are distinct, then the $2$-D subspaces must be the same which means the complex gets folded and we can omit the vertex corresponding to $r$. Similarly, if both $2$-D subspaces are different, their intersection will have to be $1$-D, and so the two $1$-D subspaces are the same. In that case the complex gets folded the other way.  If both $1$-D subspaces are the same and both $2$-D subspaces then we fold both ways and our relator corresponds to a single triangle which is connected the $D$ by one edge. Replacing the triangle with the edge in $D$ we get our result.
	\item If $r$ is the square relator with one interior double edge then we look at the link of the correspond vertex, and we have two $1$-D subspaces, one $2$-D subspace, and one $3$-D subspace, where the two larger subspaces contain both of the $1$-D subspaces. If the two $1$-D subspaces are different, then they span the $2$-D subspace, which implies the $3$-D subspaces contains the $2$-D subspaces, so we can get an smaller complex with two triangles. If the two $1$-D are the same, then the subcomplex folds, and we can remove the relator vertex and its four adjacent triangles.
\end{enumerate}
\end{proof}

This shows us that minimal area combinatorial disks correspond to reduced webs

\begin{prop}
\label{geo:reduced=minimal}
Let $D$ be a topological disk which is an immersed subcomplex of the $A_3$ Euclidean building with generic boundary polygon $P$. Then $D$ has minimal area among such disks if and only if the associated web $w$ is reduced.
\end{prop}
\begin{proof}
We first assume that $D$ is minimal but corresponding to a reducible $w$. We know that the there is a sequence of square and $I=H$ relators on $w$ that lead to a reducible face by \cref{A3 Purity}. By \cref{geo:horz}, this web also corresponds to a minimal area disk with the same boundary polygon, but then by \cref{geo:red} there is a smaller disk with the same boundary polygon, contradicting the statement that $D$ is minimal.

Now assume that $w$ is reduced. Then there is a sequence of relators which makes $w$ coherent at some boundary vertex $\ast$, and this correspondingly gives a disk $D'$ of the same area by \cref{geo:red}. Now, given a minimal combinatorial disk, we know that its corresponding web is also reduced by the above, and hence we can take it to be coherent at $\ast$ after applying relators. Then $D'$ and this disk are coherent with the same boundary distances, and hence must be equal up to lower order terms by \cref{intro:Satake}. Thus by $E^1$ convergence, there is a sequence of horizontal relators connecting their webs, and thus they are equal area, so $D(w)$ has minimal area. 
\end{proof}

The following proposition will give an interpretation of strands as geodesic strips, and will allow us to extend results about coherent webs to minimal webs. Let $w$ be a reduced web, and let $p$ be a (possibly undirected) maximal strand in the contracted web $S(w)$. We can find a cut path $\rho_p$ in an $\ep$ neighborhood of the strand on the left or the right, and this corresponds to a unique cut path in $w$ because it doesn't cross any tetravalent vertices. We'll call such a path a \tbf{strand-following geodesic}.
\begin{prop}
\label{geo:Strandgeo}
If $w$ is the reduced web, then there is a sequence of horizontal relators $r_i$ such that $w'=r_n...r_1(w)$ is coherent at the starting face $\rho_p(0)$ of a strand-following geodesic $\rho_p$, each $r_i$ is disjoint from $\rho_p$, and $\rho_p$ is a geodesic on $w'$.
\end{prop}
\begin{proof}
Since the strand doesn't cross any other strand more than once by \cref{Global}, we know that $\rho_p$ crosses every strand at most once. It is impossible for any other strand with the same end points to cross a double edge but not cross a strand twice in $S(w)$, indeed in $S(w)$ it would be crossing the intersection of two strands that must cross $p$, and the orientation could not be cyclic, and hence the triangle these strands make with $p$ would by non-cyclically oriented, and hence $w$ would be reducible by \cref{Global} as depicted below (up to dualizing): 
\begin{equation}
\begin{tikzpicture}[baseline=0ex,scale=1,] 
\draw[web] (0,0) -- (0,2)node[above, fill=white] {p}; 
\draw[dashed,blue] (0.3,0) -- (0.3,2)node[above, fill=white] {$\rho_p$};
\draw[web,midto] (0,.5) -- (.5,1); 
\draw[web] (-.5,0) -- (1.5,2);
\draw[web,midto] (0,1.5) -- (.5,1); 
\draw[web] (1.5,0) -- (-.5,2);
\end{tikzpicture}
\end{equation}
Now cut the web along $\rho_p$ to obtain two subwebs $w_l$ and $w_r$. By the proof of \cref{intro:Satake} we know that we can find a sequence of horizontal square relators and $I=H$ relators on $w_l$ and $w_r$ making each coherent at $\rho_p(0)$. These relators correspond to relators and $w$ giving a new web $w'$, and subwebs $w_l'$ and $w_r'$, and $\rho_p(0)$ is unchanged. We will prove that $w'$ is coherent at $\rho_{p}(0)$. Since $\rho_{p}$ is still following a strand, by the previous argument we know that there can't be cut paths with shorter length, and hence it is a geodesic in both $w_l'$ and $w_r'$. Any segment of a geodesic which crosses $\rho_p$ twice can be replaced by a geodesic which follows $\rho_p$ instead of crossing it twice. In particular, for a geodesic from $\rho_p(0)$ to any point there is a cut path of shorter or equal length that stays in $w_l'$ or $w_r'$, and hence is a geodesics in that subweb. Hence geodesic distances are coherent. Moreover, this also shows that geodesics in $w_l'$ and $w_r'$ are geodesics in $w'$. We want to show that every face is along some geodesic from $\rho_p(0)$. Given a face $f$ which is in $w_r'$ or $w_l'$, there is a geodesic in $w_r'$ or $w_l'$ respectively from $\rho_p(0)$ which passes through it. If the ending face is on the boundary of $w'$, then this is a geodesic in $w'$. If it ends on $\rho_p$, then we can extend it via $\rho_p$ to a maximal geodesic in $w'$, so in either case $f$ is crossed by a maximal geodesic.
\end{proof}

Now fix an $LS$ path $\vec{\la}$, and construct the corresponding fan variety configuration space as in \cite{fkk:buildings}. This is an open dense subset of an irreducible component and hence irreducible. Look at the subvarieties obtained by enforcing particular (dominant-weight valued) distances between vertices in the polygon. Since these distances are bounded by the circumference, the fan variety is the union of finitely many disjoint subvarieties corresponding to each choice of distances. By irreducibility, there is a unique open dense subvariety, which we'll call $X_{\vec{\la}}$, which is thus also a dense subset of the corresponding irreducible component of the Satake fiber.

\begin{prop}
\label{geo:config}
If $w$ is a reduced web which is equal up to lower order terms to a coherent reduced web with boundary distances $\vec{\la}$, then every configuration of $X_{\vec{\la}}$ extends uniquely to a configuration of $D(w)$
\end{prop}
\begin{proof}
Given any configuration in $f \in X_{\vec{\la}}$. 

We first prove uniqueness. Given one of the strand-following geodesics $\rho_p$ in $D(w)$, by proposition \ref{geo:Strandgeo}, the weight sequence of $\rho_p$ is the same as for a coherent web. Now the configuration of a coherent web corresponds to the fan variety and so gives an open dense subset of the Satake component. This tells us that the weight sequence of $\rho_p$ is the same as the distance between the corresponding vertices in the building, and hence if $f$ extends to $\tilde{f}$ from $D(w)$ to the Euclidean building, then $\rho_p$ must be a geodesic which is uniquely determined by this weight sequence. Since every vertex is crossed by one of these strand-following geodesics, the image of every point is determined. Since every vertex has uniquely determined image, and every higher simplex is determined uniquely by its vertices, this implies uniqueness.

For existence, we know that $w$ can be made coherent only using horizontal square and $I=H$ relators. Call this coherent web $w'$. We know by \cref{intro:Satake} that there is a unique extension of $f$ to $D(w')$. Finally, using \cref{geo:horz}, we know that we can apply horizontal relators to get a disk with the same boundary and whose web is $w$.
\end{proof}
	
In particular we get the following;
 
\begin{prop}
If $D$ is a minimal combinatorial disk in the Euclidean building with boundary a generic polygon $P \in X_{\vec{\la}}$, then every vertex in $D$ is contained in a geodesic between vertices in $P$.
\end{prop}
\begin{proof}
By \cref{geo:reduced=minimal}, the corresponding web $w$ is reduced. Then by \cref{geo:Strandgeo}, each vertex of $D(w)$ is contained in a path following a strand, and these correspond to geodesics in the image, which gives the result.  
\end{proof}

This implies that the strand-following geodesics are geodesics in a stronger sense:
\begin{coro}
If $\rho_p$ is a geodesic following a strand, then every path between vertices on $\rho_p$ has less than or equal length (so in particular have comparable lengths).
\end{coro}
\begin{proof}
Each such path gives an equal length path in the Euclidean building which we know has coherent geodesics.
\end{proof}

Finally we can show that disks whose corresponding webs differ by a relator are the same except at the relator face:
\begin{prop}
If $D$ and $D'$ are minimal combinatorial disks of the Euclidean building with boundary a generic polygon $P \in X_{\vec{\la}}$, and $r$ is a horizontal relator whose leading order term is $w-w'$, then if $v$ is a vertex in which doesn't not correspond to the relator face, then the image of $v$ and $r(v)$ are the same in the Euclidean building
\end{prop}
\begin{proof}
We know the corresponding webs $w$ and $w'$ are reduced by \cref{geo:reduced=minimal}. Let $v$ be such a vertex not corresponding to the relator face, then there is a maximal strand-following geodesic which passes through it. Since the geodesic can't cross the relator face twice, either the geodesic segment before or after $v$ must not cross the relator face. Then this segment has the same weight-sequence in $w$ and $r(w)$. Since both geodesic segments are part of geodesics connecting the same boundary vertices, both segments are inside a common apartment, and since their weight-sequences are the same they must be equal.
\end{proof}

Therefore, relators are geometrically local.

\chapter{Examples and Counterexamples}
\linespread{1}\selectfont
In this chapter we will provide counterexamples to some potentially desirable properties. We will make use of product spiders to do this. The sections here are all independent.

\section{Failure of \texorpdfstring{$E^1$}{} Convergence in \texorpdfstring{$A_1^2 \times A_2$}{} and \texorpdfstring{$A_2 \times A_2$}{}}
In this section we will show that the spiders of type $A_1^2 \times A_2$ and $A_2 \times A_2$ fail to have $E^1$ convergence, and hence product spiders of the form $A_1^n \times A_2^m$ fail to have $E^1$ convergence unless $m =0$ or the rank is less than or equal to $3$. Recall that when drawing product webs we will not draw the orientation of every edge as long as the orientation of the corresponding strand is indicated (in order to avoid clutter).

First, for $A_1^2 \times A_2$ take the following sequence of webs:
\begin{equation}
\label{Ex:A12A2}
\begin{tikzpicture}[baseline=-.3ex,scale=.6]
\draw[midto,web] (-1,-1) -- (0,0);
\draw[midto,web] (-1,1) -- (0,0);
\draw[midfrom,web] (0,0) -- (2,0);
\draw[midto,web]  (2,0) -- (3,1);
\draw[midto,web] (2,0) -- (3,-1);
\draw[red] (.9,-1) -- (.9,1);
\draw[blue] (-1,0) to[out=30,in=150] (3,0);
\end{tikzpicture}\; , \quad
\begin{tikzpicture}[baseline=-.3ex,scale=.6]
\draw[midto,web] (-1,-1) -- (0,0);
\draw[midto,web] (-1,1) -- (0,0);
\draw[midfrom,web] (0,0) -- (2,0);
\draw[midto,web]  (2,0) -- (3,1);
\draw[midto,web] (2,0) -- (3,-1);
\draw[red] (.5,-1) -- (.5,1);
\draw[red] (1.5,-1) -- (1.5,1);
\draw[blue] (-1,0) to[out=30,in=150] (3,0);
\end{tikzpicture} \; , \quad
\begin{tikzpicture}[baseline=-.3ex,scale=.6]
\draw[midto,web] (-1,-1) -- (0,0);
\draw[midto,web] (-1,1) -- (0,0);
\draw[midfrom,web] (0,0) -- (2,0);
\draw[midto,web]  (2,0) -- (3,1);
\draw[midto,web] (2,0) -- (3,-1);
\draw[red] (.4,-1) -- (.4,1);
\draw[red] (.9,-1) -- (.9,1);
\draw[red] (1.6,-1) -- (1.6,1);
\draw[blue] (-1,0) to[out=30,in=150] (3,0);
\end{tikzpicture} \; ,\quad ...
\end{equation}
Now as elements in $Sp$ we know that these are equivalent to the following sequence:
\begin{equation}
\begin{tikzpicture}[baseline=-.3ex,scale=.6]
\draw[midto,web] (-1,-1) -- (0,0);
\draw[midto,web] (-1,1) -- (0,0);
\draw[midfrom,web] (0,0) -- (2,0);
\draw[midto,web]  (2,0) -- (3,1);
\draw[midto,web] (2,0) -- (3,-1);
\draw[red] (.9,-1) -- (.9,1);
\draw[blue] (-1,0) to[out=-30,in=210] (3,0);
\end{tikzpicture}\; , \quad
\begin{tikzpicture}[baseline=-.3ex,scale=.6]
\draw[midto,web] (-1,-1) -- (0,0);
\draw[midto,web] (-1,1) -- (0,0);
\draw[midfrom,web] (0,0) -- (2,0);
\draw[midto,web]  (2,0) -- (3,1);
\draw[midto,web] (2,0) -- (3,-1);
\draw[red] (.5,-1) -- (.5,1);
\draw[red] (1.5,-1) -- (1.5,1);
\draw[blue] (-1,0) to[out=-30,in=210] (3,0);
\end{tikzpicture} \; , \quad
\begin{tikzpicture}[baseline=-.3ex,scale=.6]
\draw[midto,web] (-1,-1) -- (0,0);
\draw[midto,web] (-1,1) -- (0,0);
\draw[midfrom,web] (0,0) -- (2,0);
\draw[midto,web]  (2,0) -- (3,1);
\draw[midto,web] (2,0) -- (3,-1);
\draw[red] (.4,-1) -- (.4,1);
\draw[red] (.9,-1) -- (.9,1);
\draw[red] (1.6,-1) -- (1.6,1);
\draw[blue] (-1,0) to[out=-30,in=210] (3,0);
\end{tikzpicture} \; ,\quad ...
\end{equation}
and yet there are no subgraphs corresponding to the leading terms of relators since all internal faces are squares with three colors. This violates $E^1$ convergence. Moreover, elements in the second sequence are the only minimal vertex webs which are equal to the webs in the first sequence in $Sp$ (which can be seen by projecting to pairs of factors), so in order for a presentation with the same generating morphisms to have the $E^1$ property, it must possess all of these as relators, and so it can't be finite.

Now one question to ask is why does this spider fails to have $E^1$ convergence. One answer coming from our criterion (theorem $\ref{CC}$) is that it fails the purity criterion, basically because when we create a reducible face with horizontal relators there may not be a choice of which relator to make reducible. Another answer comes from geometry. If we look at the interior face of:
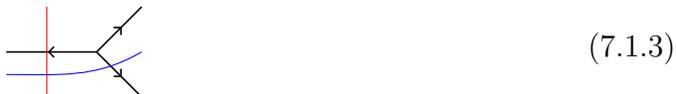
\begin{equation}
\begin{tikzpicture}[baseline=-.3ex,scale=.6]
\draw[midfrom,web] (0,0) -- (2,0);
\draw[midto,web]  (2,0) -- (3,1);
\draw[midto,web] (2,0) -- (3,-1);
\draw[red] (.9,-1) -- (.9,1);
\draw[blue] (0,-.5) to[out=0,in=210] (3,0);
\end{tikzpicture}
\end{equation}
the corresponding polygon in the dual web has a total interior angle of $330^\circ$, so it is positively curved, and yet there is no associated relator. Only after attaching another such face as in (\ref{Ex:A12A2}) does a (non-local) relator emerge.

For $A_2 \times A_2$ we place two regular $2n$-gons ($n \geq 3$) on top of each other where one is rotated by $\pi/2n$. For example for a hexagon we'd get something like this:

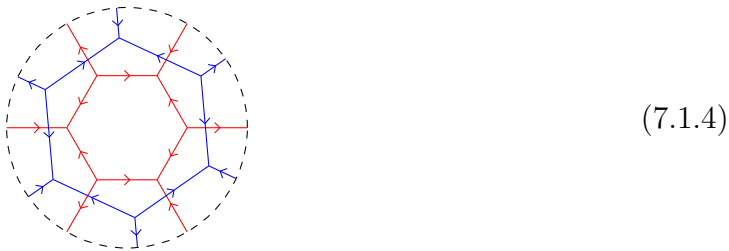
\begin{equation}
\begin{tikzpicture}[baseline=0,scale=.8] 
\draw[dashed] (0,0) circle (2cm); 
\draw[red,midto] (0:1) -- (60:1) ; 
\draw[red,midfrom] (60:1)-- (120:1);
\draw[red,midto] (120:1) --(180:1);
\draw[red,midfrom] (180:1) -- (240:1) ;
\draw[red,midto] (240:1)-- (300:1);
\draw[red,midfrom] (300:1) -- (0:1);

\draw[red,midto] (0:1) -- (0:2) ; 
\draw[red,midfrom] (60:1)-- (60:2) ;
\draw[red,midto] (120:1) --(120:2);
\draw[red,midfrom] (180:1) -- (180:2) ;
\draw[red,midto] (240:1)-- (240:2);
\draw[red,midfrom] (300:1) -- (300:2);

\draw[blue,midto] (35:1.5) -- (95:1.5) ; 
\draw[blue,midfrom] (95:1.5)-- (155:1.5);
\draw[blue,midto] (155:1.5) --(215:1.5);
\draw[blue,midfrom] (215:1.5) -- (275:1.5) ;
\draw[blue,midto] (275:1.5)-- (335:1.5);
\draw[blue,midfrom] (335:1.5) -- (35:1.5);

\draw[blue,midto] (35:1.5) -- (35:2); 
\draw[blue,midfrom] (95:1.5)-- (95:2);
\draw[blue,midto] (155:1.5) --(155:2);
\draw[blue,midfrom] (215:1.5) -- (215:2) ;
\draw[blue,midto] (275:1.5)-- (275:2);
\draw[blue,midfrom] (335:1.5) -- (335:2);
\end{tikzpicture}
\end{equation}

All of the interior faces are pentagons or larger, and hence there are no local product relators. However, there is another web which is equal in the spider $Sp$ where we switch which $2n$-gon is interior vs exterior:
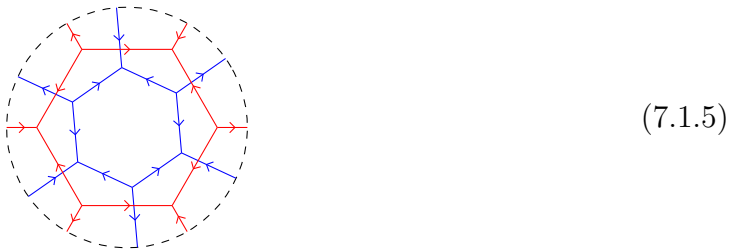
\begin{equation}
\begin{tikzpicture}[baseline=0,scale=.8] 
\begin{scope}[rotate=30]
\draw[dashed] (0,0) circle (2cm); 
\draw[blue,midto] (5:1) -- (65:1) ; 
\draw[blue,midfrom] (65:1)-- (125:1);
\draw[blue,midto] (125:1) --(185:1);
\draw[blue,midfrom] (185:1) -- (245:1) ;
\draw[blue,midto] (245:1)-- (305:1);
\draw[blue,midfrom] (305:1) -- (5:1);

\draw[blue,midto] (5:1) -- (5:2) ; 
\draw[blue,midfrom] (65:1)-- (65:2) ;
\draw[blue,midto] (125:1) --(125:2);
\draw[blue,midfrom] (185:1) -- (185:2) ;
\draw[blue,midto] (245:1)-- (245:2);
\draw[blue,midfrom] (305:1) -- (305:2);

\draw[red,midfrom] (30:1.5) -- (90:1.5) ; 
\draw[red,midto] (90:1.5)-- (150:1.5);
\draw[red,midfrom] (150:1.5) --(210:1.5);
\draw[red,midto] (210:1.5) -- (270:1.5) ;
\draw[red,midfrom] (270:1.5)-- (330:1.5);
\draw[red,midto] (330:1.5) -- (30:1.5);

\draw[red,midfrom] (30:1.5) -- (30:2); 
\draw[red,midto] (90:1.5)-- (90:2);
\draw[red,midfrom] (150:1.5) --(150:2);
\draw[red,midto] (210:1.5) -- (210:2) ;
\draw[red,midfrom] (270:1.5)-- (270:2);
\draw[red,midto] (330:1.5) -- (330:2);
\end{scope}
\end{tikzpicture}
\end{equation}

Just as in $A_1^2 \times A_2$, there are no other minimal vertex webs which are equal to these pairs in the spider, and so in order to give a presentation with the same generating morphisms, we would need infinitely many relators.

Again we can attempt to explain this. Looking at the criterion (theorem \ref{CC}) the spider fails controlled degeneration. However, geometrically this example is a bit more complicated than the case of $A_2 \times A_1^2$. In fact, in the hexagon case both of the corresponding dual webs are intrinsically flat. Embedding both of these into the building, the link of the interior vertex would be the disjoint union of two circles (one for each of the minimal webs above). So in some sense this is a dimension $>3$ problem, because we have two different flat surfaces with the same boundary.

\section{Non-Trivial Topology}
A possibly desirable quality for spiders would be for them to admit a height function on webs which has a unique highest web and lowest web. Unfortunately, there are examples where there are three extremal webs (where we define an extremal web be one where all local relators are non-adjacent), and so there is no way that such an ideal height function would exist, such as this $A_1^3$ web:
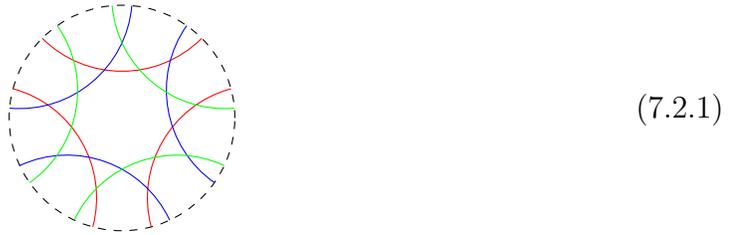
\begin{equation}
\begin{tikzpicture}[baseline=0ex,scale=1.5] 
\draw[dashed] (0,0) circle (1cm);
\draw[red] (.707,.707) to[out=225, in=-45] (-.707,.707);
\draw[red] (-.966,.259) to[out=-15, in=75] (-.259,-.966);
\draw[red] (.966,.259) to[out=195, in=105] (.259,-.966);
\begin{scope}[rotate=40]
\draw[blue] (.707,.707) to[out=225, in=-45] (-.707,.707);
\draw[blue] (-.966,.259) to[out=-15, in=75] (-.259,-.966);
\draw[blue] (.966,.259) to[out=195, in=105] (.259,-.966);
\end{scope}
\begin{scope}[rotate=80]
\draw[green] (.707,.707) to[out=225, in=-45] (-.707,.707);
\draw[green] (-.966,.259) to[out=-15, in=75] (-.259,-.966);
\draw[green] (.966,.259) to[out=195, in=105] (.259,-.966);
\end{scope}
\end{tikzpicture}
\end{equation}
where the three extremal webs correspond to removing all strands of one color from the boundary of the internal nonagon, so for example for red we get:
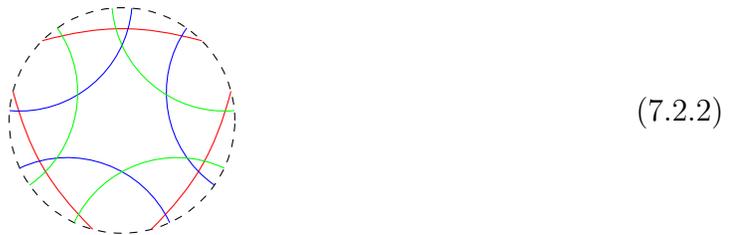
\begin{equation}
\begin{tikzpicture}[baseline=0ex,scale=1.5] 
\draw[dashed] (0,0) circle (1cm);
\draw[red] (.707,.707) to[out=165, in=15] (-.707,.707);
\begin{scope}[rotate=120]
\draw[red] (.707,.707) to[out=165, in=15] (-.707,.707);
\end{scope}
\begin{scope}[rotate=240]
\draw[red] (.707,.707) to[out=165, in=15] (-.707,.707);
\end{scope}
\begin{scope}[rotate=40]
\draw[blue] (.707,.707) to[out=225, in=-45] (-.707,.707);
\draw[blue] (-.966,.259) to[out=-15, in=75] (-.259,-.966);
\draw[blue] (.966,.259) to[out=195, in=105] (.259,-.966);
\end{scope}
\begin{scope}[rotate=80]
\draw[green] (.707,.707) to[out=225, in=-45] (-.707,.707);
\draw[green] (-.966,.259) to[out=-15, in=75] (-.259,-.966);
\draw[green] (.966,.259) to[out=195, in=105] (.259,-.966);
\end{scope}
\end{tikzpicture}
\end{equation}
If we look at the corresponding pocket, the point corresponding to the nonagon center vertex does not have an orientable neighborhood. In the link, each of the 9 adjacent relators correspond to a triangular strip and each vertex gives a half-twist, hence overall showing that the link is topologically a M\"obius band.

However, note that this example can't be extended to $A_3$ because it's impossible to orient a nonagon of strands such that the orientations of each edge of the nonagon alternate. Hence there would need to be a reducible triangle face.

\section{An Embeddable but non-Minimal Web}
One difficulty with the $A_3$ spider is that the minimal vertex spiders are not the only ones that (generically) embed in the Euclidean building, and similarly they aren't the only coherent spiders. The smallest example is as follows:

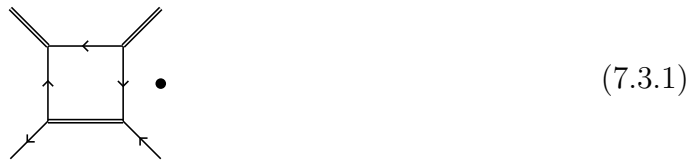
\begin{equation}
\begin{tikzpicture}[baseline=-.3ex,scale=.5] 
\draw[double,web] (-1,-1) -- (1,-1);
\draw[midto,web] (-1,-1) -- (-1,1);
\draw[midfrom,web] (-1,1)-- (1,1);
\draw[midfrom,web] (1,-1)-- (1,1);
\draw[midto,web] (-1,-1) -- (-2,-2);
\draw[double,web] (-1,1) -- (-2,2);
\draw[midfrom,web] (1,-1) -- (2,-2);
\draw[double,web] (1,1) -- (2,2);
\fill (0:2) circle (4pt);
\end{tikzpicture}
\end{equation}
The above reducible web is coherent with respect to the indicated basepoint, and it can be directly checked that the distances from the indicated point to every other face are distinct. Hence the corresponding generic map must be an embedding. We can also see that the image of any combinatorial map from the dual into the building is contained in a single apartment as follows. If we take the triangles corresponding to opposite trivalent vertices we know these must be contained in a single apartment, but every vertex in the dual web is in the union of vertices. So all the vertices, and corresponding simplices must also be in that apartment. 

\section{Independence of Axioms 1 and 2 of Coherent Webs}
Since we showed that the third axiom of coherent webs was redundant, a natural question is whether one of the other two is as well. Unfortunately, the answer is no as we can see from the following examples in the $A_3$ spider. 

The first is simple. It's just the $I$ from the $I=H$ relator:
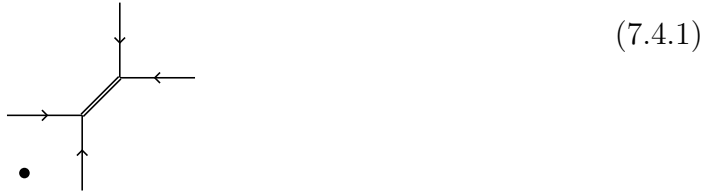
\begin{equation}
\begin{tikzpicture}[baseline=-.3ex,scale=.5] 
\draw[midfrom,web] (-1,-1) -- (1,-1);
\draw[midfrom,web] (-1,-1) -- (-1,1);
\draw[double,web] (-1,-1) -- (-2,-2);
\draw[midfrom,web] (-2,-2) -- (-2,-4);
\draw[midfrom,web] (-2,-2) -- (-4,-2);
\fill (225:5) circle (4pt);
\end{tikzpicture}
\end{equation}
The geodesics from the base point to the opposite face are length $2\la_1$ or $2\la_3$, which are incomparable weights. However, by applying the $I=H$ relator we see that the true length should be $\la_2$ which is smaller than both. The fact that the $I=H$ relator doesn't preserve coherence was one of the primary difficulties that we had to deal with.

The second is an example of a web which has coherent geodesics lengths from the base point, but there is a face which is not crossed by any geodesic:
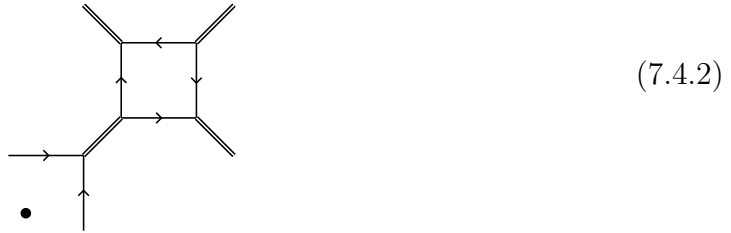
\begin{equation}
\begin{tikzpicture}[baseline=-.3ex,scale=.5] 
\draw[midto,web] (-1,-1) -- (1,-1);
\draw[midto,web] (-1,-1) -- (-1,1);
\draw[midfrom,web] (-1,1)-- (1,1);
\draw[midfrom,web] (1,-1)-- (1,1);
\draw[double,web] (-1,-1) -- (-2,-2);
\draw[double,web] (-1,1) -- (-2,2);
\draw[double,web] (1,-1) -- (2,-2);
\draw[double,web] (1,1) -- (2,2);

\draw[midfrom,web] (-2,-2) -- (-2,-4);
\draw[midfrom,web] (-2,-2) -- (-4,-2);
\fill (225:5) circle (4pt);
\end{tikzpicture}
\end{equation}
We see that the geodesics have coherent lengths, but no geodesic crosses the square face. However, if we apply the relator corresponding to this face, and then an $I=H$ relator we get a coherent web. In a sense, square relators can hide the incoherence caused by the $I=H$ relators.

From these two we see that we can't omit either of the axioms. However, seeing as the $I=H$ relator is contracted away in the contracted spider, it remains possible that in the contracted spider the first axiom is implied by the second.

\nocite{*}
\printbibliography[heading=bibintoc,title={References}]
\end{document}